%% file: Long_Wave_NS_arXiv.tex
\documentclass[11pt,reqno,a4paper]{amsart}
\usepackage[margin=1in]{geometry}
\usepackage{amsmath, amsthm, amssymb}
\usepackage{enumitem}
\usepackage{times}
\usepackage[dvipsnames]{xcolor}
\usepackage{subcaption}
\usepackage{mathtools}
\usepackage{bbm}
\usepackage{upgreek}
\usepackage{comment}
\usepackage{dsfont}
\usepackage{tikz}
\usepackage{cite}
\usepackage{cancel}
\usepackage{caption}

  \usepackage{hyperref}

 \usepackage[capitalise]{cleveref}
  \numberwithin{equation}{section}

 \usepackage{autonum}

\crefname{enumi}{item}{items}
\Crefname{enumi}{Item}{Items}

\crefformat{equation}{(#2#1#3)}
\crefrangeformat{equation}{(#3#1#4) to~(#5#2#6)}
\crefmultiformat{equation}{(#2#1#3)}%
{ and~(#2#1#3)}{, (#2#1#3)}{ and~(#2#1#3)}

\crefname{proposition}{Proposition}{Propositions}

\crefname{theorem}{Theorem}{Theorems}

\crefname{corollary}{Corollary}{Corollaries}

\crefname{lemma}{Lemma}{Lemmas}

\crefname{definition}{Definition}{Definitions}

\crefname{remark}{Remark}{Remarks}

\input{macros}

\definecolor{spectrumM}{rgb}{0.1,0.5,0.1} 
\definecolor{spectrumLpositive}{rgb}{0.8,0.2,0.2} 
\definecolor{spectrumLballs}{rgb}{0.2,0.2,0.8} 

\title[Long-wave instability of periodic shear flows]{Long-wave instability of periodic shear flows for the 2D Navier-Stokes equations}
\author{Maria Colombo}
\author{Michele Dolce}
\author{Riccardo Montalto}
\author{Paolo Ventura}
\address{Institute of Mathematics, EPFL, Station 8, 1015 Lausanne, Switzerland}
\email{maria.colombo@epfl.ch}
\email{michele.dolce@epfl.ch}
\email{paolo.ventura@epfl.ch}

\address{Dipartimento di Matematica ``Federigo Enriques'', Università degli Studi di Milano,
Via Cesare Saldini 50, 20133, Milano, Italy}
\email{riccardo.montalto@unimi.it}

\begin{document}
\begin{abstract}
In 1959, Kolmogorov proposed to study the instability of the shear flow $(\sin(y),0)$ in the vanishing viscosity regime in tori $\TT_{\alpha}\times \TT$. 
 This question was later resolved by Meshalkin and Sinai. 
We extend the problem to general shear flows $(U(y),0)$ and show that every $U(y)$ exhibits long-wave instability whenever $\|\partial_y^{-1} U\|_{L^2} > \nu$ and $\alpha\ll \nu$, with $\nu$ being the kinematic viscosity. This instability mechanism confirms previous findings by Yudovich in 1966, supported also by several numerical results, and is established through two independent approaches: one via the construction of Kato’s isomorphism and one via normal-forms. 
Unlike in many other applications of the latter methods, both proofs deal with the presence of a delicate term in the linearized operator that becomes singular as $\alpha$ approaches $0$.
\end{abstract}

\maketitle



\tableofcontents

\section{Introduction}
We consider the 2D forced Navier-Stokes equations in a rectangular torus
\begin{align}
\label{eq:NSintro}
\begin{cases}
\de_tw+v\cdot \nabla w=\nu\Delta w+f,\qquad (\tilde{x},y) \in \TT_\alpha\times \TT\\   v=\nabla^\perp\phi, \qquad \Delta\phi =w,
\end{cases}
 \end{align}
where $v$ is the velocity field, $w=\nabla^\perp\cdot v$ the vorticity, $\nu$ is the kinematic viscosity, which is proportional to the inverse Reynolds number,  $\TT_\alpha=\mathbb{R}/({\tfrac{2\pi}{\alpha}}\mathbb{Z})\simeq[0,\frac{2\pi}{\alpha})$ with $\alpha<1$ being the inverse aspect ratio of the torus and $\nabla^\perp=(-\de_y,\de_{\tilde{x}})$, $ \Delta=\de_{yy}+\de_{\tilde{x}\tilde{x}}$. The goal of this paper is understanding stability properties of shear flows
 \begin{align}
     u_E(\tilde x, y)=(U(y),0), \qquad \omega_E(\tilde x, y)=-U'(y), 
 \end{align}
 with  $U$ being sufficiently regular and $\int_{\TT} U(y)\dd y=0$. These
are steady states of \cref{eq:NSintro} with the force 
 \begin{equation}
     f(\tilde x, y)=\nu U'''(y).
 \end{equation}
 To state the main result, we first consider perturbations $\omega=\omega(t,x,y) $ of $\omega_E$ of the form 
 \begin{equation}
     w(t,\tilde{x},y)=\omega_E(y)+\omega(t,\alpha \tilde x  ,y)\, ,
 \end{equation} 
 so that $\omega$ is $2\pi$ periodic w.r.to the new variable $x$. Then we can write 
 \begin{align}
 \label{eq:omegaintro}
\begin{cases}
         \de_t\omega+\alpha U(y)\de_x\omega-\alpha U''(y)\de_x\psi=\nu \Delta_\alpha\omega-u\cdot\nabla_\alpha \omega, \qquad (x,y)\in \TT^2\\
     u=\nabla^\perp_\alpha \psi, \qquad \Delta_\alpha\psi=\omega,
     \end{cases}
 \end{align}
 where we denote $\nabla^\perp_\alpha=(-\de_y,\alpha\de_x)$ and $\Delta_\alpha=\de_{yy}+\alpha^2\de_{xx}.$ The main result of this paper, is the presence of unstable eigenvalues for the Orr-Sommerfeld operator 
 \begin{align}
     \label{def:Lalphanu}
    \boldsymbol{\cL}_{\nu,\alpha}\coloneqq \nu\Delta_\alpha-\alpha\de_x(U(y)-U''(y)\Delta_\alpha^{-1}),
 \end{align}
 where $    \boldsymbol{\cL}_{\nu,\alpha}:H^2(\TT^2)\to L^2(\TT^2)$. This type of instability arises in the  regime $\alpha\ll \nu$, as stated more precisely below.
 \begin{theorem}[Linear long-wave instability]
 \label{th:mainLin} Let $U\in H^3(\TT)$ be a shear flow profile with $0$ average. Then there exist  constants $\delta_0\coloneqq \delta_0(\|U\|_{H^3})\in (0,1)$ and $C\coloneqq C(\|U\|_{H^3}) \geq 1$ such that, for every $\nu, \alpha>0$ and $k\in \ZZ\setminus\{0\}$ with
  \begin{equation}
 \label{hyp:U}
    \|\de_y^{-1}U\|_{L^2}>\nu \quad\text{and}\quad \alpha|k|\leq \delta_0\nu\,,
 \end{equation} 
 there exists a unique (up to scalar multiplication), simple unstable eigenfunction of $    \boldsymbol{\cL}_{\nu,\alpha}$ in the class $\{ e^{ikx} f(y): f\in L^2(\TT)\}$. The remaining spectrum of the operator $  \boldsymbol{\cL}_{\nu,\alpha}$, in the same class,  is pure-point 
 and contained in $\{\Re \, z \leq -\nu/2\}$. The unstable eigenpair $\big(\boldsymbol{\lambda}_{\nu,\alpha k},e^{ikx}\boldsymbol{V}_{\nu,\alpha 
 k}(y)\big)$ fulfills the estimate
\begin{equation}\label{intunstableeigenvalue}
\Big|\boldsymbol{\lambda}_{\nu,\alpha k} - \frac{(\alpha k)^2}{\nu}\big( \|\de_y^{-1} U \|_{L^2}^2 - \nu^2\big) \Big| \leq C 
\frac{(\alpha |k|)^3}{\nu^2}  
,\qquad \|\boldsymbol{V}_{\nu,\alpha k}-U\|_{L^2}\leq C \frac{\alpha |k|}{\nu} \, . 
\end{equation}
 \end{theorem}

    \begin{figure}[h!]
    \centering
    \begin{subfigure}{0.46\textwidth}
        \centering
        \includegraphics[width=\textwidth]{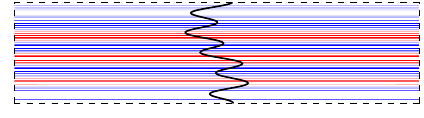}
        \caption{Shear flow $U(y)=\sin(y)+\cos(5y)$}
    \end{subfigure}
    \hspace{5mm}
    \begin{subfigure}{0.46\textwidth}
        \centering
        \includegraphics[width=\textwidth]{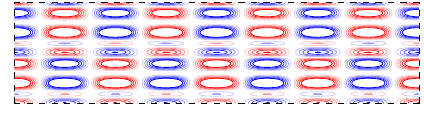}
        \caption{Approximate unstable eigenfunction}
    \end{subfigure}
        \caption{A sketch of the vorticity of the background laminar solution vs the unstable eigenfunction in a domain $\TT_{\alpha}\times \TT$ (with $\alpha=1/4$ for visualization purposes). In figure \textrm{(A)}, the black line is the shear velocity profile $U(y)=\sin(y)+\cos(5y)$ and the straight lines are the  level lines of its vorticity. In figure \textrm{(B)} are shown the level lines of  $\Re(e^{ix}\boldsymbol{V}_{\nu,\alpha k})\approx \cos(x)(\sin(y)+\cos(5y))$ (in which we neglect terms of order $\alpha/\nu$). }
\end{figure}

This result is, at least to our knowledge, the first to provide a rigorous justification of a long-wave instability mechanism for general shear flows in long rectangular tori $\TT_{\alpha}\times \TT$ with $\alpha\ll 1$. We note that the condition $\|\de_{y}^{-1}U\|_{L^2}>\nu$ is essentially the same found by Yudovich in \cite[eq (2.21)]{Yudovich66}\footnote{Unfortunately, this paper was never translated into English, and the result appears to have been somewhat forgotten.}, although it does not seem to have a rigorous mathematical justification\footnote{The analysis in \cite{Yudovich66} (see also the more recent related result \cite{revina2017stability}) appears to suffer from the same issues highlighted in \cite{lin2003instability} concerning Tollmien’s original work on shear flow instability. In particular, one has to \emph{assume} the existence of an unstable eigenpair and its analytic dependence on the parameter $\alpha$. This assumption allows for an asymptotic expansion of the eigenpair in $\alpha$, but the compatibility conditions required to close the problem at order $n$ only emerge at higher orders. Thus, even a finite-order truncation of the expansion requires a delicate analysis.}. Further discussion of related literature will be provided below.

We shall capture the long-wave instability for general shear flows with two different approaches: the first one is an adaptation of Kato's spectral theory, whereas the second one relies on normal-form methods. 
 These approaches prove Theorem \ref{th:mainLin}, but have some technical differences in the precision of the estimates and in the functional spaces considered. Two more precise statements, including more information on the asymptotic expansion of the unstable eigenfunction and on all the stable eigenvalues can be found in \cref{sec:Kato,sezione-forma-normale}. In both methods, we cannot rely on a standard perturbative approach, since the problem is singular when $\alpha\to 0$: for instance, $\boldsymbol{\cL}_{\nu,\alpha}(\sin(x))=\cO(1/\alpha)$. We then have to carefully exploit the particular structure of the problem at hand to obtain spectral instability for $\boldsymbol{\cL}_{\nu,\alpha}$ from a simpler, leading operator, which is also singular in $\alpha$. 
Both strategies introduce flexible techniques that were not used before to study stability of shear flows, and we believe that they can be potentially useful also for tackling other hydrodynamic stability problems beyond ``one-dimensional'' laminar flows as shear flows or vortices.

 \subsection{Comparison of our main theorem  with the previous literature}
 We now review some of the main results available in the literature related to our main Theorem \ref{th:mainLin}.

 \subsubsection{Instability of shear flows}
The study of dynamical properties of perturbations of laminar solutions, such as shear flows, is a foundational in hydrodynamic stability problem \cite{drazin2004hydrodynamic}. Classical contributions were made by Kelvin, Rayleigh, Orr, Heisenberg, C.C. Lin, Tollmien, among many others, between the late nineteenth and early twentieth centuries. For the Navier–Stokes equations, fundamental stability issues are closely tied to the boundary conditions: in the presence of physical boundaries with no-slip conditions at the walls, boundary layers can trigger severe instability mechanisms, see e.g.~\cite{Grenier16,Grenier16Adv,bian2023asymptotic}, as we shall discuss in more detail later on. On the other hand, fluids with periodic boundary conditions and external forcing are frequently considered in physical applications, such as two-dimensional turbulence theories, geophysical models, and molecular dynamics \cite{boffetta2012two,howard2020stability}. Motivated by the connection with turbulence, Kolmogorov in the 1950s \cite{arnol1960kolmogorov} proposed, within the same framework as the present paper, to investigate the behavior of perturbations of the shear flow $U(y)=\sin(y)$ for different values of $\alpha$ and $\nu$\footnote{According to \cite{arnol1960kolmogorov}: ``He hypothesized that for small $\nu$, a turbulent solution should appear (in the sense that there exists a non-trivial invariant measure $\mu_\nu$ in the space $(v_1,v_2)$) and that $\mu_{\nu} \to \mu_0$, where the limiting measure is concentrated on continuous functions.''}. 
The stability properties of this \textit{Kolmogorov flow} were first investigated in the pioneering work of Meshalkin and Sinai \cite{Mehsalkin61Investigation}: using continued-fraction methods, they proved that the flow is spectrally stable when $\alpha>1$ and unstable for sufficiently small $\alpha$, depending on $\nu>0$. Yudovich \cite{Yudovich66} subsequently studied a long-wave instability mechanism for general shear flows: by formally expanding in $\alpha$ (see also \cite{revina2017stability}), he predicted the instability expressed in \cref{intunstableeigenvalue}. This long-wave instability was later confirmed numerically, as shown for instance in \cite{green1974two,howard2020stability,okamoto1993bifurcation}. To the best of our knowledge, these general instability phenomena for the 2D Navier–Stokes equations in periodic domains have not been investigated further from a rigorous mathematical perspective, and our Theorem \ref{th:mainLin} appears to provide the first rigorous justification of the predictions in \cite{Yudovich66}.

In the presence of physical boundaries with no-slip velocity conditions—namely, in domains such as $(x,y)\in \RR\times \RR_{+}$ or $\RR\times [0,2]$ with $v|_{x=0}=(0,0)$—important breakthroughs were achieved by Grenier, Guo, and Nguyen \cite{Grenier16,Grenier16Adv}, and more recently revisited by Bian and Grenier \cite{bian2023asymptotic}. Their work provided a rigorous justification of several results previously predicted in the classical applied literature. In particular, their spectral instability theorem applies to a broad class of shear flows within a suitable range of physical parameters. Remarkably, even flows known to be stable for the Euler equations ($\nu=0$) become unstable at large but finite Reynolds numbers in a long-wave regime, highlighting the strongly destabilizing role of viscosity through the formation of boundary layers. These instabilities are spatially concentrated near the boundary and occur, in our notation, when $1\ll \alpha\approx\nu^{-p}$ for an appropriate range of exponents $0<p_0\leq p\leq p_1<1$, where $p_0$ and $p_1$ depend only on the shear profile $U(y)$.

For the 2D Euler equations, different instability mechanisms are available even in the absence of a long-wave regime. A first necessary condition for instability is Rayleigh’s criterion \cite{drazin2004hydrodynamic}, which requires the presence of an inflection point, i.e.~$U''(y_c)=0$ for some point $y_c$ in the domain. This condition, however, is not sufficient\footnote{As exemplified by the spectral properties of the Couette flow $U(y)=y$ in a bounded channel $\TT\times [0,1]$.}. Z. Lin \cite{lin2003instability} proved that for a broad class of shear flows (and vortices) with an inflection point, spectral instability does in fact occur. The proof provides a rigorous justification of Tollmien’s procedure, originally proposed in the 1930s. 
In addition, Z. Lin established another general abstract instability criterion in \cite{Lin04CMP}, valid for any steady state of the form $\psi_E=F(\omega_E)$. An application of this result implies the instability of Kolmogorov-type shear flows $U(y)=\sin(my)$ in the domain $\TT\times \TT$, for certain values of $m\in \mathbb{N}$, as shown in \cite{Latushkin18JMFM}. More recently, Z. Lin and C. Zeng \cite{Lin22Zeng}, among other results, generalized and made more quantitative the aforementioned criterion. Building on this framework, Liao, Lin, and Zhu \cite{liao2023stability} applied it to establish instability of certain Cat’s eye type flows.

 \subsubsection{Linear to nonlinear instability}
Once the linear stability properties of a given steady state are understood, it is natural to ask whether linear instability implies nonlinear instability (in an appropriate sense). For the Navier–Stokes equations, Yudovich \cite{yudovich1989linearization} showed that linear and nonlinear stability/instability are indeed related, at least when the velocity is measured in the $L^p(D)$ norm, with $D$ a bounded domain in dimension $n$ and $p\geq n$. Subsequently, Friedlander, Strauss, and Vishik \cite{friedlander1997nonlinear} proved a fairly general result for the Euler equations, establishing that nonlinear instability in $H^s(D)$ with $s>n/2+1$ holds whenever the linearized problem around a given steady state possesses a sufficiently large unstable eigenvalue. As a concrete example, \cite{friedlander1997nonlinear} extended the methods of \cite{Mehsalkin61Investigation} to demonstrate the passage from linear to nonlinear instability in 2D Euler for shear profiles $U(y)=\sin(my)$ on $\TT\times \TT$, with $m>1$ satisfying a Diophantine condition.
Subsequently, Grenier \cite{Grenier00CPAM} established another general criterion for the Euler equations, showing that sufficiently strong linear instability implies nonlinear instability of the velocity in $(L^\infty \cap L^2)(D)$. Later, Z. Lin \cite{Lin04IMRN} extended and refined the results of \cite{friedlander1997nonlinear} for the 2D Euler equations, while Friedlander, Pavlović, and Shvydkoy \cite{Friedlander06CMP} generalized the results of \cite{friedlander1997nonlinear,yudovich1989linearization} to demonstrate the passage from linear to nonlinear instability for the Navier–Stokes equations in $L^p(D)$ for any $1<p<\infty$.
\begin{remark}[Nonlinear long-wave instability]
All the aforementioned results on linear-to-nonlinear instability rely on unstable eigenvalues whose size is at least comparable to that of the most unstable eigenvalue(s). For the Navier–Stokes equations, in view of these results and the standard properties of the Laplacian, the presence of an unstable component in the spectrum of $\boldsymbol{\cL}_{\nu,\alpha}$ suffices to guarantee a nonlinear instability statement. However, we do not characterize the full spectrum outside the long-wave regime (in particular, for $\delta_0\nu\leq \alpha|k|\leq 1$), and thus it remains unclear whether the eigenvalues identified in Theorem \ref{th:mainLin} are indeed the most unstable ones. Consequently, existing results cannot be directly applied to promote the linear instability established in Theorem \ref{th:mainLin} to a nonlinear instability statement. More precisely, one would like to track the expected exponential growth of the eigenfunction in \cref{intunstableeigenvalue}, over a time scale on which the nonlinear correction remains negligible.
\end{remark}

\subsubsection{Stability of shear flows in tori}

We also briefly mention related results on the stability of shear flows, a topic that has received considerable attention in recent years. The main focus has been on understanding quantitative stability properties, such as \textit{inviscid damping} and \textit{enhanced dissipation}. Inviscid damping refers to the weak convergence, as $t\to \infty$, of vorticity perturbations toward an $x$-independent state in the inviscid regime, whereas enhanced dissipation describes exponential convergence toward a simpler $x$-independent diffusive dynamics on subdiffusive time scales. 
For the Kolmogorov flow in domains $\TT_\alpha\times \TT$ with $\alpha\geq 1$, both linear and nonlinear problems have been studied in \cite{Wei20Linear, Lin19Metastability, wei2019enhanced, Ibrahim19Pseudo, CotiZelati23Stationary}. More recently, Beekie, Chen, and Jia \cite{beekie2024uniform} addressed enhanced dissipation for more general shear flows with $\alpha>1$ under suitable spectral assumptions. 
The mathematical interest in these stability properties intensified after the breakthrough of Bedrossian and Masmoudi \cite{Bedrossian15PIHES}, who proved nonlinear inviscid damping for the Couette flow in $\TT\times \RR$. Many other notable stability results for shear flows in periodic channels or vortices have also been obtained, e.g.~\cite{Ionescu23Acta, Masmoudi24Annals, Bedrossian19Ann}, though we do not discuss them further here, as they fall outside the main focus of this paper.

Let us conclude our comparison with the previous literature with some remarks related to our main result Theorem \ref{th:mainLin}.
\begin{remark}[On the stability for $\alpha|k|>1$]
    Considering the operator $    \boldsymbol{\cL}_{\nu,\alpha}$ in the class of functions $\{e^{ikx}f(y)\, :\, f\in L^2(\TT)\}$ with $\alpha|k|>1 $ (the opposite of our regime in \cref{hyp:U}), a trivial scaling would put ourselves in the same setting of \cite{beekie2024uniform}. Thus, one can conclude the same stability statements of \cite{beekie2024uniform}, provided that the shear flow under consideration satisfies their hypothesis.
\end{remark}
\begin{remark}[On Taylor dispersion]
\label{rem:Taylorintro}
    Considering the operator $    \boldsymbol{\cL}_{\nu,\alpha}$ without the term involving $U''$, we recover the advection-diffusion operator
    \begin{align}
    \label{def:AdvDif}
\boldsymbol{\mathcal{T}}_{\nu,\alpha}=\nu\Delta_{\alpha}-\alpha U(y)\de_x.
    \end{align}
This operator naturally arises when considering a passive scalar advected by the shear flow $(U(y),0)$ and undergoing standard molecular diffusion. For the operator $\boldsymbol{\mathcal{T}}_{\nu,\alpha}$ in the regime $\alpha \ll \nu \leq 1$, there is a well-known stability mechanism called \textit{Taylor dispersion} \cite{taylor1953dispersion}. The interplay between advection and diffusion results in the convergence towards a simpler large-scale state on a time scale of $\cO(1)$, which is much faster than the standard diffusive time scale $\cO(\nu^{-1})$. Moreover, one can identify an effective diffusion equation modeling the evolution of averaged quantities.\footnote{In fact, Taylor dispersion can also be seen as a classical example of homogenization.}

This stability mechanism was also recently revisited and rigorously studied in \cite{coti2023enhanced,Beck20ARMA}. In particular, denoting by $\boldsymbol{\mu}_{\nu,\alpha,k}$ the largest eigenvalue of $\boldsymbol{\mathcal{T}}_{\nu,\alpha}$ in the class $\{e^{ikx} f(y)\, :\, f \in L^2(\TT)\}$ (in analogy with \cref{intunstableeigenvalue}), one expects
\begin{equation}
   \boldsymbol{\mu}_{\nu,\alpha k} \approx - \frac{(\alpha k)^2}{\nu} \Big(\nu^2 + \frac{C_U}{\nu}\Big),
\end{equation}
where $C_U>0$ is a constant depending only on $U$. This spectral property provides a way to quantify the effective diffusion coefficient. In \cite{Beck20ARMA}, for a passive scalar in an infinite pipe, it was proved that $C_U=\|\partial_y^{-1}U\|_{L^2}^2$. In our setting, by adapting Kato’s spectral approach from \cref{sec:Kato}, one can likewise show that $C_U=\|\partial_y^{-1}U\|_{L^2}^2$, as explained in more detail in Remark \ref{rem:Taylor}.\\ \indent
This highlights that for both operators $\boldsymbol{\cL}_{\nu,\alpha}$ and $\boldsymbol{\mathcal{T}}_{\nu,\alpha}$, the largest eigenvalue of $\nu\Delta_\alpha$, that is  $-\nu (\alpha k)^2$, is modified by a factor involving $\|\partial_y^{-1}U\|_{L^2}^2$. However, the presence of the term $U''\Delta_\alpha^{-1}$ in $\boldsymbol{\cL}_{\nu,\alpha}$ introduces a much more singular structure when $\alpha \ll 1$, which dramatically changes the spectral properties. In this regime, the interplay between diffusion and the Rayleigh operator gives rise to the destabilizing mechanism described in Theorem \ref{th:mainLin}.
\end{remark}

\subsection{Two approaches to instability}
\label{sec:Katonormal}

Before discussing the ideas behind the two approaches, we make some preliminary simple reductions that are common to both of them. 
First of all, observe that the operator $    \boldsymbol{\cL}_{\nu,\alpha}$ \cref{def:Lalphanu} decouples $x$-frequencies, meaning that we can take the Fourier transform in $x$ and study the $k$-by-$k$ operator 
\begin{align}
\cL_{\nu,\alpha k}\coloneqq \nu(\de_{yy}-\alpha^2k^2)-\im\alpha k\big(U(y)-U''(y)(\de_{yy}-\alpha^2k^2)^{-1}\big),
\end{align}
where $\cL_{\nu,\alpha k}:H^2(\TT)\to L^2(\TT)$.
Then, since our aim is to understand the long-wave instabilities, from now on we denote 
\begin{align}
    \eps
    \coloneqq\alpha |k|.
\end{align}
Without loss of generality, we can assume that $k>0$. Indeed, it is not hard to check that for any eigenvalue $\lambda$ of $\cL_{\nu,\alpha k}$ with $k>0$, then $\overline{\lambda}$ is an eigenvalue of $\cL_{\nu,-\alpha k}=\overline{\cL_{\nu,\alpha k}}$. We then study the onset of instability, as $\eps \neq 0$, of the operator
\begin{equation}\label{cLnueps}
    \cL_{\nu,\eps} \coloneqq \cM_{\nu,\eps} - \im \eps \cR_\eps \,.
\end{equation}
where, with the definitions of $\Pi_0$ (the $y$-average) and $\Pi_{\neq}=\mathrm{Id}-\Pi_0$ given in \cref{def:Pi0}, we define
\begin{align}
\label{cD}&\cD_\eps\coloneqq\de_{yy}-\eps^2,\\
\label{cMcR}
&\cM_{\nu,\eps}\coloneqq\nu \cD_\eps - \frac{\im}{\eps} U''(y) \Pi_0 \, , \quad \cR_{\eps}\coloneqq U(y)+U''(y)(-\de_{yy}+\eps^2  )^{-1} \Pi_{\neq}.
\end{align}
Here, $\cM_{\nu,\eps}$ plays the role of the leading order operator with a simple structure, and we will show that $\sigma_{L^2}(\cM_{\nu,\eps})=\nu \sigma_{L^2}(\cD_\eps)$. Note that $\cM_{\nu,\eps}$ cannot be regarded as a perturbation of the standard $\cD_{\eps}$: it includes a term that becomes singular as $\eps \to 0$, but which acts only on functions with nontrivial average in $y$. 
 Instead,
$\cR_{\eps}$ is a Rayleigh type operator,  bounded  in $L^2(\TT)$, where the effects of the $0$-th mode in $y$ are removed.

\emph{The goal of both approaches is to transfer information from the spectrum of 
 $    {\cM}_{\nu,\eps}$ to the full $    {\cL}_{\nu,\eps}$, keeping in mind that both operators are singular as $\eps \to 0$. }
 We perform this either via Kato's spectral techniques to construct a suitable isomorphism, or by employing normal-form methods to construct an operator that ``block-diagonalizes'' $    {\cL}_{\nu,\eps}$. To tackle the singularity as $\eps \to0$, both methods take advantage of a crucial cancellation happening when the singular part of $\cM_{\nu,\eps}$ 
hits $\cR_\e$, based on the fact,  shown in Lemma \ref{lem:fundamentalcancelation}, that 
\begin{equation}
\label{eq:crucialintro}
\Pi_0\cR_\eps=\cO(\eps^2). 
\end{equation}

 \subsubsection{\textbf{Kato reduction}}
 The key point in the classical Kato perturbation theory \cite{Kato}, 
 is the transfer of information from a spectral projection $Q$, associated with the leading operator $\cM_{\nu,\eps}$, to a spectral projection $P$ of the full operator  $\cL_{\nu,\eps}$, where
\begin{equation}
\label{RieszIntro}
    Q \coloneqq -\frac{1}{2\pi\im}\oint_{\Gamma} (\cM_{\nu,\e}-\lambda)^{-1} \drm \lambda\,,\quad P \coloneqq -\frac{1}{2\pi\im}\oint_{\Gamma} (\cL_{\nu,\e}-\lambda)^{-1} \drm \lambda\,,
\end{equation}
and $\Gamma$ is a loop in the common resolvent set $\rho_{L^2}(\cM_{\nu,\e})\cap \rho_{L^2}(\cL_{\nu,\e})$ that encircles an isolated eigenvalue of $\cM_{\nu,\e}$. 
 In previous applications of Kato theory the unperturbed operator was either a normal operator, see e.g.\ \cite{Kappeler1991}, or a Fourier multiplier, see e.g.\ \cite{BMV, BMV1,BMV2, BCMV, bianchini2025instabilities}. Here a first difficulty is that the operator $\cM_{\nu,\e}$ does not possess, as a whole, these particular properties. Our analysis will then rely crucially on the fact that the term of size $\eps^{-1}$ in $\cM_{\nu,\eps}$ is a rank-$1$ update of  the self-adjoint Fourier multiplier $\nu\cD_\eps$ in \cref{cD}. This structure enables us to compute its resolvent via a Sherman–Morrison type formula \cref{originalShermanMorrison}, despite the singular behavior as $\eps \to 0$.
A second difficulty arises when one tries to apply the Kato scheme as it has been developed in previous literature. Indeed, a standard perturbative approach requires a bound of the form 
\begin{equation}\label{hopelessestimate}
    \| P - Q\|_{\cB(L^2,L^2)} = o(1)\,,
\end{equation}
in a regime of the parameters where $\cL_{\nu,\e}$ is sufficiently close to $\cM_{\nu,\e}$. In our case, an estimate of the form \cref{hopelessestimate} is hopeless, because also $P-Q$ includes a large rank-$1$ operator of size $\cO(\nu^{-2})$, see Lemma \ref{lem:P-Q} below. Nonetheless, since terms of order $\nu^{-2}$ will always contain the projection $\Pi_0$, we are able to invert the operator
\begin{equation}\label{Katocondition}
    \uno - (P-Q)^2 = \big( \uno - (P-Q) \big) \big( \uno + (P-Q) \big) \, ,
\end{equation}
by relying on  \cref{eq:crucialintro} and the fact that each of its components is a rank-$1$ update of an isomorphism close to the identity (an operator of the form $\uno + o(1)$). Thus, we can apply again the Sherman-Morrison type formula \cref{originalShermanMorrison} to invert said rank-$1$ updates. We stress that while \cref{hopelessestimate} is only a \emph{sufficient condition}, the invertibility of \cref{Katocondition} is \emph{equivalent} to proving that 
the operator $\uno - P - Q$
is an isomorphism between the ranges of $Q$ and $P$ (and between their kernels as well). Once this isomorphism is constructed, thanks to the properties of $\cM_{\nu,\eps}$ and the projections we know that eigenvalues of $\cL_{\nu,\eps}$ remain sufficiently close.  Then, the final outcome of Kato's method in our setting is described in  Lemma \ref{lem:katoeigen} below, containing a description of the unstable eigenvector --as well as a suitable decomposition of the complementary stable hyperplane into a direct sum of two-dimensional invariant subspaces-- of $  \cL_{\nu,\eps}$  in terms of the eigenvectors of $\cM_{\nu,\eps}$ and the spectral projections. 
Finally, we observe that the boundedness of the perturbation operator 
$\cR_\e$ makes us able to  characterize the whole\footnote{see \cite{BCMV} for a case in which an unbounded perturbation keeps from describing the spectrum of a Hamiltonian operator outside fixed horizontal strips in the complex plane.} spectrum of $\cL_{\nu,\e}$ for fixed values of $\e,\nu>0$ when $\eps$ is sufficiently small.

The expansion of the unstable eigenvalue and eigenfunction \cref{unstableeigenvalue} is then computed in \cref{sec:unstable}.

\subsubsection{\textbf{Normal-Form reduction}}
We now describe the normal-form approach appearing in \cref{sezione-forma-normale} for the analysis of the spectrum of the linear operator ${\mathcal L}_{\nu, \varepsilon}$ given in \cref{cLnueps,cD,cMcR}. Normal-form methods have been fruitfully used in Fluid Mechanics in order to construct periodic and quasi-periodic traveling-wave solutions of Euler and Navier-Stokes equations in several situations: see for instance \cite{BaldiMontalto,FMVV,BFMT} for the Euler equations on rectangular tori with quasi-periodic forcing terms, \cite{FMM} for quasi-periodic traveling waves near Couette flows and \cite{vortex1,vortex2,vortex3,vortex4} for quasi-periodic vortex patches. In the aforementioned papers, one deals with the so-called problem of small divisors, namely the fact that the linearized equation at the origin has spectrum accumulating to zero. The normal-form techniques are used essentially to integrate the linearized equations at any approximate solutions and for giving a sharp analysis of the spectrum of such operators. This methods allow to construct a bounded and invertible transformation (on Sobolev spaces) that molds the linearized operator into a much simpler one with a prescribed structure, i.e.\ in most cases a diagonal or block-diagonal operator. Consequently, the spectrum of the transformed operator is easily studied. In this case, we will consider $\cM_{\nu,\eps}$ as the unperturbed operator whose spectrum is equal to the one of $\nu\cD_{\eps}$.  Taking into account the ``double resonances" of  $\nu \cD_{\eps}$ on $L^2(\TT)$ ($-\nu\eps^2$ is a simple eigenvalue and $- \nu (j^2+\eps^2)$, $j \neq 0$, is double), we block diagonalize the operator ${\mathcal L}_{\nu, \varepsilon}$ and we compute sharp asymptotic expansion of its eigenvalues. In particular, we obtain that the perturbed eigenvalue $\lambda^{(0)}_{\nu, \varepsilon}$, bifurcating from $j=0$, is simple and has positive real part, whereas for any $j \in\NN $, there are two (or one double) perturbed eigenvalues $\lambda^{(\pm j)}_{\nu, \varepsilon}$ which are $\varepsilon$-close to $- \nu j^2$. The major difficulty is again the presence of a non-perturbative term in ${\mathcal L}_{\nu, \varepsilon}$. In particular, $\frac{\im}{\eps} U''(y) \Pi_0$ in ${\mathcal M}_{\nu, \varepsilon}$ (see \cref{cMcR}) has size $\cO(\varepsilon^{- 1})$. The key point to deal with this term is that this operator is rank-1 and nilpotent.
We describe more in detail our strategy, which is split in two parts: 
\begin{enumerate}[label=(\roman*)]
\item First, in \cref{normal-form-modo-zero}, we decouple the unstable direction of $\cL_{\nu,\e}$ from its stable invariant hyperplane.  In this way, we isolate the unstable eigenvalue and obtain its asymptotic expansion.
\item Then, in \cref{block-diagonalization-stable-part}, we study the restriction of $\cL_{\nu,\e}$ to the stable hyperplane obtained in part $(\mathrm{i})$. This operator is a perturbation of size $\cO(\varepsilon)$ of $\nu \partial_{yy}$ --restricted to Fourier modes different from zero--
 and it can be transformed into a new operator with only $2\times 2$ diagonal blocks --related to the double eigenvalues of its leading part--. The block-diagonalization follows by a fixed-point argument, which we can implement because of the absence of small divisors.
\end{enumerate}

\medskip
\noindent 
{{\bf \textrm{(i)}} \textit{Decoupling of the unstable direction.}} 
To study $\cL_{\nu,\eps}$, we can use the natural direct sum decomposition of $H^s(\TT)=\mathbb{C}\oplus H^s_{0}(\TT)$, with $s\geq 2$ and $H^s_0(\TT)$ being zero-average functions, to represent  the operator as the matrix
$$
\mathtt{L}_{\nu,\eps}\coloneqq \begin{bmatrix}
\Pi_0 {\mathcal L}_{\nu, \varepsilon} \Pi_0 & \Pi_0 {\mathcal L}_{\nu, \varepsilon} \Pi_{\neq} \\
\Pi_{\neq} {\mathcal L}_{\nu, \varepsilon} \Pi_0 & \Pi_{\neq} {\mathcal L}_{\nu, \varepsilon} \Pi_{\neq}
\end{bmatrix}.
$$
By properties of the projection $\Pi_0$, 
one can think of $\Pi_0 {\mathcal L}_{\nu, \varepsilon} \Pi_{\neq}$ and $\Pi_{\neq}{\mathcal L}_{\nu, \varepsilon} \Pi_0 $ as an infinite dimensional row and column vector respectively. The part with $\Pi_0\cL_{\nu,\eps}\Pi_0$ does instead encode the simple eigenvalue associated to $\cM_{\nu,\eps}$, which is the one being perturbed to generate the instability. To capture this, we aim at simplifying the structure of the matrix ${\mathtt L}_{\nu,\eps}$ by conjugating it with a  matrix ${\mathtt T}$ so that 
\begin{equation}
\label{eq:introNF}
    {\mathtt T} \, {\mathtt L}_{\nu, \varepsilon} {\mathtt T}^{- 1} = \begin{bmatrix}
\lambda_{\nu,\eps}^{(0)} & \boldsymbol{0}^\top \\
\boldsymbol{0} & {\mathcal L}_{\nu, \varepsilon}^{(1)}
\end{bmatrix}.
\end{equation}
Here $\lambda_{\nu,\eps}^{(0)}$ is exactly the eigenvalue in Theorem \ref{th:mainLin} and the unstable direction is indeed $(1,\boldsymbol{0})^{\top}$ in the new basis, hence the desired decoupling would be achieved. Instead, $\cL_{\nu,\eps}^{(1)}$ is a is a suitable modification of $\cL_{\nu,\eps}$ acting on $H^s_0(\TT) $, accounting for the stable part of the operator discussed afterwards. To justify this decoupling of the unstable direction, it remains to find ${\mathtt T}$ (and block-diagonalize $\cL_{\nu,\eps}^{(1)}$). Inspired by the finite dimensional case, we make an ansatz of the form
$$
{\mathtt T} \coloneqq  \begin{bmatrix}
1 & X^\top \\
Y & {\rm Id}
\end{bmatrix}\,,
$$
where $X,Y$ are two $H^{s}_0(\TT)$ functions and ${\mathtt T}$ is invertible 
if $\langle X, Y \rangle_{L^2(\TT)} \neq 1$  (see Lemma \ref{invertibilita mathtt T}). To obtain the desired structure after the conjugation, one has to solve quadratic expressions for both $X$ and $Y$, which is done by a fixed-point argument in Lemmas \ref{lem:vheq} and \ref{lem:hheq}. The outcome is that 
$$
X = \cO\big( \frac{\e^3}{\nu} \big) \qquad \text{and} \qquad Y = \cO\big( \frac{1}{\nu \e} \big)\,,
$$
which imply that our transformation ${\mathtt T} = {\rm Id} + \cO(\varepsilon^{- 1})$ is ``far from the identity'', reflecting from the non-perturbative nature of the problem (compare with the need of applying the Shermann-Morrison formula in Kato's approach).
To compute the desired expansion of $\lambda_{\nu,\eps}^{(0)}$, we have to precisely expand the quadratic form involving $X,Y$ that is defining the eigenvalue, which is done in Lemma \ref{primo coniugio normal-form}.

\medskip

\noindent

{{\bf \textrm{(ii)}}. \it Block-diagonalization of the stable part.} Having at hand the decomposition in \cref{eq:introNF}, we need to block-diagonalize the  operator $\cL_{\nu,\eps}^{(1)}$, where the notion of a block-diagonal operator is the natural one based on the standard Fourier basis, see \cref{def:blockdiagonal}. A first technical step is to introduce a suitable notion of block-diagonal decay norm $|\cdot|_s$, that  provide us with a nice functional framework, see \cref{def-decay-norm} and the properties listed in the Lemmas below. For regular enough shear flow profiles, we are allowed to write  
\begin{equation}
    \cL_{\nu,\eps}^{(1)}=\nu\de_{yy}+{\mathcal Q}, \qquad |\cQ|_s=\cO(\eps).
\end{equation}
Here $\cQ:H^s_0(\TT)\to H^s_0(\TT)$ is a linear operator depending on $X,Y$ defined before. In this formulation the problem becomes amenable to the application of the usual strategy in normal forms methods. Namely, whenever $\varepsilon\ll \nu $, we can  construct by a fixed point argument 
a map $\Psi:H^s_0(\TT)\to H^s_0(\TT)$ with $|\Psi|_s = \cO(\varepsilon \nu^{- 1})$ such that 
$$
({\rm Id} + \Psi)^{- 1} {\mathcal L}_{\nu, \varepsilon} ({\rm Id} + \Psi) = \nu \partial_{yy} + {\mathcal Z}
$$
where $|{\mathcal Z}|_s=\cO(\varepsilon)$ and is  \emph{block-diagonal} with $2\times 2$ blocks. This is proved in Proposition \ref{prop-block-diag-contraction}. Thanks to this structure, we can easily conclude that the spectrum of $\nu \partial_{yy} + {\mathcal Z}$ is the union of the spectra of each block-diagonal piece, which are given by $\{ \lambda_{\nu,\eps}^{(j)}\}_{ j \in \ZZ\setminus\{0\} }$ with $\lambda_{\nu,\eps}^{(j)} = - \nu j^2 + \cO(\varepsilon)$. Hence all the eigenvalues of ${\mathcal L}_{\nu, \varepsilon}^{(1)}$ have negative real parts and the stable part is also fully characterized.

\subsection{Notation}
Given a function $f\in L^2(\TT)$, its Fourier decomposition is given by 
\begin{equation}\label{FourierCoefficients}
    f(y)=\sum_{j\in \ZZ}f_{j}e^{ijy}, \qquad f_{j}\coloneqq \frac{1}{2\pi}\int_{\TT}e^{-\im jy}f(y) \dd y.
\end{equation}
In particular we introduce the orthogonal projections
\begin{equation}
\label{def:Pi0}
\Pi_0f\coloneqq f_0=\frac{1}{2\pi}\int_{\TT}f(y)\dd y, \qquad \Pi_{\neq}\coloneqq \uno-\Pi_0\, ,
\end{equation}
and, for every $j \in {\mathbb N}$,
\begin{equation}
\label{def:Pij}
\Pi_j f(y) \coloneqq   f_j e^{\im j y} +  f_{-j} e^{-\im j y}\,, 
\end{equation}
so that $f = \Pi_0 f + \sum_{j \in {\mathbb N}} \Pi_j f
$.

Given a closed linear operator ${\mathcal R} : \mathrm{D}(\cR) \subseteq L^2(\TT) \to L^2(\TT)$, with $C^\infty(\TT) \subseteq \mathrm{D}(\cR) $, we can define its matrix elements as
$$
{\mathcal R}_j^{j'} \coloneqq  \langle {\mathcal R}[e^{i j' y}]\,,\, e^{i j y} \rangle_{L^2}, \quad \forall j, j' \in {\mathbb Z}\,.
$$
Furthermore, one has  
\begin{equation}
\label{def:Rij}
    {\mathcal R} = \sum_{j, j' \geq 0} \Pi_j {\mathcal R} \Pi_{j'}
\end{equation}
and for any $j, j' \in {\mathbb N}$, we can identify the operator $\Pi_j {\mathcal R} \Pi_{j'}$ with the $2\times 2$ matrix
$$
\Pi_j {\mathcal R} \Pi_{j'} \equiv \begin{pmatrix}
{\mathcal R}_j^{j'} & {\mathcal R}_j^{- j'} \\
{\mathcal R}_{- j}^{j'} & {\mathcal R}_{- j}^{- j'}
\end{pmatrix}\,. 
$$
In particular, the operator ${\mathcal R} $ is said to be  \emph{block-diagonal} when
\begin{equation}
\label{def:blockdiagonal}
\Pi_j {\mathcal R} \Pi_{j'} = 0, \quad \forall j, j' \geq 0, \quad j \neq j'\,. 
    \end{equation}

In the sequel, given a bounded linear operator $\cT:X\to Y$, with $X,Y$ being two normed spaces, we denote its standard operator norm as
\begin{equation}
    \|\cT\|_{\cB(X,Y)}\coloneqq \sup_{\|x\|_X=1}\|\cT x\|_Y.
\end{equation}

\medskip

\noindent
{\bf Acknowledgements} M. Colombo and M. Dolce were supported by the Swiss State Secretariat for Education, Research and Innovation (SERI) under contract number MB22.00034 through the project TENSE.  M. Dolce was supported also by the Swiss National Science Foundation (SNF Ambizione grant PZ00P2\_223294).
R. Montalto and P. Ventura  are  supported by the ERC STARTING GRANT 2021 “Hamiltonian Dynamics, Normal Forms and Water Waves” (HamDyWWa), Project Number: 101039762.

 The Views and opinions expressed are however those of the authors only and do not necessarily reflect those of the European Union or the European Research Council. Neither the European Union nor the granting authority can be held responsible for them.

\section{Linear instability via Kato  approach}
\label{sec:Kato}
 In this section, we aim at presenting the proof based on Kato's approach of the instability mentioned in Theorem \ref{th:mainLin}. The goal is the complete characterization of the spectrum of the operator $\cL_{\nu,\e}$ on $L^2(\TT)$. Note that in this section it is enough to assume that the shear profile $U(y)$ is in $C^2(\TT)$. This is clearly implied by the hypotheses of Theorem \ref{th:mainLin}, since $H^3(\TT)$ is compactly embedded in $C^2(\TT)$. The main result of this section is the following
\begin{theorem}\label{thm:linmain} Let $U\in C^2(\TT)$ be a shear flow profile with $0$ average. Then there exists constants $\delta_0\in (0,1)$ and $C>1$, depending only on $\|U\|_{C^2}$, such that, for every $\nu, \eps>0$ with
 $\eps\leq \delta_0\nu$, the following holds true for the operator $\cL_{\nu,\eps}$ defined in \cref{cLnueps}. 
 There exists a simple  eigenvalue $\lambda^{(0)}_{\nu,\eps}\in \sigma_{L^2}(\cL_{\nu,\eps})$ which admits the following asymptotic expansion in $\eps$ 
\begin{equation}\label{unstableeigenvalue}
\lambda^{(0)}_{\nu,\e} = \frac{\e^2}{\nu}\left( \|\de_y^{-1} U \|_{L^2}^2 - \nu^2 +\cO\big(\eps(\nu+\nu^{-1})\big)\right)
\end{equation}
associated with an eigenvector of the form
\begin{equation}\label{unstableeigenvector}
   \| V_{\nu,\e}^{(0)}(y)  -  (U(y)-\im \nu \eps)\|_{L^2}  \leq C \frac\eps\nu
   .
\end{equation}
In general
, we have
\begin{equation}\label{restofthespectrum}
   \sigma_{L^2} \big(\cL_{\nu,\e} \big) = \{\lambda_{\nu,\eps}^{(j
   )}\}_{j\in \ZZ} , \qquad |\lambda_{\nu,\eps}^{(j)}+\nu(j^2+\eps^2)|\leq \frac{\nu}{2}.
\end{equation}
\end{theorem}

The spectral properties encoded in Theorem  \ref{thm:linmain} are illustrated in \cref{fig:SchematicSpectrum}. Observe that the estimate for the location of $\lambda_{\nu,\eps}^{(j)}$ in \cref{restofthespectrum} can be sharpened by replacing $\nu/2$ on the right-hand side with $C\eps$, performing an expansion analogous to the one used in the proof of \cref{unstableeigenvalue}. In the same way, a more accurate description of the corresponding eigenfunctions could also be obtained. Furthermore, with the present method one can explicitly compute higher-order terms in the asymptotic expansion of $\lambda^{(0)}_{\nu,\eps}$ and its associated eigenfunction in \cref{unstableeigenvalue}. Since this is not the main focus of the paper, we omit these details here; they will, however, naturally emerge in the normal-form approach (see Theorem  \ref{teorema coniugio cal Lk}).
The rest of the section is devoted to the proof of Theorem \ref{thm:linmain}.
\begin{figure}[h!]
    \centering
    \begin{tikzpicture}[scale=1]

  \draw[->] (-7,0) -- (2,0) node[below] {\(\Re\)};
  \draw[->] (0,-1) -- (0,2) node[left] {\(\Im\)};

      \filldraw[blue] (-.2,0) circle (2pt);

  \foreach \n in {1,2} {
    \pgfmathsetmacro\x{-1.5*\n*\n - 0.2} 
    \filldraw[blue] (\x,0) circle (3pt);
  }

  \foreach \n in {1,2} {
    \pgfmathsetmacro\x{-1.5*\n*\n - 0.2}
    \draw[green!50!black, dashed] (\x,0) circle (.7);
  }
    \filldraw[red] (-1.5-.3-0.25,0.25) circle (2pt);
    \filldraw[red] (-1.5-.2+0.25,-0.2) circle (2pt);

    \filldraw[red] (-1.5*4-.2+0.4,0.2) circle (3pt);

  \filldraw[red] (.3,.25) circle (2pt);
  \node[above right] at (.3,.1) {\(\lambda^{(0)}_{\nu,\eps}\)};
  \node[above] at (-.5,0) {\(-\nu\eps^2\)};

  \node[right] at (-5.4,1.8) {\textcolor{blue}{Spectrum of \(\mathcal{M}_{\nu,\eps}\)}};
  
  \node[right] at (-5.4,1.3) {\textcolor{red}{Spectrum of \(\mathcal{L}_{\nu,\eps}\)}};
\end{tikzpicture}

    \caption{A schematic representation of the spectral picture of the operators $\cM_{\nu,\eps}$ and $\cL_{\nu,\eps}$. Bigger dots correspond to double eigenvalues. The simple eigenvalue $-\nu \eps^2$ of $\cM_{\nu,\eps}$ is mapped in the unstable one $\lambda_{\nu,\eps}^{(0)}$. All other eigenvalues $-\nu(j^2+\eps^2)$ of $\cM_{\nu,\eps}$ are double and they either split in two eigenvalues $\lambda^{(\pm j)}_{\nu,\eps}$ or they form a new double eigenvalue. In both cases, they remain  in a ball of radius $\nu/2$ (which can be refined to a ball of radius $\cO(\eps)$).}\label{fig:SchematicSpectrum}
\end{figure}

\subsection{The operators $\cM_{\nu,\e}$ and $\cR_{\nu,\e}$}
We begin with an analysis of the leading and perturbative components of the operator $\cL_{\nu,\e}$ as in \cref{cLnueps,cMcR}. First of all, we observe that $\cR_\eps$ is bounded from $L^2(\TT)\to L^2(\TT)$ with the explicit bound given by 
    \begin{equation}\label{boundcR}
\|\cR_{\eps}\|_{\cB(L^2,L^2)}\leq \|U\|_{L^\infty}+\|U''\|_{L^\infty}.
    \end{equation}
    Then, the term of order $\eps^{-1}$ in $\cM_{\nu,\eps}$ only acts through the projection onto zero modes $\Pi_0$. This projection, when combined with the perturbative term $\cR_{\eps}$, entails crucial cancellations which will be used in the rest of the paper.
\begin{lemma}\label{lem:fundamentalcancelation} Let $\cR_\eps$ be the operator defined in \cref{cMcR}. Then, for all $f\in L^2(\TT)$, the following holds true
\begin{equation}\label{Cancelation}
    \Pi_0 \cR_\e  f = \e^2 \langle (\de_y+\eps)^{-1} U,(\de_y+\eps)^{-1} \Pi_{\neq} f\rangle_{L^2}.
\end{equation}
In particular, $\Pi_0\cR_\e 1=0$ and 
\begin{equation}\label{estimatePi0cR}
    \|\Pi_0 \cR_\e\|_{\cB(L^2,\RR)} = \e^2 \frac{\| (\de_y+\e)^{-1} U \|^2_{L^2}}{\|U\|_{L^2}} .
\end{equation}
\end{lemma}
\begin{proof} The identity   $\Pi_0\cR_\e 1 = \Pi_0 U = 0$ follows by \cref{hyp:U}. Let $f \in L^2 $, then 
\begin{align}
\Pi_0\cR_\e f &= \Pi_0\cR_\e \Pi_{\neq}f =\Pi_0 \Big(\big( U + U''(-\de_{yy}+\eps^2  )^{-1} \big)\Pi_{\neq}f \Big)  \\
\label{auxPi0cR}&= \langle U , \Pi_{\neq}f\rangle + \langle U, \de_{yy}(-\de_{yy}+\eps^2 )^{-1} \Pi_{\neq}f\rangle ,
\end{align}
where in the second term we integrated by parts twice (using the self-adjointness of $\de_{yy}$ in $L^2$). Since $\de_{yy}=(\de_{yy}-\eps^2)+\eps^2$, note that 
\begin{equation}
    \langle U, \de_{yy}(-\de_{yy}+\eps^2 )^{-1} \Pi_{\neq}f\rangle=-\langle U,  \Pi_{\neq}f\rangle+\eps^2\langle U, (-\de_{yy}+\eps^2 )^{-1} \Pi_{\neq}f\rangle.
\end{equation}
Hence, the first term in \cref{auxPi0cR} cancels with the first term on the right hand side of the identity above. From Parseval's theorem, we have that
\begin{equation}
    \Pi_0\cR_\e f = \eps^2\sum_{j\in \ZZ\setminus \{0\}} \frac{1}{j^2+\eps^2}U_j \overline{f}_j=\eps^2\sum_{j\in \ZZ\setminus \{0\}} \frac{1}{(\im j+\eps)(-\im j+\eps)}U_j \overline{f}_j,
\end{equation}
from which the identity \cref{Cancelation} follows. Note that one could have also used $-\de_{yy}+\eps^2  = (-\de_y + \e)(\de_y +\e) $ and then integration by parts.
For the bound \cref{estimatePi0cR}, it is enough to observe that the maximum of $\|\Pi_0 \cR_\e f\|_{L^2}$ on $\{\| f\|_{L^2} = 1\}$  is attained at $f= U/\|U\|_{L^2}$.
\end{proof}

Then, we have the following basic spectral properties of $\cM_{\nu,\eps}$.
\begin{lemma}\label{unperturbedspectrum}
Let $\eps  \neq 0 $ and $\cM_{\nu,\eps}: H^2(\TT)\to L^2(\TT)$ be the operator defined in \cref{cMcR}. Then, 
\begin{equation}
\label{eq:idspecMD}
 \sigma_{L^2}(\cM_{\nu,\eps})=\nu\sigma_{L^2}(\cD_{\eps})=-\nu \{(j^2+\eps^2)\}_{j\in \ZZ}  .
\end{equation}
Moreover, the set of eigenfunctions $\{U(y)-\im \eps\nu,e^{\pm \im j y}\}_{j\in \NN}$, respectively associated with the eigenvalues $\{-\nu\eps^2,-\nu(j^2+\eps^2)\}_{j\in \NN}$,  is  a basis of  $L^2(\TT)$. 
\end{lemma}
\begin{proof} For $j\neq0$, it is enough to observe that $\cM_{\nu,\eps}\Pi_j=\nu \cD_{\eps}\Pi_j$. Instead, the fact that $\cM_{\nu,\eps}(U(y)-\im \eps \nu)=-\nu \eps^2(U(y)-\im \eps \nu)$ is a direct computation. The new basis coincides with the standard Fourier basis except for the new $0$-th eigenvector $U(y)-\im \eps \nu$. The latter is linearly independent from the eigenvectors $e^{\im j x}$, $j\in \ZZ\setminus\{0\} $, because it has nonzero average.
\end{proof}
The Lemma \ref{unperturbedspectrum} suggests that the dependence on $\nu$ can be easily tracked.
 Then,  the part of order $\eps^{-1}$ is a rank-$1$ operator in the sense that $\mathrm{Ran}(U''(y)\Pi_0)=\mathrm{span}\{U''(y)\}$, which has dimension $1$. This is another fundamental observation that helps in computing explicitly the resolvent of $\cM_{\nu,\eps}$ by viewing it as a rank-$1$ modification of $\nu \cD_{\eps}$. Indeed,  this kind of structure is well-known in the finite dimensional case, and we can extend the celebrated Sherman-Morrison formula to an infinite dimensional setting of interest in our case. We refer to \cite{DENG20111561} for a more general result.
\begin{lemma}
\label{lem:ShermanMorrisoninf}
    Let $H$ be a Hilbert space with inner product $\langle\cdot,\cdot\rangle_H$, $\cA:D(\cA)\subset H\to H $ be a closed, densely-defined, linear operator with  bounded inverse $\cA^{-1}$ and $f,g\in H$.   
    Then, the operator $\cA+f\langle g,\cdot \rangle_H$ is invertible if and only if $
        \langle g,\cA^{-1} f\rangle_H+1\neq 0$.    Moreover 
    \begin{equation}\label{originalShermanMorrison}
        (\cA+f\langle g,\cdot\rangle_H)^{-1}=\cA^{-1}-\frac{\cA^{-1}(f\langle g,\cdot\rangle_H)\cA^{-1}}{1+\langle g,\cA^{-1} f\rangle_H}.
    \end{equation}
    
\end{lemma}
The proof of this lemma is a direct computation that we include in \cref{rank1proof}. To apply the lemma above to the operator $\cM_{\nu,\eps}$, we simply observe that $\Pi_0=\langle 1,\cdot \rangle_{L^2}.$ Therefore we have the following.
\begin{lemma}\label{lem:ResolventcM} Let $ \zeta \in \rho_{L^2}(\cD_{\eps})$ belong to the resolvent set of $\cD_{\eps}$. Then, for every $\nu >0$,
\begin{align}\label{cMresolvent}
        (\cM_{\nu,\e} -\nu\zeta)^{-1} = \frac{1}{\nu} (\cD_\e -\zeta)^{-1}  +  \frac{u_{\e}(\zeta,y)}{\im \e \nu^2(\e^2+\zeta)}\Pi_0\,,
  \end{align}
where    $ u_{\e}(\zeta,y) \coloneqq  \big(\cD_{\eps}-\zeta\big)^{-1}  U''(y)= U(y)+ (\e^2+\zeta) \big(\cD_{\eps}-\zeta\big)^{-1} U(y)$.
Moreover,  The mapping $\zeta \mapsto u_\eps(\zeta,\cdot)$ is analytic on the half-plane $\{ z\in\CC\;:\; \Re\, z > - 1 \}$.
 
\end{lemma}
\begin{proof}
By \cref{eq:idspecMD}, one has $\nu \zeta \in \rho_{L^2}(\nu\cD_\eps)$. We can thus apply Lemma \ref{lem:ShermanMorrisoninf} with $H=L^2$, $\cA=\nu(\cD_{\eps}-\zeta)$, $f=-\im U''(y)/\eps$ and $g=1$ so that $\Pi_0=\langle g,\cdot\rangle_{L^2}$. We only need to check the condition for the invertibility. Since $\Pi_0$ commutes with $\de_{y}$ and $\Pi_0 U=0$, one has
\begin{equation}
    \langle 1, (\cD_\eps-\zeta)^{-1} U''(y)\rangle_{L^2}=\Pi_0\big((\cD_\eps-\zeta)^{-1} U''(y)\big)=0.
\end{equation}
Thus, by  Lemma \ref{lem:ShermanMorrisoninf} and since
$\Pi_0$ and $ \cD_\eps-\zeta$ commute, we have
\begin{equation}\label{compositionaux}
  (\cM_{\nu,\eps} -\nu \zeta)^{-1} = \frac1\nu \big(\cD_{\eps}-\zeta\big)^{-1}  + \frac{\im}{\e\nu^2} \big(\cD_{\eps}-\zeta\big)^{-1}  U''(y) \big(\cD_{\eps}-\zeta\big)^{-1} \Pi_0 \,.
\end{equation}

By \cref{cD}, the operator $(\cD_\e - \zeta)^{-1} \Pi_0 =- (\e^2+\zeta)^{-1} \Pi_0 $ always returns a scalar. Thus we can substitute   $(\cD_\e - \zeta)^{-1} U''(y) $ with $u_\e(\zeta,y)$ in \cref{compositionaux} to obtain \cref{cMresolvent}. 

To prove the desired properties of $u_{\eps}(\zeta,\cdot)$, note that we have
\begin{equation}
u_{\e}(\zeta,y) = -\sum_{j \neq 0} \frac{U_j e^{\im j y}}{j^2+\e^2+\zeta}\, .
\end{equation} 
Thus  the mapping $\zeta \mapsto u_{\e}(\zeta,y) $  has its greatest pole in $-(1+\e^2)<-1$ and is analytic for  $\Re\, \zeta > -1$.

\end{proof}

We now have at hand all the properties of operators $\cM_{\nu,\eps}$ and $\cR_{\nu,\eps}$ that are needed to start the study of the spectrum of $\cL_{\nu,\eps}.$

\subsection{Spectral projections and resolvent bounds}

We now proceed with the rigorous definition of the spectral projections involved in the analysis. Let $j\in \NN_0$ and $\Gamma_j$ be the loop in the complex plane parameterized by 
\begin{equation}\label{Gammaj}
   \Gamma_j \;:\; \TT \ni \theta \mapsto -(j^2+ \e^2)  + \tfrac12 e^{\im \theta} \in \CC\,.
\end{equation}
 It encircles the isolated double (or simple if $j=0$) eigenvalue $-(j^2+\e^2)$ that the self-adjoint operator $\cD_\e$ in \cref{cD} shares with the operator $\cM_{1,\e}$ in \cref{cMcR}.  The rest of the common spectrum of the two operators lies outside $\Gamma_j$. In particular
$
     dist\big(\Gamma_j, \sigma_{L^2}(\cD_\e)\big) = \tfrac12
$
 and, since $\cD_\e$ is self-adjoint,
\begin{equation}\label{estimateFourierMultiplier}
 \| (\cD_\e - \zeta)^{-1} \|\leq 2\,, \qquad \forall \zeta \in \Gamma_j\,,\ j\in \NN_0\, .
\end{equation}
Then, by Lemma \ref{unperturbedspectrum} we know that the eigenvalues of $\cM_{\nu,\eps}$ are encircled within the loop $\nu\Gamma_j$. Recalling the definition of the Riesz projections in \cref{RieszIntro}, we set $\Gamma=\nu\Gamma_j$, where $\Gamma_j$ is the loop in \cref{Gammaj}, and, after the change of variables $\lambda=\nu \zeta$, we introduce the auxiliary Riesz projection 
\begin{equation}\label{auxiliaryRieszproj}
    Q_{\nu,\e}^{(j)} \coloneqq -\frac{\nu}{2\pi\im} \oint_{\Gamma_j} (\cM_{\nu,\e} -\nu \zeta)^{-1} \drm \zeta \, ,\quad [ Q_{\nu,\e}^{(j)}]^2 =  Q_{\nu,\e}^{(j)}\, .
\end{equation}
Standard properties of the Riesz projection (see e.g.\ \cite{RieszNagy}) ensure that
\begin{enumerate}
    \item $\mathrm{Ran}\, Q_{\nu,\e}^{(j)}$ is the eigenspace of $\cM_{\nu,\e}$ associated with $-\nu(j^2+\e^2)$;
    \item $\mathrm{Ker}\, Q_{\nu,\e}^{(j)}$ is the direct sum of all the remaining eigenspaces of $\cM_{\nu,\e}$.
\end{enumerate}
We aim to extend the above spectral picture to the operator $\cL_{\nu,\e}=\cM_{\nu,\e} - \nu \zeta -\im \e \cR_\e $. We anticipate that $\nu\Gamma_j$ will also be contained in the resolvent set $\rho_{L^2}(\cL_{\nu,\eps})$ and therefore we define the Riesz projection 
\begin{equation}\label{Rieszproj}
    P_{\nu,\e}^{(j)} \coloneqq -\frac{\nu}{2\pi\im} \oint_{\Gamma_j} (\cL_{\nu,\e} -\nu\zeta)^{-1} \drm \zeta \;:\; L^2(\TT) \to L^2(\TT)\quad \text{and}\quad \cV^{(j)}_{\nu,\e} \coloneqq \mathrm{Ran}\,P_{\nu,\e}^{(j)}\,,
\end{equation}
where $[P_{\nu,\e}^{(j)}]^2 =P_{\nu,\e}^{(j)}$ and $\cV^{(j)}_{\nu,\e}$ is an invariant set 
of the operator $\cL_{\nu,\e}$.  We also
 observe that, for every $\zeta \in \Gamma_j$, we have that $\nu \zeta \in \rho(\cM_{\nu,\e})$ and so, by Lemma \ref{lem:ResolventcM}, $\cM_{\nu,\e} - \nu \zeta$ is invertible with bounded inverse. 
The resolvent operator of $\cL_{\nu,\e}$ is then given by
\begin{subequations}\label{resolventcL}
\begin{align}\label{resolventcL1}
 (\cL_{\nu,\e} -\nu\zeta)^{-1}    &= \big( \uno - \im \e (\cM_{\nu,\e} - \nu \zeta)^{-1}  \cR_\e \big)^{-1} (\cM_{\nu,\e} -\nu\zeta)^{-1} \\\label{resolventcL2}
    &=  \sum_{p=0}^{\infty} \big( \im \e  (\cM_{\nu,\e} -\nu\zeta)^{-1} \cR_\e\big)^{p}  (\cM_{\nu,\e} -\nu\zeta)^{-1} \, .
\end{align}
\end{subequations}
provided one can invert the operator $ \uno - \im \e (\cM_{\nu,\e} - \nu \zeta)^{-1}\cR_{\eps}$ and compute its Neumann expansion. We justify this in the following.

\begin{lemma}[Bounds for the resolvent of $\cL_{\nu,\e} $]
\label{lem:inverseMR}
   Let $j\in \NN_0$, $ \zeta \in \Gamma_j$,  $\e,\nu>0$ be such that
   \begin{equation}\label{smallcond}
       \frac{\e}{\nu}\|U\|_{H^2} < \frac{1}{4}
   \end{equation}
   Then, 
    \begin{equation}\label{uniformcondition}
        \| \im \e (\cM_{\nu,\e} -\nu\zeta)^{-1} \cR_\e\|_{\cB(L^2,L^2)} < 3  \frac{\e}{\nu}\|U\|_{C^2} .
    \end{equation}
\end{lemma}
The previous Lemma with formula \eqref{resolventcL1} show that, for $j$, $ \zeta $,  $\e,\nu>0$ as in the Lemma,  $\cL_{\nu,\e} -\nu\zeta$ is invertible with bounded inverse, whose operator norm can be controlled using Lemma \ref{lem:ResolventcM} and
\begin{equation}\label{Neumannestimate}
    \|(  \uno-\im \e (\cM_{\nu,\e} -\nu\zeta)^{-1} \cR_\e)^{-1} \|_{\cB(L^2,L^2)} \leq \frac{1}{1-3 \frac{\e}{\nu}\|U\|_{C^2}} \leq 4\, ,
\end{equation}
thus justifying \cref{resolventcL1,resolventcL2}. 
\begin{proof}[ Proof of Lemma \ref{lem:inverseMR}]
    One has $\zeta \in \rho_{L^2}(\cD_\eps)$ and, by Lemma \ref{lem:ResolventcM},
\begin{equation}\label{auxiteration1}
        \im \e (\cM_{\nu,\e} -\nu\zeta)^{-1} \cR_\e = \frac{\im \e}{\nu} (\cD_\e -\zeta)^{-1} \cR_\e  +  \frac{u_{\e}(\zeta,y)}{\nu^2(\e^2+\zeta)}\Pi_0 \cR_\e\,,
    \end{equation}
To estimate the operator norm of the right hand side,  we first observe that   $    \| {u_{\e}(\zeta,y)}
     \Pi_0 \cR_\e f \|_{L^2} =  \| {u_{\e}(\zeta,y)}\|_{L^2}
     |\Pi_0 \cR_\e f|$, for all $f \in L^2(\TT)$ .
By \cref{estimateFourierMultiplier,boundcR,estimatePi0cR,cMresolvent}, we get
\begin{equation}\label{compositionestimateaux}
    \|\im \e (\cM_{\nu,\e} -\nu\zeta)^{-1} \cR_\e \|_{\cB(L^2,L^2)} \leq \frac{2\e}{\nu} (\|U\|_{L^\infty} + \|U''\|_{L^\infty})  + \frac{2\e^2}{\nu^2} \|u_\e \|_{L^2} \frac{\| (\de_y+\e)^{-1} U \|^2_{L^2}}{\|U\|_{L^2}}\,,
\end{equation}
By \cref{estimateFourierMultiplier},
 $\| u_{\e}(\zeta,\cdot)\|_{L^2}    = \|\big(\cD_{\eps}-\zeta\big)^{-1}  U''(y)\|_{L^2}    \leq  
 2\| U''\|_{L^2}\, $.
By denoting as $U_j$ the Fourier coefficients of $U$, so that $U=\sum_{j \neq 0} U_j e^{\im j y}$ because of  \cref{hyp:U}, we have
\begin{equation}\label{worsenestimate}
 (\de_y+\e)^{-1} U = {\sum}_{j\neq 0}  \frac{U_j}{\im j + \e } e^{\im j y}\quad \text{and} \quad \| (\de_y+\e)^{-1} U \|_{L^2} \leq \|U\|_{L^2}\, .
\end{equation}
Hence we obtain
\begin{align}\label{compositionestimate} 
\| \im \e (\cM_{\nu,\e} -\nu\zeta)^{-1} \cR_\e\|_{\cB(L^2,L^2)} &\leq \frac{2\e}{\nu} \Big( \|U\|_{C^2}+ \frac{2\e}{\nu} \|U\|_{L^2}\|U''\|_{L^2} \Big) \, .
\end{align}
Under the assumption \cref{smallcond}, the factor in parenthesis in the right-hand side is bounded by $(3/2)\|U\|_{C^2}$, thus proving \cref{uniformcondition}.
\end{proof}

Thanks to condition \cref{uniformcondition}, we can safely  define the Riezs projection in \cref{Rieszproj} for every $j\in \NN_0$. As explained in the introduction, with the standard Kato's approach one would like to prove that $P_{\nu,\e}^{(j)}$ is sufficiently close to $Q_{\nu,\e}^{(j)}$. Unfortunately, this is not the case in our problem. Instead, we observe that the new projection is close to a rank-$1$ update of the projection $Q_{\nu,\e}^{(j)}$ in \cref{auxiliaryRieszproj}, as we show in the lemma below. 
\begin{lemma}\label{lem:P-Q}  There exists $\delta_U >0$ and $C>1$, depending on $\|U \|_{H^2}$ but independent of $\eps,\nu,j$, such that, if $\nu^{-1}\e<\delta_U\,$, $j\in\NN_0$,   then the Riezs projections $Q_{\nu,\e}^{(j)}$ and $ P_{\nu,\e}^{(j)}$ in \cref{auxiliaryRieszproj} and \cref{Rieszproj} satisfy
\begin{equation}\label{Katodistance}
    P_{\nu,\e}^{(j)} - Q_{\nu,\e}^{(j)} = \frac{1}{\nu^{2}} \big(w^{(j)}_{\e}(y) +\frac{\eps}{\nu} f^{(j)}_{\nu,\eps}(y)
    \big)\Pi_0+ \frac{\e}{\nu}\cF_{\nu,\eps}^{(j)}
\end{equation}
where
the function $w^{(j)}_{\e}(y)$ is given by
\begin{equation}\label{def:wj}
     w^{(j)}_{\e}(y)  \coloneqq -\frac{1}{2\pi\im} \oint_{\Gamma_j} \frac{(\cD_\e-\zeta)^{-1}\Pi_{\neq}\cR_\e}{\e^2+\zeta} u_\e(\zeta,y) \drm \zeta\, ,
\end{equation}
with $u_\e$ defined in Lemma \ref{lem:inverseMR}. The function $f_{\nu,\eps}^{(j)}$ and the linear operator $\cF_{\nu,\eps}^{(j)}:L^2(\TT) \to L^2(\TT)$ are bounded independently of $j,\eps,\nu$ and satisfy
\begin{align}
\label{bd:Pi0F1F2}
    \|\cF_{\nu,\eps}^{(j)}\|_{\cB(L^2,L^2)}+ \frac{1}{\eps \nu}\|\Pi_0\cF_{\nu,\eps}^{(j)}\|_{\cB(L^2,\RR)}+  \|f_{\nu,\eps}^{(j)}\|_{L^2}+\frac{1}{ \eps\nu}|\Pi_0f_{\nu,\eps}^{(j)}|\leq C ,
\end{align}
\end{lemma}
\begin{remark}
    We stress that the function $w_{\eps}^{(j)}$ is not small in general. Indeed, by Lemma \ref{lem:ResolventcM} observe that we can rewrite 
    \begin{align}
    w_\e^{(j)} &=-\frac{1}{2\pi\im} \oint_{\Gamma_j} \frac{(\cD_\e-\zeta)^{-1}\Pi_{\neq} \cR_\e}{\e^2+\zeta} U(y) \drm \zeta -\frac{1}{2\pi\im} \oint_{\Gamma_j} {(\cD_\e-\zeta)^{-1}\Pi_{\neq} \cR_\e} (\cD_\e-\zeta)^{-1} U(y) \drm \zeta \\
    &=: a^{(j)}_\e(y)+b^{(j)}_\e(y) \, . 
\end{align}
Taking the $\ell$-th Fourier coefficient of $a^{(j)}_{\eps}$ we have\begin{align}
a^{(j)}_{\e,\ell} =    \mathbbm{1}_{\{\ell\neq 0\}}[\cR_\e U]_\ell    \frac{1}{2\pi\im}\oint_{\Gamma_j} \frac{\drm \zeta}{(\e^2+\zeta)(\ell^2+\e^2+\zeta)} =  \frac{[\cR_\e U]_\ell}{\ell^2}\big(\mathbbm{1}_{\{j=0\}\cap \{\ell\neq 0\}}-\mathbbm{1}_{\{\ell=\pm j\}\cap \{\ell\neq 0\}}\big)
\end{align}
where in the last identity we used the residue theorem and the definition of $\Gamma_j$.
The $\ell$-th Fourier coefficient of $b^{(j)}_{\eps}$ is instead given by 
\begin{align}
    b_{\eps,\ell}^{(j)}=-\mathbbm{1}_{\{\ell\neq 0\}}\sum_{n\neq 0}  \cR_{\e,\ell-n} U_n    \frac{1}{2\pi\im} \oint_{\Gamma_j} \frac{\drm \zeta}{(\ell^2+\e^2+\zeta)(n^2+\eps^2+\zeta)} .
\end{align}
When $j=0$, since also $n\neq0$ by $U_0=0$, we can conclude that $b_{\eps}^{(0)}=0$ whereas $a_{\eps}^{(0)} =\partial_{yy}^{-1} \Pi_{\neq} \cR_\e U$ is a non-zero function that is not small in general. One can also deduce that there are no special cancellations between $a_{\eps}^{(j)}$ and $b_{\eps}^{(j)}$ for $j\neq 0$, meaning that indeed $\nu^{-2}w_{\eps}^{(j)}\Pi_0$ is a large rank-1 component in $P_{\nu,\eps}^{(j)}-Q_{\nu,\eps}^{(j)}$.
\end{remark}
\begin{proof}[Proof of Lemma \ref{lem:P-Q}] 
By the definitions \cref{auxiliaryRieszproj} and \cref{Rieszproj}, since $\cL_{\nu,\eps}=\cM_{\nu,\eps}-\im \eps \cR_{\eps}$, we observe that 
\begin{align}
    &P_{\nu,\e}^{(j)} - Q_{\nu,\e}^{(j)}=-\frac{\nu}{2\pi\im} \oint_{\Gamma_j} \Big((\cL_{\nu,\e} -\nu\zeta)^{-1} -(\cL_{\nu,\eps}-\nu\zeta)^{-1}(\cM_{\nu,\e} -\nu\zeta-\im \eps \cR_{\eps})(\cM_{\nu,\e} -\nu\zeta)^{-1} \Big)\drm \zeta \\
    &\;\,=-\frac{\im \nu \e}{2\pi\im} \oint_{\Gamma_j} (\cL_{\nu,\e} -\nu\zeta)^{-1} \cR_\e  (\cM_{\nu,\e} -\nu\zeta)^{-1} \drm \zeta.
\end{align}
Thus, by formulas \cref{resolventcL1} and \cref{cMresolvent} we find that
\begin{align}
   &P_{\nu,\e}^{(j)} - Q_{\nu,\e}^{(j)} \;\,\stackrel{\cref{resolventcL}}{=}-\frac{\im \nu \e}{2\pi\im} \oint_{\Gamma_j}  \big( \uno -  \im \e  (\cM_{\nu,\e} -\nu\zeta)^{-1} \cR_\e\big)^{-1}  (\cM_{\nu,\e} -\nu\zeta)^{-1}  \cR_\e  (\cM_{\nu,\e} -\nu\zeta)^{-1}\drm \zeta \\ \label{1and2}
&\stackrel{\cref{cMresolvent}}{=} -\frac{1}{2 \pi\im} \oint_{\Gamma_j}  \big( \uno -  \im \e  (\cM_{\nu,\e} -\nu\zeta)^{-1} \cR_\e\big)^{-1} ( \im \e(\cM_{\nu,\e} -\nu\zeta)^{-1}  \cR_\e)  (\cD_\e -\zeta)^{-1} \drm \zeta  \\
   &\qquad - \frac{1}{2\nu\pi\im} \oint_{\Gamma_j}  \frac{\big( \uno -  \im \e  (\cM_{\nu,\e} -\nu\zeta)^{-1}\cR_\e\big)^{-1}}{\e^2 + \zeta}   (\cM_{\nu,\e} -\nu\zeta)^{-1}  \cR_\e \circ u_{\e}(\zeta,y)  \Pi_0 \drm \zeta =: \cI_1 + \cI_2(y) \Pi_0\, .
\end{align}
By \cref{estimateFourierMultiplier,uniformcondition,Neumannestimate}, the term $\cI_1$ is bounded 
by
\begin{equation}\label{estimatefirstcirculation}
   \|\cI_1\|_{\cB(L^2,L^2)}
   \leq 2\frac{3 \frac{\eps}{\nu}\|U\|_{C^2}}{1-3 \frac{\eps}{\nu}\|U\|_{C^2}} <   9\frac{\e}{\nu}\|U\|_{C^2} \, ,\quad\text{provided }\ 3 \frac{\e}{\nu}\|U\|_{C^2} < \frac13\, ,
\end{equation}
where the last two inequalities are due to the inequality $\tfrac{2x}{1-x}<3x$ for $0\leq x< \tfrac13$. Moreover, by the Neumann expansion in \cref{resolventcL2} and using  $\|\Pi_0\cR_\eps\|_{\cB(L^2,\RR)}\leq \eps^2\|U\|^2_{L^2}$ from Lemma \ref{lem:fundamentalcancelation}, arguing as before we deduce that 
\begin{align}
\|\Pi_0\cI_1\|_{\cB(L^2,\RR)}\leq 9\frac{\eps^3}{\nu}\|U\|_{C^2}^3.
\end{align}
Hence, the operator
\begin{equation}\label{cF1}
\frac{\eps}{\nu}\mathcal{F}_{\nu,\eps}^{(j)}\coloneqq\mathcal{I}_1
\end{equation}
is bounded independently of $j$, $\e$ or $\nu$ with bounds as in \cref{bd:Pi0F1F2}.

On the other hand, the term $\cI_2(y)\Pi_0$ in \cref{1and2} is a rank-$1$ operator whose magnitude does not decrease for $\e$ small.
Let us inspect its components. By \cref{resolventcL2,uniformcondition}, we can use the Neumann series to isolate the perturbative part in the operator below as 
\begin{equation}\label{firstpiece}
   \big(\uno -  \im \e  (\cM_{\nu,\e} -\nu\zeta)^{-1} \cR_\e\big)^{-1}\eqqcolon\uno+\frac{\e}{\nu} \cA_1\,,
\end{equation}
for an operator $\cA_1$ that, by the arguments analogous to the ones used in \cref{estimatefirstcirculation}, is estimated by
\begin{equation}\label{firstpiecebd}
    \|\cA_1\|_{\cB(L^2,L^2)} <  9\|U\|_{C^2}, \quad \|\Pi_0\cA_1\|_{\cB(L^2,\RR)}< 9\eps^2\|U\|_{C^2}^3\,,\quad \text{provided }\    \frac{\e}{\nu}\|U\|_{C^2} < \frac19\, .
\end{equation}
Analogously, by \cref{cMresolvent} we split
\begin{equation}
\label{secondpiece}
    (\cM_{\nu,\e} -\nu\zeta)^{-1}  \cR_\e  =\frac1\nu \big(\big(\cD_{\eps}-\zeta\big)^{-1}\cR_\e   +\frac{\eps}{\nu}\cA_2\big)\,,\quad \text{where}\quad  \cA_2\coloneqq \frac{ u_\e(\zeta,y)}{\im (\e^2 + \zeta)}  \frac{1}{\e^2}\Pi_0\cR_\e\,.
\end{equation}
By \cref{estimatePi0cR} we know that $\|\e^{-2}\Pi_0\cR_\e\|_{\cB(L^2,\RR)} \leq \|U\|_{L^2} $. Hence, in view of \cref{estimateFourierMultiplier}, the definition of $u_{\eps}(\zeta,y)$ in Lemma \ref{lem:ResolventcM} and the fact that $\zeta \in \Gamma_j$, we get 
\begin{equation}\label{secondpiecebd}
     \| \cA_2\|_{\cB(L^2,L^2)}  \leq 4   \|U\|_{L^2}^2, \qquad \Pi_0\cA_2=0,
\end{equation}
where the last identity follows since $\Pi_0u_\eps(\zeta,y)=0$.
Therefore, using the splitting \cref{firstpiece,secondpiece},  we can write 
\begin{align}
\cI_2 &=- \frac{1}{2\nu^2\pi\im} \oint_{\Gamma_j}  \frac{\uno+\frac{\eps}{\nu}\cA_1}{\e^2 + \zeta}  \big((\cD_\eps-\zeta)^{-1} \cR_\e +\frac{\eps}{\nu}\cA_2\big)\circ u_{\e}(\zeta,y)  \drm \zeta \\ \label{cI2finale}
&=\frac{1}{\nu^2}\big(\widetilde{w}_{\eps}^{(j)}(y)+\frac{\eps}{\nu}\cI_3(y)\big)\,,
\end{align}
where we define
\begin{align}\label{def:wtildej}
     \widetilde{w}^{(j)}_{\e}(y)  \coloneqq -\frac{1}{2\pi\im} \oint_{\Gamma_j} \frac{(\cD_\e-\zeta)^{-1}\cR_\e}{\e^2+\zeta} u_\e(\zeta,y) \drm \zeta\, &=w^{(j)}_{\e}(y)-\frac{1}{2\pi\im} \oint_{\Gamma_j} \frac{(\cD_\e-\zeta)^{-1}\Pi_0\cR_\e}{\e^2+\zeta} u_\e(\zeta,y) \drm \zeta\,\\
     &\coloneqq w^{(j)}_{\e}(y)+\cI_4(y),
\end{align}
whereas 
$\cI^{(j)}_3$ contains all the terms where $\cA_1$ or $\cA_2$ appear at least one time. By
\cref{firstpiecebd,secondpiecebd} we get
\begin{equation}\label{estimatesecondcirculation}
   \| \cI_3\|_{L^2} +\eps^{-2}|\Pi_0\cI_3|\leq  C\,,
\end{equation}
for some constant $C>0$ depending on $U$ but independent of $j,\eps,\nu\in \NN_0$. 
To bound $\cI_4$ we note that
\begin{align}
    \cI_4 = &-\frac{1}{2\pi\im} \oint_{\Gamma_j} \frac{\Pi_0 \cR_\e}{(\e^2+\zeta)^2} u_\e(\zeta,y) \drm \zeta \\
    \stackrel{\cref{Cancelation}}{=} &-\frac{\e^2 }{2\pi \im} \oint_{\Gamma_j} \frac{1}{(\e^2+\zeta)^2}  \langle (\de_y +\e)^{-1} U, (\de_y +\e)^{-1}u_\e(\zeta,y) \rangle \drm \zeta \\
    \stackrel{\text{Lemma } \ref{lem:ResolventcM}}{=} &-\frac{\e^2 }{2\pi \im} \oint_{\Gamma_j} \frac{1}{(\e^2+\zeta)^2}  \langle (\de_y +\e)^{-1} U, (\de_y +\e)^{-1} (\cD_\e - \zeta)^{-1} U''\rangle \drm \zeta \\ \label{computeresidue}
    {=} &-\e^2 \sum_{\ell\neq 0} \frac{\ell^2 |U_\ell|^2}{\ell^2+\e^2} \Big( \frac{1}{2\pi\im} \oint_{\Gamma_j} \frac{1}{(\e^2+\zeta)^2 (\ell^2 + \e^2 +\zeta)} \drm \zeta \Big) \,,
\end{align}
where the last formula holds for $j\neq 0$.
By the residue theorem we get 
        \begin{equation}
    \label{averagewjpf}
\cI_4 =\begin{cases}
    \dfrac{\e^2}{(\e^2+j^2)} (|U_j|^2+|U_{-j}|^2) &\textup{if }j\in \mathbb{N}_0\,,\\[2mm]
    \e^2 \| (\de_{y}+\e)^{-1} \de_{y}^{-1} U\|_{L^2}^2 &\textup{if }j=0\, .
    \end{cases}
    \end{equation}
In any case, we can conclude that $\|\cI_4\|_{L^2}\leq C \eps^2$ and therefore we can define 
\begin{equation}
\label{cF2}
 f_{\nu,\eps}^{(j)}=\cI_3+\frac{\nu}{\eps}\cI_4
\end{equation}
Since $\cI_4=\Pi_0\cI_4$, by the bound on $\cI_4$ and \cref{estimatesecondcirculation} we see that $f_{\nu,\eps}^{(j)}$ satisfy bounds in agreement with \cref{bd:Pi0F1F2}. Finally,
\cref{1and2,estimatefirstcirculation,cF1,cI2finale,estimatesecondcirculation,cF2} justify identity \cref{Katodistance} and the bound \cref{bd:Pi0F1F2}.
\end{proof}

\subsection{Kato reduction}

We are finally ready to define, following Kato \cite{Kato}, 
the isomorphism at the core of this section. Since here we are only interested 
 in the correspondence between the ranges of the projections $Q_{\nu,\eps}^{(j)}$ and $P_{\nu,\eps}^{(j)}$,
we can simply define the desired mapping as 
\begin{equation}\label{Katomorphism}
    \cU_{\nu,\e}^{(j)} \coloneqq 
    \uno - P_{\nu,\e}^{(j)} - Q_{\nu,\e}^{(j)} \, ,
\end{equation}
which, since $(P_{\nu,\e}^{(j)})^2 = P_{\nu,\e}^{(j)}$ and $(Q_{\nu,\e}^{(j)})^2 = Q_{\nu,\e}^{(j)}$, maps 
\begin{equation}\label{RanRanKerKer}
 \mathrm{Ran}\, Q_{\nu,\e}^{(j)}\quad \stackrel{  {\scriptsize \begin{matrix}\cU_{\nu,\e}^{(j)} \end{matrix}} }{\longrightarrow}\quad  \cV_{\nu,\e}^{(j)} \coloneqq \mathrm{Ran}\, P_{\nu,\e}^{(j)}  \quad \stackrel{  {\scriptsize \begin{matrix}\cU_{\nu,\e}^{(j)} \end{matrix}} }{\longrightarrow} \quad \mathrm{Ran}\, Q_{\nu,\e}^{(j)}  \,. 
\end{equation}
To show that $\cU_{\nu,\e}^{(j)}$ is an isomorphism, we observe that
the inverse of $\cU_{\nu,\e}^{(j)}$ is given by
\begin{equation}
    [\cU_{\nu,\e}^{(j)}]^{-1} =\Big( \uno - \big(P_{\nu,\e}^{(j)} - Q_{\nu,\e}^{(j)}\big)^2 \Big)^{-1} \big(\, \uno - P_{\nu,\e}^{(j)} - Q_{\nu,\e}^{(j)}\;) \, ,
\end{equation}
provided $\uno - \big(P_{\nu,\e}^{(j)} - Q_{\nu,\e}^{(j)}\big)^2$ is invertible. We justify the last passage in the following
\begin{lemma}
\label{lem:KatoIso}
    Under the hypothesis of Lemma \ref{lem:P-Q},  the operator $\uno - \big(P_{\nu,\e}^{(j)} - Q_{\nu,\e}^{(j)}\big)^2 $, with $P_{\nu,\e}^{(j)}$ and $ Q_{\nu,\e}^{(j)}$ in \cref{Rieszproj,auxiliaryRieszproj}, is invertible and the operator $\cU_{\nu,\e}^{(j)}$ in \cref{Katomorphism} is an isomorphism.
\end{lemma}
\begin{proof}
Let us define the auxiliary operator
\begin{equation}\label{cT}
    \cT_{\nu,\e}^{(j)} \coloneqq P_{\nu,\e}^{(j)} - Q_{\nu,\e}^{(j)} -\frac1{\nu^2} (w^{(j)}_{\e}(y)+\frac{\eps}{\nu}f_{\nu,\eps}^{(j)}(y)) \Pi_0 \stackrel{\cref{Katodistance}}{=} \frac{\e}{\nu} \cF_{\nu,\eps}^{(j)}  ,
\end{equation}
with $\cF_{\nu,\eps}^{(j)}$ and $f_{\nu,\eps}^{(j)}$ given in Lemma \ref{lem:P-Q}, frow which we deduce the following bounds 
\begin{equation}\label{boundIpmcT}
   \| \cT^{(j)}_{\nu,\e} \|_{\cB(L^2,L^2)} \leq  C\frac{\e}{\nu} \eqqcolon \gamma \,, \qquad  \| \Pi_0\cT^{(j)}_{\nu,\e} \|_{\cB(L^2,\RR)}\leq C\eps^2,
\end{equation}
where $C>1$ is a constant independent of $\eps,\nu,j$.
If
$\gamma <1$, the operators 
$\uno \pm  \cT_{\nu,\e}^{(j)}$ are invertible with inverse given by a Neumann series and bounded by
\begin{equation}\label{boundIpmcTinv}
    \| (\uno \pm  \cT_{\nu,\e}^{(j)})^{-1} \|_{\cB(L^2,L^2)} \leq \frac{1}{1-\gamma}\,.
\end{equation}
Now, let us consider the operator 
\begin{equation}\label{prodottonotevole(A^2-B^2)}
    \uno - (P_{\nu,\e}^{(j)} - Q_{\nu,\e}^{(j)} )^2 = \big(  \uno - (P_{\nu,\e}^{(j)} - Q_{\nu,\e}^{(j)} )\big) \big(  \uno + (P_{\nu,\e}^{(j)} - Q_{\nu,\e}^{(j)} ) \big)\, .
\end{equation}
By \cref{cT}, the two factors are both rank-$1$ updates of invertible operators when written as follows
\begin{equation}
       \uno \pm (P_{\nu,\e}^{(j)} - Q_{\nu,\e}^{(j)} ) = \underbrace{\uno \pm \cT_{\nu,\e}^{(j)} }_{\text{invertible}}\pm \underbrace{\nu^{-2} W_\eps^{(j)} \Pi_0}_{\text{rank }1} \, , \qquad W_{\eps}^{(j)}(y)=w_\e^{(j)}(y)+\eps\nu^{-1}f_{\nu,\eps}^{(j)}(y).
\end{equation}
We then apply  Lemma \ref{lem:ShermanMorrisoninf}, with 
\begin{equation}
    H=L^2(\TT)\,,\quad \cA = \uno \pm \cT_{\nu,\e}^{(j)}\, ,\quad f= \pm \nu^{-2} W_\e^{(j)}\quad\text{and}\quad g=1\ \textup{ (so that }\langle g, \cdot \rangle = \Pi_0\textup{)}\, ,
\end{equation}
to obtain
\begin{equation}\label{tojustify}
    \big( \uno \pm (P_{\nu,\e}^{(j)} - Q_{\nu,\e}^{(j)} ) \big)^{-1} \stackrel{\cref{originalShermanMorrison}}{=}\big( \uno \pm \cT_{\nu,\e}^{(j)}\e) \big)^{-1} \mp \nu^{-2}\frac{ \big( \uno \pm \cT_{\nu,\e}^{(j)} \big)^{-1} W_{\e}^{(j)}(y) \Pi_0 \big( \uno \pm \cT_{\nu,\e}^{(j)} \big)^{-1}}{1\pm \nu^{-2} \Pi_0  \big( \uno \pm \cT_{\nu,\e}^{(j)} \big)^{-1} W_{\e}^{(j)}(y) }\, .
\end{equation}
To be able to apply  Lemma \ref{lem:ShermanMorrisoninf} we  need to check that $\big|\nu^{-2} \Pi_0  \big( \uno \pm \cT_{\nu,\e}^{(j)} \big)^{-1} W_{\e}^{(j)}(y) \big|\neq 1 $. We have
\begin{align}
    \Pi_0  \big( \uno \pm \cT_{\nu,\e}^{(j)} \big)^{-1} W_{\e}^{(j)} =\Pi_0W_{\eps}^{(j)}+\sum_{n=1}^\infty (\mp 1)^n\Pi_0 (\cT_{\nu,\eps}^{(j)})^nW_{\eps}^{(j)}.
\end{align}
Since $\Pi_0 w_{\eps}^{(j)}=0$ and exploiting the bounds \cref{boundIpmcT,bd:Pi0F1F2}, we deduce that 
\begin{align}
     |\Pi_0  \big( \uno \pm \cT_{\nu,\e}^{(j)} \big)^{-1} W_{\e}^{(j)}|\leq C\eps^2,
\end{align}
 provided $\eps \nu^{-1}$ is sufficiently small, which is guaranteed upon restricting further $\delta_U$ in Lemma \ref{lem:P-Q} if necessary. This justifies \cref{tojustify}.
Consequently the whole operator $\uno - \big(P_{\nu,\e}^{(j)} - Q_{\nu,\e}^{(j)}\big)^2$, as well as
the mapping $\cU_{\nu,\e}^{(j)}$ in \cref{Katomorphism}, is bounded with bounded inverse.
\end{proof}

We are ready to characterize the $\cL_{\nu,\e}$-invariant subspaces $\cV^{(j)}_{\nu,\e} $ defined in \cref{RanRanKerKer}, proving in particular  Theorem \ref{thm:linmain} except for the expansion \cref{unstableeigenvalue}.
\begin{lemma}
\label{lem:katoeigen}Let $\e,\nu>0$ be in the open regime where the operator $\cU_{\nu,\e}^{(j)}$ in \cref{Katomorphism} is an isomorphism. 
Let $j\in\NN_0$ and $\Gamma_j$ as in \cref{Gammaj}. 
Then, the $\cL_{\nu,\e}$-invariant subspace $\cV^{(j)}_{\nu,\e} \coloneqq \mathrm{Ran}\,P_{\nu,\e}^{(j)} $
is given by
\begin{equation}\label{Katobasis}
   \cV^{(j)}_{\nu,\e} = \begin{cases}
   \mathrm{span}\, \big\{ P_{\nu,\e}^{(0)}[ U(y) -\im \nu \e]\coloneqq V_{\nu,\eps}^{(0)} \big\} & \mbox{if }j= 0\,,\\[2mm]
   \mathrm{span}\, \big\{  P_{\nu,\e}^{(j)} e^{\im j y} \,,\; P_{\nu,\e}^{(j)} e^{-\im j y}\big\} & \mbox{if }j\neq 0\, ,
   \end{cases}
\end{equation}
and one has $L^2(\TT) = \bigoplus_{j\in \NN_0} \cV^{(j)}_{\nu,\e}$. Moreover, 
\begin{align}
    \label{bd:eigen}\|V_{\nu,\eps}^{(0)}(y) - (U(y)-\im \nu \eps) \|_{L^2}+\|P_{\nu,\e}^{(j)} e^{\pm\im j y} - e^{\pm\im j y} \|_{L^2} \leq C \frac{\e}{\nu}\, ,\quad  \forall j>0\,,
\end{align}
for a constant $C>0$ independent of $\eps,\nu,j$.
\end{lemma}

\begin{proof}
We observe that, by Lemmas \ref{unperturbedspectrum} and \ref{auxiliaryRieszproj},
\begin{equation}\label{trivialproj}
    Q_{\nu,\e}^{(0)} [ U(y) -\im \nu \e ] = U(y) -\im \nu \e \quad \text{and} \quad Q_{\nu,\e}^{(j)}e^{\pm \im j y} = e^{\pm \im j y} \,,\quad \forall j\in \NN\, .
\end{equation}
Then, for every $j \in \NN_0$, by \cref{RanRanKerKer} and since the operator $\cU_{\nu,\e}^{(j)}$ in \cref{Katomorphism} is an isomorphism,  
\begin{equation}
    \cV^{(j)}_{\nu,\e} = \mathrm{span}\, \big\{  \cU_{\nu,\e}^{(j)} e^{\im j y} \,,\; \cU_{\nu,\e}^{(j)} e^{-\im j y}\big\}\, ,\quad \cV^{(0)}_{\nu,\e} = \mathrm{span}\, \big\{ \cU_{\nu,\e}^{(0)}[ U(y) -\im \nu \e] \big\}  \,,
\end{equation}
where, by \cref{Katomorphism,trivialproj},
\begin{equation}\label{eqn:rankpj}
   \cU_{\nu,\e}^{(j)} e^{\pm \im j y} = -P_{\nu,\e}^{(j)} e^{\pm \im j y} \, ,\quad \cU_{\nu,\e}^{(0)}[\im U(y) + \nu \e] = -P_{\nu,\e}^{(0)}[ U(y) -\im \nu \e]\, ,
\end{equation}
from which we obtain \cref{Katobasis}. Finally, to prove  \cref{bd:eigen}, note that by \cref{Katodistance} and $\Pi_0U=0$, we have
\begin{align}
V_{\nu,\eps}^{(0)}(y) - (U(y)-\im \nu \eps) =(P_{\nu,\eps}^{(0)}-Q_{\nu,\eps}^{(0)})(U(y)-\im \nu \eps)\\
=-\im\frac{\eps}{\nu}\big(w_{\eps}^{0}(y)+\frac{\eps}{\nu}\cF_2^{(0)}(y)\big)+\frac{\eps}{\nu}\cF_1^{(0)}(U(y)-\im \nu \eps).
\end{align}
Hence, from the identity above and Lemma \ref{lem:P-Q}, we deduce a bound complying with \cref{bd:eigen}. Arguing analogously for the other eigenvectors, we prove \cref{bd:eigen} also for all the other $j>0$.

From \cref{eqn:rankpj}, for every $j \in \mathbb N$ we see that $P_{\nu,\e}^{(j)}$ has at least rank $2$, so that there are at least $2$ eigenvaules encircled by $\nu \Gamma_j$. In fact the isomorphism characterizes the spectrum, therefore there are exactly two eigenvalues (or possibly a double eigenvalue) in each circle, as claimed in \cref{restofthespectrum}. 
\end{proof}

Theorem \ref{thm:linmain} follows from Lemmas \ref{lem:KatoIso} and \ref{lem:katoeigen}  except for the expansion in \cref{unstableeigenvalue}, which is proved in the next section.

\smallskip

\subsection{The unstable eigenvalue}
\label{sec:unstable}
The goal of this section is to prove the remaining point in Theorem \ref{thm:linmain}, namely the expansion of the unstable eigenvalue $\lambda_{\nu,\e}^{(0)}$ of $\cL_{\nu,\e}$ stated in \cref{unstableeigenvalue} in Theorem \ref{thm:linmain}. In the sequel, since we only deal with the case $j=0$ we omit the superscripts $(0)$ from the notation.

To find the expression of $\lambda_{\nu,\eps}$ we recall that, by Lemma \ref{lem:katoeigen} and $ V_{\nu,\e}
  =P_{\nu,\eps}
  (U(y)-\im \nu\eps) $
we have
\begin{equation}
    \label{eq:trivlambda}
\lambda_{\nu,\eps}V_{\nu,\eps}=\cL_{\nu,\eps}V_{\nu,\eps}=\cL_{\nu,\eps}P_{\nu,\eps}[U(y)-\im \nu \eps].
\end{equation} 
Taking the $L^2(\TT)$ inner product with $V_{\nu,\eps}$, we know that 
\begin{equation}
    \label{eq:trivlambda1}
\lambda_{\nu,\eps}=\frac{1}{\|V_{\nu,\eps}\|^2_{L^2}}\langle\cL_{\nu,\eps}V_{\nu,\eps},V_{\nu,\eps}\rangle.
\end{equation}
We would then need to expand the expression above. However, we will just expand $\cL_{\nu,\eps}V_{\nu,\eps}$ and then take the inner product with $V_{\nu,\eps}$, which gives slightly suboptimal estimates in the errors but simplifies the computations. 

Since the eigenvalue will scale with a factor $\eps^2$, we need to keep track of the expansions up to second order. To isolate the leading order operators, we find it convenient to introduce the following shorthand\smallskip

 \noindent{\bf Notations.}   Let $\ell \in \NN$ and $p_1,\dots,p_\ell,q_1,\dots,q_\ell\in \ZZ$, $\kappa_i\coloneqq\eps^{p_i}\nu^{q_i}$, $i=1,\dots,\ell$. Let $f\in L^2(\TT)$  and $\cT:L^2(\TT)\to L^2(\TT)$ be a bounded operator.
    We say that $\cT=\cO_{\rm op}(\kappa_1+\dots+\kappa_\ell)$ if $$\|\cT\|_{\cB(L^2,L^2)}\leq C (\kappa_1 + \dots + \kappa_\ell)$$ for an absolute $C>0$ independent of $\eps$ and $\nu$. 
    We write $f=\cO_{\rm fun}(\kappa_1+\dots+ \kappa_\ell)$ if $$\|f\|_{L^2}\leq C(\kappa_1+\dots+\kappa_\ell)$$
    for an absolute $C>0$ independent of $\eps$ and $\nu$.
    \smallskip

With the above notation, we deduce from Lemma \ref{lem:katoeigen} that
\begin{equation}
    \label{exp:Vnueps0}
  V_{\nu,\e}
  =P_{\nu,\eps}
  (U(y)-\im \nu\eps) = (U(y)-\im \nu\eps)+\cO_{\rm fun}\big(\frac{\eps}{\nu}\big).
\end{equation}
Since  the circulation of the identity vanishes, our analysis starts with the 
observation
\begin{align}\label{circulateuno}
  \cL_{\nu,\e} P_{\nu,\e} &=  -\frac{\nu}{2\pi\im} \oint_{\Gamma_0}  (\cL_{\nu,\e} -\nu\zeta + \nu\zeta) (\cL_{\nu,\e} -\nu\zeta)^{-1} \drm \zeta  \\
     &= -\frac{\nu}{2\pi\im} \oint_{\Gamma_0}  \nu\zeta (\cL_{\nu,\e} -\nu\zeta)^{-1} \drm \zeta\\
     &=-\frac{\nu}{2\pi\im} \oint_{\Gamma_{0}}  \nu\zeta \big( \uno -  \im \e  (\cM_{\nu,\e} -\nu\zeta)^{-1} \cR_\e\big)^{-1}  (\cM_{\nu,\e} -\nu\zeta)^{-1} 
     \drm \zeta \,, 
\end{align}
where in the last step we used that $\nu\zeta$ lies in the resolvent set of $\cM_{\nu,\e}$ by construction and the operator $\uno -  \im \e  (\cM_{\nu,\e} -\nu\zeta)^{-1} \cR_\e$ is invertible by Lemma \ref{lem:inverseMR}.  By Lemma \ref{unperturbedspectrum}, we know that 
\begin{equation}
\label{eigenpairsunplem0}
    (\cM_{\nu,\e} -\nu\zeta)^{-1} [ U(y) -\im \nu \e]=-\frac{1}{\nu(\eps^2+\zeta)}[ U(y) -\im \nu \e].
\end{equation}
Therefore, exploiting the Neumann series representation in \cref{resolventcL2} up to order $2$,   using Lemmas \ref{lem:inverseMR} and \ref{lem:ResolventcM}, and combining \cref{eigenpairsunplem0} with the bound in \cref{uniformcondition}  we get
\begin{align}
     &\cL_{\nu,\e} V_{\nu,\e}  \\
&\;\,\stackrel{\cref{eigenpairsunplem0}}{=}\,  -\frac{\nu}{2\pi\im} \oint_{\Gamma_0} \frac{ -\zeta}{\e^2+\zeta} \big( \uno -  \im \e  (\cM_{\nu,\e} -\nu\zeta)^{-1} \cR_\e\big)^{-1}   [U(y) -\im \nu \e] \drm \zeta \\
    \label{exp:CL} &\stackrel{\cref{resolventcL2}}{=} \frac{\nu}{2\pi\im} \oint_{\Gamma_0} \frac{ \zeta}{\e^2+\zeta}   [ U(y) -\im\nu \e] \drm \zeta \\
     &\qquad+\;\frac{\nu}{2\pi\im} \oint_{\Gamma_0} \frac{ \zeta}{\e^2+\zeta} \Big(\im \e  \big(\nu\cD_{\eps}-\nu\zeta\big)^{-1} \cR_\e  +\frac{u_{\e}(\zeta,y)}{\nu^2(\e^2+\zeta)}\Pi_0 \cR_\e \Big)   [U(y) -\im \nu \e] \drm \zeta \\
     &\qquad+\;\frac{\nu}{2\pi\im} \oint_{\Gamma_0} \frac{ \zeta}{\e^2+\zeta} \Big(\im \e  \big(\nu\cD_{\eps}-\nu\zeta\big)^{-1} \cR_\e  +\frac{u_{\e}(\zeta,y)}{\nu^2(\e^2+\zeta)}\Pi_0 \cR_\e \Big)^2   [ U(y)-\im \nu \e] \drm \zeta  +  \cO_{\mathrm{fun}}\big(\frac{\e^3}{\nu^2}\big)\\
     &\quad \coloneqq v_1+v_2+v_3+\cO_{\mathrm{fun}}\big(\frac{\e^3}{\nu^2}\big)
     \end{align}
     Within the expansion above, we still have terms of order  $\eps^3$ that are to be absorbed in the error term. Indeed, by recalling that the loop $\Gamma_0$ encircles a portion of the complex plane containing $-\e^2$, we apply the residue theorem to the circulation of a simple pole that gives $v_1$ to obtain
     \begin{equation}
       v_1 = \frac{\nu}{2\pi\im} \oint_{\Gamma_0} \frac{ \zeta}{\e^2+\zeta}   [ U(y) -\im\nu \e] \drm \zeta  = -\nu\e^2 \big(U(y) -\im\nu \e\big) \, .
     \end{equation}
     In view of \cref{exp:Vnueps0}, we conclude that
     \begin{equation}
     \label{err:v1sub}
         v_1=-\nu\eps^2V_{\nu,\eps}+\cO_{\rm fun}(\eps^3).
     \end{equation}
      $v_2$ contains the leading order correction to the eigenvalue, arising from the part exhibiting $\Pi_0\cR_\eps$. Here, observe that by Lemma \ref{lem:fundamentalcancelation} we have $\Pi_0\cR_{\eps}1=0$. Thus, we isolate the leading order term in $v_2$ and rewrite $v_2= \frac{\eps^2}{\nu} V_{\rm lead} + \eps^3 V_{\rm err}^1$ where
      \begin{align}
      \label{def:Vlead}V_{\rm lead}&\coloneqq \frac{1}{2\pi\eps^2 \im}\oint_{\Gamma_0}  \frac{\zeta u_{\e}(\zeta,y)}{(\e^2+\zeta)^2} \Pi_0 \cR_\e [U(y)]\dd \zeta\\
        V_{\rm err}^1&=\frac{\nu}{2\pi\eps^2\im}\oint_{\Gamma_0} \frac{\im\zeta}{\e^2 +\zeta}  \big(\nu\cD_{\eps}-\nu\zeta\big)^{-1} \cR_\e [U(y) -\im  \nu \e]\dd \zeta.
    \end{align}
     For $v_3$, by Lemmas \ref{lem:fundamentalcancelation} and \ref{lem:ResolventcM}, we know that  $u_\eps(\lambda,y)\Pi_0\cR_{\eps}=\cO_{\rm op}(\eps^2)$. On the other hand, the terms containing $(\nu\cD_{\eps}-\lambda)^{-1}$ can contribute to the order of the pole. Hence, for the moment we can only consider as errors the terms containing at least one $\Pi_0\cR_{\eps}$, and we get
     \begin{align}
         v_3=\frac{\eps^2\nu}{2\pi \im}\oint_{\Gamma_0}\frac{-\zeta}{\eps^2+\zeta}(\nu\cD_{\eps}-\nu \zeta)^{-1}\cR_{\eps}(\nu\cD_{\eps}-\nu \zeta)^{-1}\cR_{\eps}[U(y)-\im \nu \eps]\dd \zeta+\cO_{\rm fun}\big(\frac{\eps^3}{\nu},\frac{\eps^4}{\nu^3}\big).
     \end{align}
     We denote 
\begin{align}
    \label{def:Verr}& V_{\rm err}^2=-\frac{\nu}{2\pi\eps\im}\oint_{\Gamma_0} \frac{\zeta}{\e^2 +\zeta} \Big(\big(\nu\cD_{\eps}-\nu\zeta\big)^{-1} \cR_\e\Big)^2 [U(y) -\im  \nu \e]\dd \zeta.
\end{align}  Then, on account of the observations above and since $\eps\nu^{-1}<1$, we rewrite 
\begin{align}
    \label{contolinearefinale}
     \cL_{\nu,\eps}V_{\nu,\eps}&= -\nu\e^2V_{\nu,\eps}+\frac{\eps^2}{\nu} V_{\rm lead} + \eps^3 (V_{\rm err}^1+ V_{\rm err}^2)+\cO_{\rm fun}\big(\eps^3(1+\frac{1}{\nu^2})\big) 
\end{align}
We next claim an expansion of $V_{\rm lead}$, containing the desired expansion for $\lambda_{\nu,\eps}$, and a bound on $V_{\rm err}$ as
\begin{equation}\label{residuoimportante}
   V_{\rm lead} = \| \de_y^{-1} U \|_{L^2}^2V_{\nu,\eps}+ \cO_{\rm fun}\big(\eps(1+\frac{1}{\nu})\big)\, ,
\end{equation}
  \begin{equation}
\label{bd:Verr}
V_{\rm err}^1+V_{\rm err}^2=\cO_{\rm fun}(1),
    \end{equation}
    The proof is postponed, since the bound on $V_{\rm err}^1$ and $V_{\rm err}^2$ does not directly follow from a direct inspection of the operators involved, but it requires a more delicate analysis.

 Appealing to \cref{eq:trivlambda,contolinearefinale,residuoimportante,bd:Verr}, we know that 
\begin{align}
\lambda_{\nu,\eps}V_{\nu,\eps}=\frac{\eps^2}{\nu}\big(\|\de_{y}^{-1}U\|^2_{L^2}-\nu^2\big)V_{\nu,\eps}+\frac{\eps^2}{\nu}\cO_{\rm fun}\big(\eps(1+\frac{1}{\nu})\big)+\cO_{\rm fun}\big(\eps^3(1+\frac{1}{\nu^2})\big).
\end{align}
By taking the $L^2$-inner product with $V_{\nu,\eps}/\|V_{\nu,\eps}\|_{L^2}^2$, we deduce that
\begin{equation}
   \lambda_{\nu,\e} = \frac{\eps^2}{\nu}\Big(\|\de_{y}^{-1}U\|^2_{L^2}-\nu^2+\cO\big(\eps(\nu+\nu^{-1})\big)\Big)
\end{equation}
which proves \cref{unstableeigenvalue}. This concludes the proof of Theorem \ref{thm:linmain}, provided the two claims are proved.
\begin{remark}
    The error of size $\cO(\eps\nu)$ is related to the fact that we are directly approximating $V_{\nu,\eps}$ with $U(y)-\im \nu \eps$, as done for instance in \cref{err:v1sub}. We anticipate that with the normal forms technique we know that this error is in fact an error of size $\cO(\eps^2)$. To obtain this result also here, it is enough to expand the inner product determining $\lambda_{\nu,\eps}$, keeping in mind that  $V_{\nu,\eps}$ satisfies an expansion as in \cref{exp:CL} without the factor $\zeta$ in the numerator. 
\end{remark}

\begin{proof}[Proof of \cref{residuoimportante}]
From Lemma \ref{lem:fundamentalcancelation}, we know that 
\begin{equation}
    \frac{1}{\eps^2}\Pi_0\cR_\eps[U(y)]=\|(\de_y+\eps)^{-1}U\|^2_{L^2}.
\end{equation}
Thus, by the properties of $u_\eps$ in Lemma \ref{lem:ResolventcM}, we get
\begin{align}
V_{\rm lead} &=   \|(\de_y+\e)^{-1} U\|^2_{L^2}\frac{1}{2\pi \im}\oint_{\Gamma_0}  \frac{(\zeta + \e^2 - \e^2) u_{\e}(\zeta,y)}{(\e^2+\zeta)^2} \drm \zeta \\
&=    \|(\de_y+\e)^{-1} U\|^2_{L^2}\frac{1}{2\pi\im}\oint_{\Gamma_0}  \frac{u_{\e}(\zeta,y)}{\e^2+\zeta} \drm \zeta + \cO_{\rm fun}(\e^2) \\
&=  \|(\de_y+\e)^{-1} U\|^2_{L^2}u_{\e}(-\e^2,y)  + \cO_{\rm fun}(\e^2)\, ,
\end{align}
where in the last step we used the residue theorem. By  the definition of $u_\eps$ in Lemma \ref{lem:ResolventcM},  we see that 
\begin{equation}
u_{\e}(-\e^2,y) =  U(y) \, .
\end{equation}
Finally, expanding also $\|(\de_y+\e)^{-1}U\|^2_{L^2}=\|\de_y^{-1}U\|^2_{L^2}+\cO(\eps^2) $ and using \cref{exp:Vnueps0}, we get 
\cref{residuoimportante}.
\end{proof}
\begin{proof}[First part of the proof of \cref{bd:Verr}, $V_{\rm err}^1=\cO_{\rm fun}(1)$]
    From the definition of $\cR_\eps$ in \cref{cMcR}, note that 
    \begin{align}
        \cR_{\eps}[U(y)-\im \nu \eps]&=\Pi_{\neq}\big[U^2(y)+U''(-\de_{yy}+\eps^2)^{-1}U(y)-\im \nu \eps U(y)\big]+\Pi_0\cR_{\eps}[U(y)-\im \nu \eps]\\
        &\coloneqq \Pi_{\neq} f_U+\Pi_0\cR_\eps[U(y)-\im \nu \eps]
    \end{align}
    Since $(\nu\cD_\eps-\nu\zeta)^{-1}\Pi_0 g=-\nu^{-1}(\eps^2+\zeta)^{-1}\Pi_0 g$, we deduce that 
    \begin{align}
        V_{\rm err}^1=\frac{1}{2\pi\eps^2\im}\oint_{\Gamma_0} \frac{\im\zeta}{\e^2 +\zeta}  \big(\cD_{\eps}-\zeta\big)^{-1} \Pi_{\neq}f_U\dd \zeta+\frac{1}{2\pi \nu \im}\oint_{\Gamma_0} \frac{-\im\zeta}{(\e^2 +\zeta)^2} \frac{\Pi_0\cR_\eps}{\eps^2}[U(y)-\im \nu \eps] \dd \lambda.
    \end{align}
    For the first circulation, thanks to the presence of $\Pi_{\neq}$, observe that $(\cD_\eps-\zeta)^{-1}\Pi_{\neq}$ does not change the order of the pole. Therefore, we gain from the factor $\zeta$ in the numerator of the first circulation, resulting in a term of order $\cO_{\rm fun}(1/\nu).$ Instead, the second circulation has a pole of order $2$, meaning that we do not gain anything when computing the integral. However, thanks to Lemma \ref{lem:fundamentalcancelation}, we know that  $\Pi_0\cR_\eps=\cO_{\rm op}(\eps^2)$. We can then conclude that $V_{\rm err}^1=\cO_{\rm fun}(1).$
\end{proof}
\begin{proof}[Second part of the proof of \cref{bd:Verr}, $V_{\rm err}^2=\cO_{\rm fun}(1)$]
    For  $V_{\rm err}^2$
    we can proceed analogously (and we also gain a factor of $\eps$). More precisely, denoting $$\cA=(\nu\cD_{\eps}-\nu\zeta)^{-1}\cR_\eps,$$ one can split
    \begin{align}
\cA^2=(\Pi_{\neq}\cA)\circ(\Pi_{\neq}\cA)+(\Pi_{\neq}\cA)\circ(\Pi_{0}\cA)+(\Pi_{0}\cA)\circ(\Pi_{\neq}\cA)+(\Pi_{0}\cA)\circ(\Pi_{0}\cA).
    \end{align}
    Then, since $\Pi_0\cA=-(\nu \eps^2+\nu\zeta)^{-1}\Pi_0\cR_\eps$, we get 
    \begin{align}
        \cA^2=\,&(\Pi_{\neq}\cA)\circ(\Pi_{\neq}\cA)-\frac{1}{\nu(\eps^2+\zeta)}(\Pi_{\neq}\cA)\circ(\Pi_{0}\cR_{\eps})\\
        &-\frac{1}{\nu(\eps^2+\zeta)}(\Pi_{0}\cR_\eps)\circ(\Pi_{\neq}\cA)+\frac{1}{\nu^2(\eps^2+\zeta)^2}(\Pi_{0}\cR_{\eps})\circ(\Pi_{0}\cR_{\eps}).
    \end{align}
When $\Pi_{\neq}$ occurs twice, we are not changing the order of the pole and therefore we gain a factor $\eps^2$ from the residue theorem. When $\Pi_0$ occurs once, we do not gain from the integral but we know that $\Pi_0\cR_{\eps}=\cO_{\rm op}(\eps^2)$ by Lemma \ref{lem:fundamentalcancelation}. The last term with two $\Pi_0$ has zero residue and therefore we conclude that 
    $V_{\rm err}^2=\cO_{\rm fun}(\eps(1+\frac{1}{\nu})).$
\end{proof}

\begin{remark}
\label{rem:Taylor} 
    To end this section, let us explain how to prove the Taylor dispersion mechanism \cite{taylor1953dispersion} introduced in Remark \ref{rem:Taylorintro}. Following the same reductions done to arrive at $\cL_{\nu,\eps}$, we know that one needs to study properties of the operator
    \begin{equation}
        \mathcal{T}_{\nu,\eps}=\nu\cD_{\eps}-\im \eps U(y).
    \end{equation}
        Using the (standard) Kato's approach, it is enough to follow the steps in this section with the changes
        \begin{equation}
            \cM_{\nu,\eps}\to \nu \cD_{\eps}, \qquad \cR_{\eps}\to U(y), \qquad U(y)-\im \nu \eps \to 1.
        \end{equation}
        Then, all the proofs greatly simplify\footnote{In the sense that $\im \eps U(y)$ is really a perturbation of $\nu\cD_{\eps}$ since $\Pi_0U=0$ and $\eps\ll \nu$} and one can compute the correction to the largest eigenvalue of $\nu\cD_{\eps}$, given by $-\nu\eps^2$, by following the computations in \cref{sec:unstable}. In particular, the leading order term in the correction of the eigenvalue will be given by the circulation of the order $n=2$ in the Neumann series for $(\uno-\im\eps (\nu\cD_{\eps}-\nu\zeta)^{-1}\circ U(y))^{-1}(1)$ (the analogue of $v_3$ defined in \cref{sec:unstable}). This is 
        \begin{align}
            &-\frac{\nu \eps^2}{2\pi \im }\oint_{\Gamma_0}\frac{\zeta}{\eps^2+\zeta}\Big((\nu \cD_{\eps}-\nu \zeta)\circ U(y)\Big)^{2}(1)\dd \zeta
        \end{align}
        The $\Pi_{\neq}$ part of the function above does not change the order of the pole, and it is thus a term of order $\cO_{\rm fun}(\eps^4/\nu)$. Instead, for the $\Pi_0$ component note that the circulation becomes 
        \begin{equation}
            -\frac{ \eps^2}{\nu}\frac{1}{2\pi \im }\oint_{\Gamma_0}\frac{\zeta}{-(\eps^2+\zeta)^2}\langle U(y),(\de_{yy}-(\eps^2+\zeta))^{-1} U(y)\rangle\dd \zeta.
        \end{equation}
        Applying the residue theorem and integrating by parts, we see that we get
        \begin{equation}
            -\frac{\eps^2}{\nu}\|\de_y^{-1}U\|^2_{L^2}+\cO_{\rm fun}(\nu^{-1}\eps^3).
        \end{equation}
        Being the unperturbed eigenvector $1$, this is the desired correction of the eigenvalue and is in agreement with the result announced in Remark \ref{rem:Taylorintro}. 
        
        Let us mention that a case where the Taylor dispersion property of passive scalars can be exploited in fluid dynamics is for perturbations around the Couette flow $U(y)=y$ in $\RR^2$. Indeed, the linearized problem involves exactly an operator like $\mathcal{T}_{\nu,\eps}$ and Arbon and Bedrossian \cite{Arbon25CMP} were recently able to prove that Taylor dispersion holds also at the nonlinear level (using a hypocoercivity method).
        
        Finally,  in Kinetic Theory the diffusion is replaced by a more complicated collision operator with a larger kernel. In some cases, the Kato's spectral approach used in \cite{Gervais24Hydrodynamic} to deduce hydrodynamic limits should readily imply the Taylor dispersion mechanism in the Boltzmann equation, which was first proved in \cite{BCZD24Taylor} with hypocoercivity methods (and for more general collision kernels). 
\end{remark}
\section{Normal-form approach}\label{sezione-forma-normale}
In this section we perform various conjugations which lead to a complete block-diagonalization of the operator $\cL_{\nu,\e}$ in \cref{cLnueps}. We denote by $H^s \equiv H^s(\TT)$, $s \geq 0$ the standard $L^2$-based Sobolev space with norm $\| \cdot \|_s \equiv \| \cdot \|_{H^s}$. For simplicity, we require slightly more regularity on the background shear flow as we assume $U\in H^{s+2}(\TT)$ for some $s > \tfrac12$. In this way
 the operator $\cR_{\e} $ in \cref{cMcR} is a bounded endomorphism of $H^{s}(\TT)$, the latter being an algebra. We denote by $\| \cdot \|_{2 \times 2}$ the standard Hilbert-Schmidt norm on $2 \times 2$ matrices. 
The main result of this section is the following
\begin{theorem}\label{teorema coniugio cal Lk}
Let $s > \frac12$ and $U \in H^{s + 2}(\TT)$ be a shear flow profile with zero average. There exists constants $\delta_0 \in (0,1)$ and $C>1$ depending only on $\|U\|_{s+ 2}$ and $s$, such that, if $\varepsilon \nu^{- 1} \leq \delta_0$ then there exists an invertible map $\Phi_{\nu, \varepsilon} \in {\mathcal B}(H^s, H^s)$ with inverse $\Phi_{\nu, \varepsilon}^{- 1} \in {\mathcal B}(H^s, H^s)$ such that 
$$
{\mathcal N}_{\nu, \varepsilon} \coloneqq  \Phi_{\nu, \varepsilon}^{- 1} {\mathcal L}_{\nu, \varepsilon} \Phi_{\nu, \varepsilon} 
$$
is block diagonal, i.e. 
$$
{\mathcal N}_{\nu, \varepsilon}[h] = \lambda^{(0)}_{\nu,\eps} \Pi_0 h + \sum_{j \in {\mathbb N}} \Pi_j {\mathcal N}_{\nu, \varepsilon} \Pi_j [\Pi_j h]\,,
$$
with the projections $\Pi_j$  in \cref{def:Pij}. 
The $2\times 2$ matrices $\Pi_j\cN_{\nu,\eps}\Pi_j$ and the maps $\Phi_{\nu, \varepsilon}^{\pm 1}$ satisfy the estimates
\begin{equation}\label{stime mathcal Dk}
\begin{aligned}
& \| \Pi_j {\mathcal N}_{\nu, \varepsilon} \Pi_j + \nu j^2 \uno \|_{2 \times 2} \leq C\varepsilon , \qquad \| \Phi_{\nu, \varepsilon}^{\pm 1} \|_{{\mathcal B}(H^\sigma, H^\sigma)} \leq C\eps^{-1}\,, 
\end{aligned}
\end{equation}
for all $0 \leq \sigma \leq s$ and $j\in \NN$.
Moreover,  the eigenpair $(\lambda^{(0)}_{\nu,\eps},V_{\nu,\eps}^{(0)})$
satisfies
\begin{align}
\label{exp:NFfinal}
    \lambda^{(0)}_{\nu,\eps}=\frac{\e^2}{\nu} \Big( \|\de_y^{-1}U\|^2_{L^2} - \nu^2 + \cO \big(\frac{\e}{\nu}+ \eps^2\big)\Big), \qquad     \| V_{\nu,\e}^{(0)} -(U(y) -\im \nu \e)\|_{s+2}\leq C  (\frac{\eps^2}{\nu^{2}}+\eps^2 ) \, .
\end{align}
In general, we have $\sigma_{L^2}(\cL_{\nu, \varepsilon})=\{\lambda_{\nu,\eps}^{(\pm j)}\}_{j\in \NN}$ with 
$$
|\lambda_{\nu, \varepsilon}^{(\pm j)} + \nu j^2| \leq C \varepsilon, \quad \forall j \in\NN.
$$
\end{theorem}
Note that if we choose $s = 1$, i.e. $U \in H^3(\TT)$, the theorem above clearly implies Theorem \ref{th:mainLin}. 

We split the proof of Theorem \ref{teorema coniugio cal Lk} into two subsections: we first identify the leading order transformations removing the parts of order $\eps^{-1}$. Such transformation will also isolate the unstable part of the spectrum from its stable part. Then, we proceed with a more standard perturbative argument to block-diagonalize the full operator. 

 To ease the notation, in the proofs we will write $    a\lesssim b$ where $a\leq Cb$ for a constant $C>0$ depending only on $\|U\|_{s + 2}$ and $s$, but independent of $\nu,\eps$. To invert some operators, we  need  \begin{equation}\label{invertibilityconditionstep1}
     \, \e \leq \delta_0 \nu\,,
\end{equation} with $\delta_0=\delta_0(s,\|U\|_{{s+2}})$ sufficiently small. The precise smallness condition can be deduced from the rather explicit bounds given below. However, we do not keep track of this precise constant. We will abuse in notation and always denote with $\delta_0$ the necessary smallness parameter needed in each different statement. The $\delta_0$ in Theorem \ref{teorema coniugio cal Lk} is clearly the smallest of these ones.
 
\subsection{Decoupling of stable and unstable modes}\label{normal-form-modo-zero}
We aim to separate the one-dimensional unstable eigenspace of the operator $\cL_{\nu,\e}$ in \cref{cLnueps} from its stable invariant space of codimension $1$.
We shall exploit the splitting, for every $s\geq 0$,
\begin{equation}
    \label{thesplittingHs}
    H^s(\TT)=\CC\oplus H^s_0(\TT)\,,\quad  \text{where}\quad  H^s_0(\TT) \coloneqq  \big\{ f\in H^s(\TT)\;:\; \Pi_0 f =0 \big\} 
\end{equation}
is the closed subspace of $H^s(\TT)$ formed by zero-average functions.
According to the splitting \cref{thesplittingHs}, we regard any bounded operator $\cA : H^{s+t}(\TT) \to H^s(\TT)$ of order $t\geq 0$
as a matrix $\mathtt{A}: \CC\oplus H^s_0(\TT) \to \CC\oplus H^s_0(\TT)$ of the form
\begin{align}\label{matrixrepresentation}
\mathtt{A} \coloneqq  \begin{bmatrix} \Pi_0 \cA[1] & (\Pi_{\neq} \cA^*[1])^\top  \\
\Pi_{\neq} \cA[1] & \Pi_{\neq} \cA \Pi_{\neq} \end{bmatrix}\,, 
\end{align}
where $\cA^*$ is the adjoint operator of $\cA$ with respect to the $L^2$-inner product and $f^\top$ is the bounded linear functional that sends
$L^2(\TT)\ni g \mapsto \langle g, f \rangle  \in \CC
$, for every $f\in L^2(\TT)$. We point out that the mapping $f \mapsto f^\top $ is antilinear, namely $(\zeta f)^\top = \overline{\zeta} f^\top $ for every $\zeta\in \CC$.
The 
action of the matrix $\ttA$ in \cref{matrixrepresentation} is linked to its associated operator $\cA$: more precisely, for every  $\cA \in \mathcal{B}\big( H^{s+t}(\TT); H^s(\TT)\big)$, the associated matrix $\mathtt{A} $ in \cref{matrixrepresentation} satisfies
\begin{equation}\label{matrixaction}
\mathtt{A} \begin{bmatrix} \Pi_0 h \\ \Pi_{\neq} h(y) \end{bmatrix} = \begin{bmatrix} \Pi_0 \cA[ h ]  \\
\Pi_{\neq} \cA[ h ](y) \end{bmatrix}\, \qquad \mbox{for every } h \in C^\infty(\TT).
\end{equation}
Indeed,
\begin{align}\footnotesize 
\begin{bmatrix} \Pi_0 \cA[1] & (\Pi_{\neq} \cA^*[1])^\top  \\
\Pi_{\neq} \cA[1] & \Pi_{\neq} \cA \Pi_{\neq} \end{bmatrix}\begin{bmatrix} \Pi_0 h \\ \Pi_{\neq} h \end{bmatrix} = \begin{bmatrix} (\Pi_0 \cA[1]) \Pi_0 h  + \langle \Pi_{\neq} h ,\Pi_{\neq} \cA^*[1]\rangle \\
(\Pi_0 h) \Pi_{\neq} \cA[1] + \Pi_{\neq} \cA[\Pi_{\neq} h]  \end{bmatrix}   
 = \begin{bmatrix} \Pi_0 \cA[\Pi_0 h ]  + \langle  \cA[\Pi_{\neq} h], 1 \rangle \\
\Pi_{\neq} \cA[\Pi_0 h ] + \Pi_{\neq} \cA[\Pi_{\neq} h]  \end{bmatrix}
\end{align}
which gives the right-hand side of \cref{matrixaction}.

In our case, first observe that 
\begin{align}
\cL_{\nu,\eps}^*&=\nu \cD_{\eps}+\frac{\im }{\eps}\Pi_0\circ (U''(y)\mathrm{Id})+\im \eps \big(U(y)+\Pi_{\neq}\circ(-\de_{yy}+\eps^2)^{-1}\circ (U'' \mathrm{Id})\big),
\\
\Pi_{\neq}(\cL_{\nu,\eps}^*[1])&=\im \eps \big(U(y)-(\de_{yy}-\eps^2)^{-1}(\de_{yy}-\eps^2+\eps^2)U(y)\big)=-\im \eps^3(\de_{yy}-\eps^2)U(y).
\end{align}
Therefore, the matrix $\mathtt{L}_{\nu,\e} $  associated as in \cref{matrixrepresentation} with the operator $\cL_{\nu,\e}$ in \cref{cLnueps} is given by
\begin{align}\label{cLsharp}
\mathtt{L}_{\nu,\e} \stackrel{\cref{matrixrepresentation}}{=} &\begin{bmatrix} \Pi_0 \cL_{\nu,\e}[1] & (\Pi_{\neq} \cL_{\nu,\e}^*[1])^\top  \\
\Pi_{\neq} \cL_{\nu,\e}[1] & \Pi_{\neq} \cL_{\nu,\e} \Pi_{\neq} \end{bmatrix} \\
 \stackrel{\cref{cLnueps}}{=}  &\begin{bmatrix}
-\nu \e^2 & \big( -\im \eps^3(\de_{yy}-\eps^2)^{-1}U(y)\big)^\top \\
-\frac{\im}{\e} U''(y) - \im \e U(y) & \Pi_{\neq} \cL_{\nu,\e} \Pi_{\neq} 
\end{bmatrix} 
\end{align} 
In the following we will denote $\cL_{\nu,\e}^{\sharp} \coloneqq  \Pi_{\neq} \cL_{\nu,\e} \Pi_{\neq}$ for brevity.

We aim to conjugate the matrix $\mathtt{L}_{\nu,\e}$ in \cref{cLsharp} with a transformation of the form
\begin{equation}\label{transformationmatrix}
\mathtt{T} \coloneqq  \begin{bmatrix}
1 & X(y)^\top \\ Y(y) & \uno
\end{bmatrix}\,,
\end{equation}
where $X$ and $Y$ are $2\pi$-periodic, average-free functions. We summarize the properties of such a transformation in the following lemma. Here and in the sequel, for every $f \in H^s_0(\TT)$ and $g\in L^2_0(\TT)$, we write as $f(y) g(y)^\top $ the bounded rank-$1$ operator of order $0$ that sends $L^2(\TT)\ni h \mapsto \langle h, g \rangle f(y) \in H^s_0(\TT)$.
\begin{lemma}\label{invertibilita mathtt T}
 Let $s\geq 0$, $X\in L^2_0(\TT)$  and $Y\in H^s_0(\TT)$.  If $\langle Y , X \rangle  \neq 1$, then
 the matrix $\ttT$ is invertible.
with inverse 
given by the matrix product
\begin{equation}\label{theinversetransformation}
\mathtt{T}^{-1} = \begin{bmatrix}
1 & -X(y)^\top \\ -Y(y) & \uno
\end{bmatrix} {\footnotesize \begin{bmatrix}
\dfrac{1}{1-\langle Y , X \rangle}  & 0^\top \\ 0 & \uno + \dfrac{Y(y) X(y)^\top }{1-\langle Y, X \rangle }
\end{bmatrix}}\, .
\end{equation}
\end{lemma}
\begin{proof}
Let us observe that
\begin{align}\label{matrixinthemiddle}
  {\footnotesize \begingroup 
\setlength\arraycolsep{0pt} \begin{bmatrix}
1 & X(y)^\top \\ Y(y) & \uno
\end{bmatrix}    \begin{bmatrix}
1 & -X(y)^\top \\ -Y(y) & \uno
\end{bmatrix}\endgroup = \begingroup 
\setlength\arraycolsep{-4pt}\begin{bmatrix} 1 - \langle Y,X \rangle & 0^\top \\ 0 & \uno - Y(y) X(y)^\top \end{bmatrix}\endgroup } = {\footnotesize \begingroup 
\setlength\arraycolsep{0pt}     \begin{bmatrix}
1 & -X(y)^\top \\ -Y(y) & \uno
\end{bmatrix} \begin{bmatrix}
1 & X(y)^\top \\ Y(y) & \uno
\end{bmatrix}\endgroup} \, .
\end{align}
As a consequence, if the ``diagonal'' matrix in the middle of \cref{matrixinthemiddle} is invertible, we have
\begin{equation}\label{partialTinverse}
    \ttT^{-1} = \begin{bmatrix}
1 & -X(y)^\top \\ -Y(y) & \uno
\end{bmatrix}  \begin{bmatrix} 1 - \langle Y,X \rangle & 0^\top \\ 0 & \uno - Y(y) X(y)^\top \end{bmatrix}^{-1}
\end{equation}
If $\langle Y, X \rangle \neq 0$, in view of Lemma \ref{lem:ShermanMorrisoninf}, we have
\begin{equation}\label{SMapply}
    \begin{bmatrix} 1 - \langle Y,X \rangle & 0^\top \\ 0 & \uno - Y(y) X(y)^\top \end{bmatrix}^{-1} =  {\footnotesize \begingroup 
\setlength\arraycolsep{-4pt}\begin{bmatrix}
\dfrac{1}{1-\langle Y , X \rangle}  & 0^\top \\ 0 & \uno + \dfrac{Y(y) X(y)^\top }{1-\langle Y, X \rangle }
\end{bmatrix}\endgroup}\, .
\end{equation}
Identities  \cref{SMapply,partialTinverse} give \cref{theinversetransformation}.
\end{proof}
\begin{remark}
The invertibility of the matrix $\ttT$ means that its associated operator, in the sense of  \cref{matrixaction}, is bounded and invertible, with bounded inverse which is associated to the above matrix $\ttT^{-1}$.
\end{remark}

Under a conjugation of the form \cref{transformationmatrix} the matrix $\mathtt{L}_{\nu,\e} $ in \cref{cLsharp} takes the form given in the following
\begin{lemma}\label{lem:abstarctconjugation}  Let $s\geq 0$ and $X,Y\in H^{s+2}_0(\TT)$ be such that $\langle Y , X \rangle  \neq 1$. Then the matrix $\mathtt{L}_{\nu,\e} $ in \cref{cLsharp} is conjugated by the invertible transformation $\mathtt{T}$ in \cref{transformationmatrix} into the following matrix
\begin{equation}\label{abstractconj}
\mathtt{L}_{\nu,\e}^{X,Y} \coloneqq  \mathtt{T} \mathtt{L}_{\nu,\e} \mathtt{T}^{-1} = \begin{bmatrix}
a_{\nu,\e}^{X,Y}  & \big(B_{\nu,\e}^{X}(y)\big)^\top \\[1mm] C_{\nu,\e}^{Y}(y) & \cD_{\nu,\e}^{X,Y}
\end{bmatrix} {\footnotesize \begin{bmatrix}
\dfrac{1}{1-\langle Y , X\rangle}  & 0^\top \\ 0 & \uno + \dfrac{Y(y) X(y)^\top }{1-\langle Y , X \rangle }
\end{bmatrix} }\, ,
\end{equation}
where
\begin{align}\label{aBCcD}
&a_{\nu,\e}^{X,Y}\coloneqq  -\nu\e^2 -\langle \frac{\im}{\e} U''+ \im \e U ,X\rangle + \e^3 \langle Y,\im(\de_{yy}-\e^2)^{-1}U \rangle - \langle \cL_{\nu,\e}^\sharp Y, X \rangle\, ,\\
&B_{\nu,\e}^{X}(y) \coloneqq  -\im \e^3  (\de_{yy}-\e^2)^{-1}U(y) + \big( (\cL_{\nu,\e}^\sharp)^* + \nu\e^2 \big) X(y) + \langle \frac{\im}{\e}U'' +\im \e U, X \rangle X(y)   \, , \\
&C_{\nu,\e}^{Y}(y) \coloneqq  - \frac{\im}{\e}U''(y) -\im \e U(y) -\big( \cL_{\nu,\e}^\sharp + \nu\e^2 \big) Y(y) +\e^3 \langle Y, \im(\de_{yy}-\e^2)^{-1}U \rangle Y(y)  \,,\\
&\cD_{\nu,\e}^{X,Y} \coloneqq  \cL_{\nu,\e}^\sharp +\nu\e^2 Y(y) X(y)^\top+ \big( \frac{\im}{\e} U''(y) + \im \e U(y) \big) X(y)^\top - \e^3 Y(y) \big( \im(\de_{yy}-\e^2)^{-1} U(y)\big)^\top \, .
\end{align}
\end{lemma}
\begin{proof} Let us consider a generic matrix of the form 
$
    \ttL \coloneqq  {
    \begingroup 
\setlength\arraycolsep{0pt}\begin{bmatrix}
        a & B(y)^\top \\
        C(y) & \cD
    \end{bmatrix}\endgroup}
$, with $a\in \CC$, $B,C \in H^{s+2}_0(\TT)$ and $\cD : H^{s+2}_0 \to L^2_0$ being a bounded operator. In view of \cref{transformationmatrix} we have
\begin{align}
    \ttT \ttL = \begin{bmatrix}
        a + \langle C,X \rangle & B(y)^\top + X(y)^\top \circ \cD \\
        a Y(y) + C(y) & Y(y) B(y)^\top + \cD
    \end{bmatrix} 
\end{align}
where we can write the upper-right entry  as $B(y)^\top + \big( \cD^* X(y) \big)^\top$ since
\begin{align}
X(y)^\top  \cD f = \langle \cD f , X \rangle = \langle  f ,\cD^* X \rangle = \big( \cD^* X(y) \big)^\top f\,,\quad \forall f \in L^2_0(\TT)\, . 
\end{align}
By \cref{theinversetransformation} we have
\begin{align}
    \ttT \ttL \ttT^{-1} &={\footnotesize \begingroup 
\setlength\arraycolsep{0pt} \begin{bmatrix}
        a + \langle C,X \rangle & B(y)^\top + \big( \cD^* X(y) \big)^\top \\
        a Y(y) + C(y) & Y(y) B(y)^\top + \cD
    \end{bmatrix} \begin{bmatrix}
1 & -X(y)^\top \\ -Y(y) & \uno
\end{bmatrix}\endgroup \begingroup 
\setlength\arraycolsep{-5pt}\begin{bmatrix}
\dfrac{1}{1-\langle Y , X \rangle}  & 0^\top \\ 0 & \uno + \dfrac{Y(y) X(y)^\top }{1-\langle Y, X \rangle }
\end{bmatrix}\endgroup} \\ \label{LconjTaux}
&= \begingroup 
\setlength\arraycolsep{0pt} \begin{bmatrix}
    \widetilde{a} & \widetilde{B}(y)^\top \\ \widetilde{C}(y) & \widetilde{\cD}
\end{bmatrix}\endgroup {\footnotesize\begingroup 
\setlength\arraycolsep{-5pt}\begin{bmatrix}
\dfrac{1}{1-\langle Y , X \rangle}  & 0^\top \\ 0 & \uno + \dfrac{Y(y) X(y)^\top }{1-\langle Y, X \rangle }
\end{bmatrix}\endgroup} \, ,
\end{align}
where
\begin{align}\label{LconjTauxbis}
    & \begingroup 
\setlength\arraycolsep{0pt} \begin{bmatrix}
    \widetilde{a} & \widetilde{B}(y)^\top \\ \widetilde{C}(y) & \widetilde{\cD}
\end{bmatrix}\endgroup \coloneqq  \begingroup\setlength\arraycolsep{0pt} \begin{bmatrix}
        a + \langle C,X \rangle & B(y)^\top + \big( \cD^* X(y) \big)^\top \\
        a Y(y) + C(y) & Y(y) B(y)^\top + \cD
    \end{bmatrix} \begin{bmatrix}
1 & -X(y)^\top \\ -Y(y) & \uno
\end{bmatrix}\endgroup \\
= & \begingroup 
\setlength\arraycolsep{2pt} \begin{bmatrix}
    a + \langle C,X \rangle - \langle Y,B\rangle - \langle \cD Y, X\rangle & - a X(y)^\top - \langle C,X \rangle X(y)^\top + B(y)^\top + \big( \cD^* X(y) \big)^\top \\
    aY(y) + C(y) - \langle Y,B \rangle Y(y) - (\cD Y)(y) & - a Y(y) X(y)^\top - C(y) X(y)^\top +Y(y) B(y)^\top + \cD
\end{bmatrix} \endgroup\, .
\end{align}
Formula \cref{abstractconj}  descends from  \cref{LconjTaux} and the terms $a_{\nu,\e}^{X,Y}, B_{\nu,\e}^{X,Y}, C_{\nu,\e}^{X,Y}, \cD_{\nu,\e}^{X,Y} $ are given respectively by the entries  $ \widetilde{a}$, $ \widetilde{B}$, $ \widetilde{C}$ and $\widetilde{\cD}$ in \cref{LconjTauxbis} where we substituted, in view of \cref{cLsharp}, 
\begin{align} 
    a\to -\nu\e^2, \quad B(y)\to -\im \e^3 (\de_{yy} -\e^2)^{-1} U(y), \quad C(y)\to -\frac{\im}{\e} U''(y) - \im \e U(y), \quad \cD\to \cL_{\nu,\e}^\sharp.
\end{align}
This concludes the proof of the lemma.
\end{proof}
We are now about to construct the functions $X$ and  $Y$ that annihilate respectively the terms $B_{\nu,\e}^X$ and $C_{\nu,\e}^Y$ appearing in \cref{abstractconj}-\cref{aBCcD}. The key step of the construction consists in inverting the operator $\cA_{\nu,\e}: H^{s+2}_0(\TT) \to H^s_0(\TT) $ given by
\begin{equation}\label{cAnueps}
    \cA_{\nu,\e} \coloneqq  \cL_{\nu,\eps}^\sharp+\nu\eps^2 = \nu \de_{yy} - \im \e \cR_{\e}^\sharp \,, \quad\text{where}\ \ \cR_{\e}^\sharp \coloneqq  \Pi_{\neq} \cR_{\e} \Pi_{\neq}\, ,
\end{equation}
with $\cR_\e$ in \cref{cMcR}. 
Since $H^s$ is an algebra for $s>1/2$, the operator $\cR_{\e}^\sharp$ is bounded by, for every $h \in H_0^s(\TT)$,
\begin{equation}
    \| \cR_\e^\sharp h \|_s \lesssim \big( \| U\|_s \|h\|_s + \|U''\|_s \| (-\de_{yy} + \e^2)^{-1} h\|_s \big) \lesssim  \|h\|_s\, .  
\end{equation}
As a consequence, provided that \cref{invertibilityconditionstep1} holds with $\delta_0$ sufficiently small, the operator $\cA_{\nu,\e}$ in \cref{cAnueps} is invertible with inverse given by a Neumann series and bounded by
 \begin{equation}\label{boundinversecAnueps}
     \|\cA_{\nu,\e}^{-1}\|_{{\mathcal B}(H^s_0, H^{s + 2}_0) }   
     \lesssim \frac 1\nu\, .
 \end{equation}
 We now observe that the adjoint operator of $\cA_{\nu,\e}$ in \cref{cAnueps} is a well-defined bounded linear operator $\cA_{\nu,\e}^*: H^{s+2}_0(\TT) \to H^s_0(\TT)$ given by 
\begin{equation}\label{cAstar}
    \cA_{\nu,\e}^* \coloneqq  \nu \de_{yy} + \im \e \Pi_{\neq}\big( U(y) + (-\de_{yy}+\e^2 )^{-1} \circ U''(y) \big) \Pi_{\neq}\, .
\end{equation}
Under the condition \cref{invertibilityconditionstep1}, the operator $\cA_{\nu,\e}^*$ is invertible with inverse fulfilling the estimate
 \begin{equation}\label{boundinversecAstarnueps}
     \|[\cA_{\nu,\e}^*]^{-1}\|_{{\mathcal B}(H^s_0, H^{s + 2}_0) } \lesssim \frac{1}{\nu}
     \, .
 \end{equation}
%
%
In the regime given by \cref{invertibilityconditionstep1}, the aforementioned functions $X$ and $Y$ are found as fixed points of quadratic mappings. In such a setting we have the following abstract result.
\begin{lemma}\label{lem:quadraticfixedpoint}
Let $(\ttX,\|\cdot\|)$ be a Banach space and $Q: \ttX \times \ttX \to \ttX $ be a bilinear mapping satisfying
\begin{equation}\label{continuitybilinearform}
    \| Q(u,v)\| \leq c \|u\| \|v\|\,,\quad \forall\, u , v \in \ttX\, ,
\end{equation}
for some $c>0$. 
Let $u_0\in \ttX$ such that $\| u_0 \|< \dfrac{1}{4c} $. Then the  fixed-point problem
\begin{equation}\label{abstractfixedpoint}
    u = F(u)\coloneqq  u_0 + Q(u,u)
\end{equation}
possesses a solution $u \in \ttX$ satisfying
 \begin{equation}\label{expansionabstractfixedpoint}
       \| u - u_0 \| \leq 4 c \|u_0\|^2 \, .
    \end{equation}
\end{lemma}
The proof of this lemma is a standard application of Banach fixed point theorem and is thus omitted.
We are now in a position to prove the following
\begin{lemma}[``vertical'' homological equation]\label{lem:vheq}
Let $\varepsilon \nu^{- 1} \leq \delta_0$  for some small $\delta_0 \coloneqq  \delta_0(s, \|U\|_s)$. 
Then there exists $Y_{\nu,\e} \in H^{s+2}_0(\TT)$ such that the term $C_{\nu,\e}^Y$ in \cref{aBCcD} vanishes at $Y= 
Y_{\nu,\e}$. Moreover, with $\cA_{\nu,\e}$ in \cref{cAnueps}, the function $Y_{\nu,\e}$ satisfies
\begin{equation}
\label{actualY}
\big\| Y_{\nu,\e} + \cA_{\nu,\e}^{-1}\Big(\frac{\im}{\e}U''(y) +\im \e U(y) \Big)  \big\|_{s+2}  \lesssim \frac{\eps}{\nu^3}
\end{equation}
\end{lemma}
\begin{proof}  In view of \cref{aBCcD}, we aim to find a vector $Y\in H^{s+2}_0(\TT)$ that solves the following equation
   \begin{equation}\label{homoeq2}
     \frac{\im}{\e}U''(y) +\im \e U(y) + \cA_{\nu,\e} Y(y) -\e^3 \langle Y, \im(\de_{yy}-\e^2)^{-1}U \rangle Y(y)   = 0\, ,
   \end{equation}
   with $\cA_{\nu,\e}$ in \cref{cAnueps}. If \cref{invertibilityconditionstep1} holds, 
we can recast equation \cref{homoeq2} into the following fixed-point problem
\begin{equation}\label{homofixedpoint2}
    Y(y) = -  \cA_{\nu,\e}^{-1}\Big(\frac{\im}{\e}U''(y) +\im \e U(y) \Big)  + \e^3 \langle Y, \im(\de_{yy}-\e^2)^{-1}U \rangle  \cA_{\nu,\e}^{-1} Y(y) \, .
\end{equation}
We now aim to apply Lemma \ref{lem:quadraticfixedpoint} with $\ttX \coloneqq  H^{s+2}_0(\TT)$ and
\begin{equation}\label{choiceQu0forY}
    Q(u,v)(y)\coloneqq   \e^3 \langle u, \im(\de_{yy}-\e^2)^{-1}U \rangle  \cA_{\nu,\e}^{-1} v(y)\,,\quad u_0\coloneqq  -  \cA_{\nu,\e}^{-1}\Big(\frac{\im}{\e}U''(y) +\im \e U(y) \Big) \, .
\end{equation}
By applying first the Cauchy-Schwarz inequality and then \cref{boundinversecAnueps}, note that
\begin{align}
    \| Q(u,v)\|_{s+2} &\leq \e^3 \| (\de_{yy}-\e^2)^{-1}U \|_{L^2} \|u\|_{L^2} \|\cA_{\nu,\e}^{-1} v \|_{s+2} 
    \lesssim \frac{\eps^3}{\nu}\|u\|_{s+2} \| v \|_{s+2}\, .
\end{align}
For the function $u_0$, in view of \cref{boundinversecAnueps} we deduce that
\begin{equation}\label{estimateonu0forY}
    \|u_0\|_{s+2} \lesssim \frac{1}{\nu\e}.
\end{equation}
To apply Lemma \ref{lem:quadraticfixedpoint}, we just need to check that $\nu\eps\gg \eps^3\nu^{-1}$, which is certainly satisfied provided that $\delta_0$ in \cref{invertibilityconditionstep1} is sufficiently small.
Thus all the assumptions of Lemma \ref{lem:quadraticfixedpoint} are verified and we conclude that the mapping $\Psi$ in \cref{homofixedpoint2} has a fixed point $Y_{\nu,\e} \in H_0^{s+2}(\TT)$ that satisfies \cref{actualY}.
\end{proof}

\begin{lemma}[``horizontal'' homological equation] \label{lem:hheq} Let $\varepsilon \nu^{- 1} \leq \delta_0$  for some small $\delta_0 \coloneqq  \delta_0(s, \|U\|_s)$.
Then there exists $X_{\nu,\e} \in H^{s+2}_0(\TT)$, satisfying
\begin{equation}
\label{actualX}
\big\| X_{\nu,\e}-\im \e^3  [\cA_{\nu,\e}^*]^{-1} (\de_{yy}-\e^2)^{-1}U(y) \big\|_{s+2} \lesssim \frac{\eps^5}{\nu^3}
\end{equation}
with $\cA_{\nu,\e}^*$ in \cref{cAstar},
such that the term $B_{\nu,\e}^X$ in \cref{aBCcD} vanishes at $X= 
X_{\nu,\e}$. 
\end{lemma}
\begin{proof}  In view of \cref{aBCcD}, we aim to find a vector $X\in H^{s+2}_0(\TT)$ that solves the following equation
   \begin{equation}\label{homoeq1}
    -\im \e^3  (\de_{yy}-\e^2)^{-1}U(y) + \cA_{\nu,\e}^* X(y) + \langle \frac{\im}{\e}U'' +\im \e U, X \rangle X(y)  = 0
   \end{equation}
   with $\cA_{\nu,\e}^*$ in \cref{cAstar}. If \cref{boundinversecAstarnueps} holds, 
we can recast equation \cref{homoeq1} into the following fixed-point problem
\begin{equation}\label{homofixedpoint1}
    X(y) =   \im \e^3  [\cA_{\nu,\e}^*]^{-1} (\de_{yy}-\e^2)^{-1}U(y) - \langle \frac{\im}{\e}U'' +\im \e U, X \rangle [\cA_{\nu,\e}^*]^{-1}X(y)  =: \Phi(X)(y)\, .
\end{equation}
We now aim to apply Lemma \ref{lem:quadraticfixedpoint} with $\ttX \coloneqq  H^{s+2}_0(\TT)$ and
\begin{equation}\label{choiceQu0forX}
    Q(u,v)(y)\coloneqq  - \langle \frac{\im}{\e}U'' +\im \e U, u \rangle [\cA_{\nu,\e}^*]^{-1}v(y)\,,\quad u_0\coloneqq  \im \e^3  [\cA_{\nu,\e}^*]^{-1} (\de_{yy}-\e^2)^{-1}U(y) \, .
\end{equation}
By applying first the Cauchy-Schwarz inequality and then \cref{boundinversecAstarnueps}, we get
\begin{align}
    \| Q(u,v)\|_{s+2} &\leq \| \frac{\im}{\e}U'' +\im \e U\|_{L^2} \|u \|_{L^2} \|[\cA_{\nu,\e}^*]^{-1}v \|_{s+2} \lesssim \frac{1}{\nu \eps}\|u\|_{s+2} \| v \|_{s+2}\\
    \|u_0\|_{s+2}&\lesssim \frac{\eps^3}{\nu}.
\end{align}
All the hypotheses of Lemma \ref{lem:quadraticfixedpoint} are verified and we conclude that the mapping $\Phi$ in \cref{homofixedpoint1} has a fixed point $X_{\nu,\e} \in H_0^{s+2}(\TT)$ that
satisfies \cref{actualX}.
\end{proof}
Let us observe that, under condition \cref{invertibilityconditionstep1} with $\delta_0$ sufficiently small, the inner product between $Y_{\nu,\e}$ in Lemma \ref{lem:vheq} and   $X_{\nu,\e}$ in Lemma  \ref{lem:hheq}  is way smaller than $1$. Indeed,
\begin{equation}\label{YperX}
    \big| \langle Y_{\nu,\e},X_{\nu,\e} \rangle\big| \stackrel{\cref{actualY},\, \cref{actualX}}{\lesssim} \frac{\e^2}{\nu^2} \stackrel{\cref{invertibilityconditionstep1}}{<} \frac12  \, .
\end{equation}
We are now in a position to conjugate the matrix $\ttL_{\nu,\e}$ in \cref{cLsharp} into the decoupled matrix 
\begin{equation}\label{matrixL1}
    \ttL_{\nu,\e}^{(1)}\coloneqq  \ttT_{\nu,\e} \ttL_{\nu,\e} \ttT_{\nu,\e}^{-1}  = \ttL {\scriptsize \begingroup 
\setlength\arraycolsep{0pt} \begin{matrix} X_{\nu,\e} &, & Y_{\nu,\e} \\ \nu, \e & & \end{matrix}\endgroup}\, ,
\end{equation}
where $\ttT_{\nu,\e}$ is the matrix  $\ttT$ in  \cref{transformationmatrix} with $X\to X_{\nu,\e}$ and $Y\to Y_{\nu,\e}$, whereas $\ttL_{\nu,\e}^{X,Y}$ is the matrix in \cref{abstractconj}.
We describe the entries of the new matrix $\ttL_{\nu,\e}^{(1)}$ in 
the following
\begin{lemma}\label{primo coniugio normal-form}
The matrix $\ttL_{\nu,\e}^{(1)}$ in \cref{matrixL1} is given by
\begin{equation}
    \ttL_{\nu,\e}^{(1)} = \begin{bmatrix}
        \lambda_{\nu,\e}^{(0)} & 0^\top \\
        0 & \cL_{\nu,\e}^{(1)}
    \end{bmatrix} \,,
\end{equation}
with $\lambda_{\nu,\eps}^{(0)}$ as in \cref{exp:NFfinal}
and  $\cL_{\nu,\e}^{(1)} : H^{s+2}_0(\TT) \to H^s_0(\TT)$ is given by
\begin{align}\label{stableoperator}
\cL_{\nu,\e}^{(1)} \coloneqq   \Big( {\mathcal L}_{\nu, \varepsilon}^\sharp +  A(y) X_{\nu,\e}(y)^\top +  Y_{\nu,\e}(y) B(y)^\top\Big) \circ \Big( {\rm Id} + \frac{1}{1 - \langle X_{\nu,\e}, Y_{\nu,\e} \rangle} Y_{\nu,\e}(y) X_{\nu,\e}(y)^\top \Big)\, ,
\end{align}
where 
\begin{align}
\label{def:Anor}
& A(y) \coloneqq  \nu \varepsilon^2 Y_{\nu,\e}(y) +\frac{\im}{\eps} U''(y)  +\im\varepsilon U(y)\,, \\
\label{def:Bnor}& B(y) \coloneqq -\im\varepsilon^3 (\partial_{yy} - \varepsilon^2)^{- 1} U(y)\, .
\end{align}
\end{lemma}
\begin{proof}
In view of Lemmas \ref{lem:vheq} and \ref{lem:hheq} we write
    \begin{equation}\label{XandY}
        X_{\nu,\e} = \im \e^3  [\cA_{\nu,\e}^*]^{-1} (\de_{yy}-\e^2)^{-1}U + \im \frac{\e^5}{\nu^3}  \xi_{\nu,\e}\,,\quad Y_{\nu,\e} = -\cA_{\nu,\e}^{-1}\Big(\frac{\im}{\e}U'' +\im \e U \Big) - \im \frac{\e}{\nu^3} \upsilon_{\nu,\e} 
    \end{equation}
    where
    \begin{equation}\label{xiandupsilon}
       \| \xi_{\nu,\e} \|_{s+2} \lesssim 1
       \qquad \| \upsilon_{\nu,\e} \|_{s+2} \lesssim 1
    \end{equation}
    By Lemmas \ref{lem:abstarctconjugation}, \ref{lem:vheq} and \ref{lem:hheq}, the first column and the first row of the matrix $\ttL_{\nu,\e}^{(1)}$ are zero except for the diagonal entry. 
    We keep the notation of Lemma \ref{lem:abstarctconjugation} and compute the remaining entries. Let us observe that, by \cref{cLsharp,cAnueps},
    \begin{equation}\label{anchequesta}
        \cL_{\nu,\e}^{\sharp} = \cA_{\nu,\e} - \nu\e^2\, .
    \end{equation}
Then we compute    \begin{align}
        &a{\scriptsize \begingroup 
\setlength\arraycolsep{0pt} \begin{matrix} X_{\nu,\e} &,  Y_{\nu,\e} \\ \nu, \e & & \end{matrix}\endgroup} = -\nu\e^2  + \e^3 \langle Y_{\nu,\e}, \im(\de_{yy}-\e^2)^{-1}U \rangle  -\langle \frac{\im}{\e} U''+ \im \e U ,X_{\nu,\e}\rangle- \langle \cL_{\nu,\e}^\sharp Y_{\nu,\e}, X_{\nu,\e} \rangle \\
&\stackrel{\cref{XandY}}{=} -\nu\e^2  -\e^2 \langle \cA_{\nu,\e}^{-1} (U'' + \e^2 U)+ \frac{\e^2}{\nu^3} \upsilon_{\nu,\e}, (\de_{yy}-\e^2)^{-1}U\rangle\\
&\qquad - \e^2 \langle  U'' +\e^2 U,   [\cA_{\nu,\e}^*]^{-1} (\de_{yy}-\e^2)^{-1}U + \frac{\e^2}{\nu^3} \xi_{\nu,\e} \rangle \\
&\qquad + \e^2 \langle  (\cA_{\nu,\e}-\nu\e^2) \big(\cA_{\nu,\e}^{-1} (U''+ \e^2 U) + \frac{\e^2}{\nu^3} \upsilon_{\nu,\e}\big), [\cA_{\nu,\e}^*]^{-1} (\de_{yy}-\e^2)^{-1} U+ \frac{\e^2}{\nu^3} \xi_{\nu,\e} \rangle 
    \end{align}
    where in the last passage we used also \cref{anchequesta}. Isolating the terms of order $\eps^2$, we see a crucial cancellation between the terms in the last two lines appearing above, namely 
    \begin{align}
        - \e^2 \langle  U'' ,   [\cA_{\nu,\e}^*]^{-1} (\de_{yy}-\e^2)^{-1}U \rangle+ \e^2 \langle  \cA_{\nu,\e} \cA_{\nu,\e}^{-1} U'', [\cA_{\nu,\e}^*]^{-1} (\de_{yy}-\e^2)^{-1} U\rangle =0.
    \end{align}
By taking into account similar cancellations happening at lower orders, involving in particular the term $\e^2 \langle  U'' +\e^2 U, \frac{\e^2}{\nu^3} \xi_{\nu,\e} \rangle $ appearing twice with opposite signs in the last and second last line, and the term $ -\frac{\e^4}{\nu^3}\langle  \upsilon_{\nu,\e}, (\de_{yy}-\e^2)^{-1}U\rangle$ appearing twice in the second and last line,
we obtain
\begin{align}
&a{\scriptsize \begingroup 
\setlength\arraycolsep{0pt} \begin{matrix} X_{\nu,\e} &,  Y_{\nu,\e} \\ \nu, \e & & \end{matrix}\endgroup} = - \e^2 \big( \nu + \langle \cA_{\nu,\e}^{-1} U'', (\de_{yy}-\e^2)^{-1}U\rangle \big) - \e^4 \langle U , [\cA_{\nu,\e}^*]^{-1} (\de_{yy}-\e^2)^{-1} U \rangle   \\
&\quad  -\nu\e^4 \langle \cA_{\nu,\e}^{-1} (U''+ \e^2 U) +\frac{\eps^2}{\nu^3}\upsilon_{\nu,\eps}, [\cA_{\nu,\e}^*]^{-1} (\de_{yy}-\e^2)^{-1} U+ \frac{\e^2}{\nu^3} \xi_{\nu,\e} \rangle. 
\end{align}
This gives, by \cref{boundinversecAnueps,boundinversecAstarnueps,xiandupsilon},
\begin{equation}
   a{\scriptsize \begingroup 
\setlength\arraycolsep{0pt} \begin{matrix} X_{\nu,\e} &,  Y_{\nu,\e} \\ \nu, \e & & \end{matrix}\endgroup} = - \e^2 \big( \nu + \langle \cA_{\nu,\e}^{-1} U'', (\de_{yy}-\e^2)^{-1}U\rangle \big) + \cO \Big( \frac{\e^4}{\nu} 
\Big)\, .
\end{equation}
    We now exploit that 
    \begin{equation}
    \label{eq:A-1Usecond}
        \cA_{\nu,\e}^{-1} U'' \stackrel{\cref{cAnueps}}{=} \frac1\nu U + \im \frac{\e}{\nu} \cA_{\nu,\e}^{-1} \cR_{\e}^{\sharp} U
    \end{equation} to further simplify the above expression into
\begin{align}\label{cumbersomea}
 a{\scriptsize \begingroup 
\setlength\arraycolsep{0pt} \begin{matrix} X_{\nu,\e} &,  Y_{\nu,\e} \\ \nu, \e & & \end{matrix}\endgroup} = - \e^2 \big( \nu + \tfrac1{\nu} \langle U, (\de_{yy}-\e^2)^{-1}U\rangle \big) + \cO \Big(\frac{\e^3}{\nu^2}+ \frac{\e^4}{\nu} \Big)\, ,
\end{align}
where we observe that
\begin{align}\label{segnogiusto}
   \langle U,(\de_{yy}-\eps^2)^{-1}U\rangle=-\|\de_y^{-1}U\|^2_{L^2}+\eps^2\langle \de_{yy}^{-1}U,(\de_{yy}-\eps^2)^{-1}U\rangle \, .
\end{align}
Finally, the first diagonal entry of the matrix $\ttL_{\nu,\e}^{(1)}$ is given by, in view of \cref{abstractconj,cumbersomea,segnogiusto,YperX},
\begin{equation}
    \lambda_{\nu,\e}^{(0)} =  \frac{a{\scriptsize \begingroup 
\setlength\arraycolsep{0pt} \begin{matrix} X_{\nu,\e} &,  Y_{\nu,\e} \\ \nu, \e & & \end{matrix}\endgroup}}{1-\langle Y_{\nu,\e}, X_{\nu,\e}\rangle} = \frac{\e^2}{\nu} \Big( \|\de_y^{-1}U\|^2_{L^2} - \nu^2 + \cO \big(\frac{\e}{\nu}+ \eps^2\big)\Big)  \, ,
\end{equation}
which proves the desired expansion in \cref{exp:NFfinal}.
Finally, the definition of the operator $\cL_{\nu,\e}^{(1)}$ descends directly from Lemma \ref{lem:abstarctconjugation}.
 \end{proof}

\subsection{Block-diagonalization of the stable part}\label{block-diagonalization-stable-part}
It now remains to study the operator  $ \cL_{\nu,\e}^{(1)}$  in \cref{stableoperator}.
\begin{equation}
\begin{aligned}
  \cL_{\nu,\e}^{(1)} =  \Big( {\mathcal L}_{\nu, \varepsilon}^\sharp +  A(y) X(y)^\top +  Y(y) B(y)^\top\Big) \circ \Big( {\rm Id} + \frac{1}{1 - \langle X, Y \rangle} Y(y) X(y)^\top \Big)
   \end{aligned}
\end{equation}
where 
$X \equiv X_{\nu, \varepsilon}$ and $Y \equiv Y_{\nu, \varepsilon}$ come from Lemmas \ref{lem:hheq} and \ref{lem:vheq} respectively, and fulfill
\begin{equation}\label{stima X Y per block decoupling}
\| X \|_{s + 2} \lesssim_s \varepsilon^3 \nu^{- 1} \quad \text{and} \quad \| Y \|_{s + 2} \lesssim_s \varepsilon^{- 1} \nu^{- 1}\, .
\end{equation}
We observe  that, under condition \cref{invertibilityconditionstep1} with $\delta_0$ small enough,  we have, by \cref{YperX,def:Anor,def:Bnor,stima X Y per block decoupling},
\begin{equation}\label{stima A B per block decoupling}
\| A \|_s \lesssim_s \varepsilon^{- 1}\,, \quad \| B \|_{s + 4} \lesssim_s \varepsilon^3\,, \quad |\langle X, Y \rangle| \lesssim \e^2 \nu^{- 2}\,.
\end{equation}
We thus  write the linear operator
${\mathcal L}_{\nu, \varepsilon}^{(1)}$ as a perturbation of the operator $\nu\de_{yy}$, namely
\begin{equation}\label{forma cal L nu varepsilon per normal-form}
{\mathcal L}_{\nu, \varepsilon}^{(1)}  = \nu \partial_{yy}+\cQ 
\end{equation}
where the operator ${\mathcal Q}:H^s_0(\TT) \to H^s_0(\TT)$  is given by
\begin{align}
\label{def:Q}
{\mathcal Q} \coloneqq \,& \frac{\nu}{1 - \langle X, Y \rangle} Y''(y) X(y)^\top \\
&+ \Big( -\im \eps{\mathcal R}_\varepsilon  + A(y) X(y)^\top + Y(y) B(y)^\top\Big) \circ \Big( {\rm Id} + \frac{1}{1 - \langle X, Y \rangle} Y(y) X(y)^\top \Big)\,,
\end{align}
where $Y''$ is the second derivative of the function $Y$. To block-diagonalize the operator ${\mathcal L}_{\nu, \varepsilon}^{(1)}$ via a normal-form reduction scheme, we introduce the following
\subsubsection*{\textbf{Technical tools}}
First of all, we recall the notation for the orthogonal projections $\Pi_j$ given in \cref{def:Pij} the matrix representation in \cref{def:Rij}. We denote by ${\mathcal M}_{bd}$, the set of the block diagonal operators, namely
\begin{equation}\label{def-block-diag}
{\mathcal M}_{bd} \coloneqq  \big\{ {\mathcal T} \in \cB(L^2_0, L^2_0)\; :\; \Pi_j {\mathcal T} \Pi_{j'} = 0, \quad \forall j \neq j' \big\}\,.
\end{equation}
Then, given $s \geq 0$ and a closed linear operator $\cT$ on $L^2(\TT)$, we define its block-decay norm as
\begin{align}\label{def-decay-norm}
    |{\mathcal T}|_s \coloneqq  \sup_{j' \in {\mathbb N}} \Big( \sum_{j \in {\mathbb N}} \langle j - j' \rangle^{2s} \|  \Pi_{j}{\mathcal T}\, \Pi_{j'} \|_{2\times 2}^2 \Big)^{\frac12}\,,\qquad \langle x\rangle\coloneqq (1+x^2)^{\frac12}\,,
\end{align}
where $\|\cdot\|_{2\times 2}$ is the standard Hilbert-Schmidt of $2\times 2$ matrix. We also define the space $ \mathcal{M}^s$ containing all  the operators $\cT$ for which $|\cT|_s < +\infty$. Clearly one can verify immediately that 
\begin{equation}\label{inclusioni-cal-Ms}
0 \leq s \leq s' \Longrightarrow {\mathcal M}^{s'} \subseteq {\mathcal M}^s \quad \text{and} \quad |\cdot|_s \leq |\cdot|_{s'}\,. 
\end{equation}
By exploiting the definitions of block-diagonal operator and the one of the norm $|\cdot|_s$, for every ${\mathcal T} \in {\mathcal M}_{bd}$ and $s \geq 0$, we have that ${\mathcal T} \in {\mathcal M}^s$ and 
$
|{\mathcal T}|_s = \sup_{j \in {\mathbb N}} \| \Pi_j {\mathcal T} \Pi_j \|_{2\times 2} 
$.
For any $s \geq 0$, we also define the set with ``off-diagonal'' operators
\begin{equation}\label{def-off-diag}
{\mathcal M}^s_{od} \coloneqq  \big\{ {\mathcal T} \in {\mathcal M}^s : \Pi_j {\mathcal T} \Pi_j = 0, \quad \forall j \in {\mathbb N} \big\}\,. 
\end{equation}
Clearly 
$$
{\mathcal M}^s = {\mathcal M}_{bd} \oplus {\mathcal M}^s_{od}, \quad \forall s \geq 0\,,
$$
and let $\Pi_{bd}$ and $ \Pi_{od}$ be the projections respectively onto ${\mathcal M}_{bd}$ and $ {\mathcal M}_{od}^s$  induced by the direct sum above.

Let us recall from \cite{KIRCHOFF} the main crucial quantitative properties of the  space ${\mathcal M}^s$.  
\begin{lemma}\label{lemma-algebra-decay}
Let $s>1/2$ and $\cT,\cT_1,\cT_2$ be operators in $\cM^s$. The following holds true: 
  \begin{enumerate}[label=(\roman*)]
      \item \label{item:algebra}The space $\cM_s$ is an algebra, namely
    \begin{equation}
        |\cT_1 \cT_2|_s \lesssim |\cT_1|_s |\cT_2|_s\,, 
    \end{equation}
    \item \label{item:nalgebra} There exists a constant $C(s) > 0$ such that for any integer $n \geq 1$, $$|{\mathcal T}^n|_s \leq C(s)^{n - 1} |{\mathcal T}|_s^n.$$
    \item \label{item:Id+T}There exists a constant $\delta=\delta(s) \in (0, 1)$ such that, if $|{\mathcal T}|_s \leq \delta$, then the operator ${\rm Id} + {\mathcal T}$ is invertible and $({\rm Id} + {\mathcal T})^{- 1} \in {\mathcal M}^s$ with
   $$
   \Big| ({\rm Id} + {\mathcal T})^{- 1} - {\rm Id} \Big|_s \lesssim |{\mathcal T}|_s\,. 
   $$
   \item \label{item:Bsigma}For any $0 \leq \sigma \leq s$, one has ${\mathcal T} \in {\mathcal B}(H^\sigma_0, H^\sigma_0)$ with
   $$
   \| {\mathcal T} \|_{{\mathcal B}(H^\sigma_0, H^\sigma_0)} \lesssim  |{\mathcal T}|_s \,. 
   $$
  \end{enumerate}
\end{lemma}
\begin{proof}
The properties \textit{(i)-(iv)} directly follow from Lemmas 2.7-2.10 in \cite{KIRCHOFF}.
\end{proof}
In Proposition \ref{prop-block-diag-contraction} below we need to compose block-diagonal and off-diagonal operators, therefore we observe that
\begin{equation}\label{lemma-comp-diag-block-diag}
 {\mathcal Z} \cT, \cT {\mathcal Z} \in {\mathcal M}^s_{od} \qquad \mbox{for every }{\mathcal Z} \in {\mathcal M}_{bd}, {\mathcal T} \in {\mathcal M}^s_{od}. 
\end{equation}
Indeed for any $j \in {\mathbb N}$, one has that 
$$
\Pi_j {\mathcal Z} \cT \Pi_j = \sum_{j' \in {\mathbb N}} (\Pi_j {\mathcal Z} \Pi_{j'}) (\Pi_{j'} \cT \Pi_j) = (\Pi_j {\mathcal Z} \Pi_j) (\Pi_j \cT \Pi_j) = 0
$$
where we used that, by definition, $\Pi_j {\mathcal Z} \Pi_{j' } = 0$ if $j \neq j'$ and $\Pi_j \cT \Pi_j = 0$ for any $j \in {\mathbb N}$. 

In order to estimate the block decay norm $|\cdot |_s$ of the remainder ${\mathcal Q}$ in \cref{forma cal L nu varepsilon per normal-form}, we also need the following
\begin{lemma}\label{stima resto cal L nu (1) cal Q}
Let $m,s\geq 0$, $a, Q\in H^s_0(\TT)$ and $P\in H^{s+m}_0(\TT)$. The following holds true:
\begin{enumerate}[label=(\roman*)]
    \item \label{item:multiplication}The multiplication operator ${\mathrm M}_a : u \mapsto a u$ belongs to ${\mathcal M}^s$, with $|{\mathrm M}_a|_s \lesssim \| a \|_s$.    
    \item  \label{item:top}The operator ${\mathcal A} \coloneqq  (\partial_y^m P) Q^\top \in {\mathcal M}^s$ and $$|(\partial_y^m P) Q^\top|_s \lesssim \| P \|_{s + m} \| Q \|_s.$$
\end{enumerate}
\end{lemma}
\begin{proof} 
Recalling the notation \cref{FourierCoefficients}, we have that for any $j, j' \in {\mathbb N}$, 
$$
\Pi_j {\mathrm M}_a \Pi_{j'} \equiv \begin{pmatrix}
a_{j - j'} & a_{j + j'} \\
a_{- j - j'} & a_{j' - j}
\end{pmatrix}\,.
$$
Using the trivial fact that if $j, j' \in {\mathbb N}$, then $\langle j - j' \rangle^s = \langle j' - j \rangle^s \leq \langle j + j' \rangle^s = \langle - j - j' \rangle^s$, one has that for any $j' \in {\mathbb N}$
$$
\begin{aligned}
\sum_{j \in {\mathbb N}} \langle j - j' \rangle^{2 s} \| \Pi_j {\mathrm M}_a \Pi_{j'}\|_{2\times 2}^2 & = \sum_{j \in {\mathbb N}} \langle j - j' \rangle^{2s} | a_{j - j'}|^2 + \sum_{j \in {\mathbb N}} \langle j - j' \rangle^{2s} | a_{j + j'}|^2 \\
& \quad + \sum_{j \in {\mathbb N}} \langle j - j' \rangle^{2s} |a_{- j + j'}|^2 + \sum_{j \in {\mathbb N}} \langle j - j' \rangle^{2s} |a_{- j - j'}|^2 \\
& \leq \sum_{j \in {\mathbb N}} \langle j - j' \rangle^{2s} |a_{j - j'}|^2 + \sum_{j \in {\mathbb N}} \langle j + j' \rangle^{2s} |a_{j + j'}|^2 \\
& \quad + \sum_{j \in {\mathbb N}} \langle j' - j \rangle^{2s} |a_{ j' - j}|^2 + \sum_{j \in {\mathbb N}} \langle - j - j' \rangle^{2s} |a_{- j - j'}|^2 \leq 4\| a \|_s^2\,.
\end{aligned}
$$
The claimed bound (i) then follows by taking the supremum over $j' \in {\mathbb N}$.

\smallskip

A direct calculation shows that 
$$
\Pi_j {\mathcal A} \Pi_{j'} \equiv \begin{pmatrix}
\im^m j^m P_j  Q_{j'} & \im^m j^m P_{j}  Q_{- j'} \\[2mm]
\im^m (- j)^m  P_{-j} Q_{j'} & \im^m (- j)^m  P_{- j}  Q_{- j'}
\end{pmatrix}
$$
Hence, for any $j' \in {\mathbb N}$, one has that 
$$
\begin{aligned}
\sum_{j \in {\mathbb N}} \langle j - j' \rangle^{2s} \| \Pi_j {\mathcal A} \Pi_{j'} \|_{2\times 2}^2 & \lesssim \sum_{j \in {\mathbb N}} \langle j - j' \rangle^{2s} |j|^{2 m} |P_{j}|^2 |Q_{j'}|^2 + \sum_{j \in {\mathbb N}} \langle j - j' \rangle^{2s}|j|^{2 m} |P_{j}|^2 |Q_{-j'}|^2 \\
& \quad + \sum_{j \in {\mathbb N}} \langle j - j' \rangle^{2s} |j|^{2 m} |P_{-j}|^2 |Q_{j'}|^2 \\
&\quad + \sum_{j \in {\mathbb N}} \langle j - j' \rangle^{2s} |j|^{2 m} |P_{-j}|^2 |Q_{-j'}|^2 \,.
\end{aligned}
$$
Using that for any $\alpha \geq 0$, $\langle j - j' \rangle^\alpha \lesssim \langle j \rangle^\alpha + \langle j' \rangle^\alpha \lesssim \langle j \rangle^\alpha \langle j' \rangle^\alpha$ (and clearly that $\langle k \rangle = \langle - k \rangle$ for any $k$), one obtains (ii) (passing to the supremum over $j'$ in the latter sum)
$$
\begin{aligned}
|{\mathcal A}|_s^2 & \lesssim \sup_{j' \in {\mathbb Z} \setminus \{ 0 \}} \sum_{j \in {\mathbb Z} \setminus \{0 \} }\langle j \rangle^{2 (s + m)} |P_{j}|^2 \langle j' \rangle^{2 s} |Q_{j'}|^2  \\
& \lesssim_s  \| P \|_{s + m}^2 \| Q \|_s^2\,.
\end{aligned}
$$
\end{proof}

Finally, a key point in the normal-form reduction scheme is to be able to find the desired transformation by solving a commutator equation. In our specific setting, we need the following.

\begin{lemma}\label{eq-omologica-Pneq}
Let $s \geq 0$, ${\mathcal R} \in {\mathcal M}^s_{od}$. Then there exists a unique solution $\Psi  \in {\mathcal M}^{s + 1}_{od}$ of the equation
$$
[\nu \partial_{yy}\,,\, \Psi] + {\mathcal R}= 0.
$$
The solution operator $ {\bf T}({\mathcal R})\coloneqq \Psi$ satisfies the bounds
 \begin{equation}
     |{\bf T}({\mathcal R})|_s \leq |{\bf T}({\mathcal R})|_{s + 1} \lesssim \nu^{- 1} |{\mathcal R}|_s.
 \end{equation} 
\end{lemma}
\begin{proof}
We need to look for an operator $\Psi$ that satisfies for any $j , j' \in {\mathbb N}$, $j \neq j'$
$$
\Pi_j \Big( [\nu\partial_{yy}\,,\, \Psi] + {\mathcal R} \Big) \Pi_{j'} = 0
$$
namely 
$$
\nu (j'^2 - j^2) \Pi_j \Psi \Pi_{j'} + \Pi_j {\mathcal R} \Pi_{j'} = 0, \quad \forall j, j' \in {\mathbb N}, \quad j \neq j'\,. 
$$
We then define the unique solution as
$$
\begin{aligned}
& \Psi \coloneqq  {\bf T}({\mathcal R}) \coloneqq   \sum_{\begin{subarray}{c}
j, j' \in {\mathbb N} \\
j \neq j'
\end{subarray}}  \frac{1}{\nu (j^2 - j'^2)} \Pi_j {\mathcal R} \Pi_{j'}\,.
\end{aligned}
$$
Note that $|j^2 - j'^2| \gtrsim |j - j'|$ for any $j \neq j'$.
Thus, for every $j'\in {\mathbb N}$, 
$$
\sum_{j \in {\mathbb N}} \langle j - j' \rangle^{s + 1} \| \Pi_j {\bf T}({\mathcal R}) \Pi_{j'} \|_{2\times 2}^2 \lesssim \nu^{- 2} \sum_{j \in {\mathbb N}} \langle j - j' \rangle^{2 s} \| \Pi_j {\mathcal R} \Pi_{j'} \|_{2\times 2}^2 \lesssim \nu^{- 2} |{\mathcal R}|_s^2.
$$
By taking the supremum with respect to $j' \in \NN$, we obtain
$
|{\bf T}({\mathcal R})|_{s + 1}^2 \lesssim \nu^{- 2} |{\mathcal R}|_s^2
$.
\end{proof}

\subsubsection*{\textbf{Bound on the remainder}}
We are now ready to give the estimate of the remainder ${\mathcal Q}$ in \cref{forma cal L nu varepsilon per normal-form}, whose bound is collected in the following.
\begin{lemma}\label{lemma-decay-remainder-Pi-neq}
Let $s > \frac12$, $U \in H^{s + 2}(\TT)$. Then there exists $\delta_0 \in (0, 1)$, depending only on $s$ and $\|U\|_{s+2}$, such that if $\varepsilon \nu^{- 1} \leq \delta_0$, then  ${\mathcal Q} \in {\mathcal M}^s$ and 
$|{\mathcal Q}|_s \lesssim \varepsilon$.
\end{lemma}
\begin{proof}
By combining Lemma \ref{stima resto cal L nu (1) cal Q} with the estimates \cref{stima X Y per block decoupling,stima A B per block decoupling}, and taking $\varepsilon \nu^{- 1}\leq \delta_0 $ with $\delta_0$ small enough, one immediately obtains that 
\begin{align}
 \Big| \frac{\nu}{1 - \langle X, Y \rangle}  Y''(y) X(y)^\top \Big|_s & \lesssim \nu \| Y \|_{s + 2} \| X \|_s \lesssim \nu (\varepsilon^{- 1} \nu^{- 1}) (\varepsilon^3 \nu^{- 1}) \lesssim \varepsilon\,, \\
  \, \Big| \frac{1}{1 - \langle X, Y \rangle}  Y(y) X(y)^\top \Big|_s & \lesssim \| Y \|_s \| X \|_s \lesssim \varepsilon^2 \nu^{- 2}\lesssim 1\,, \\
 |A X^\top|_s & \lesssim \| A \|_s \| X \|_s \lesssim \varepsilon^{- 1} (\varepsilon^3 \nu^{- 1})  \lesssim \varepsilon \,, \\
 |Y B^\top|_s & \lesssim \| Y \|_s \| B \|_s \lesssim \varepsilon^{- 1} \nu^{- 1} \varepsilon^3 \lesssim \varepsilon\,. 
\end{align}
Moreover combining the trivial fact that $|(- \partial_{yy} + \varepsilon^2)^{- 1}|_s \lesssim 1$ with Lemma \ref{lemma-algebra-decay}-\cref{item:algebra} and Lemma \ref{stima resto cal L nu (1) cal Q}-\cref{item:multiplication}, one also gets that 
$$
|{\mathcal R}_\varepsilon|_s \lesssim \varepsilon \| U \|_{s + 2} \lesssim \varepsilon\,. 
$$
All the previous estimates, together with  Lemma \ref{lemma-algebra-decay}-\cref{item:algebra}  imply the claimed bound on ${\mathcal Q}$ for some $\delta_0$ sufficiently small.
\end{proof}
\subsubsection*{\textbf{The block diagonalization step}}
The main result of this section is the definition of an off-diagonal operator $\Psi$ that block-diagonalize the operator $\cL_{\nu,\eps}^{(1)}$. More precisely, we have the following.

\begin{proposition}\label{prop-block-diag-contraction}
    Let $s > \frac12$, $U \in H^{s + 2}(\TT)$. Then there exists $\delta_0 \in (0, 1)$ small enough, depending only on $s, \|U\|_{s+2}$, such that if $\e \nu^{- 1} \leq \delta_0$ then there exists unique  $\Psi \in \mathcal{M}_{od}^s$, $\cZ\in\mathcal{M}_{bd}$ such that
    \begin{equation}
       \nu |\Psi|_s +|\cZ|_s\lesssim \eps 
    \end{equation} 
    for which the following block-diagonalization holds true
    \begin{equation}
        (\uno +  \Psi)^{-1} \cL_{\nu,\e}^{(1)} (\uno + \Psi) = \nu \de_{yy} +  \cZ \coloneqq  {\mathcal N}_{\nu, \e} \, .
    \end{equation}
\end{proposition}
\begin{proof}
Recall that ${\mathcal L}_{\nu, \varepsilon}^{(1)} = \nu \partial_{yy} +  {\mathcal Q}$. We want to solve the equation 
$$
(\nu \partial_{yy} + {\mathcal Q}) ({\rm Id} + \Psi) =  ({\rm Id} + \Psi)(\nu \partial_{yy} + {\mathcal Z})
$$
which is equivalent to the equation
$$
[\nu \partial_{yy}\,,\, \Psi] = {\mathcal Z}  -  {\mathcal Q} -  {\mathcal Q} \Psi + \Psi {\mathcal Z} 
$$ We look for solutions $\Psi \in {\mathcal M}^s_{od}$ and ${\mathcal Z} \in {\mathcal M}_{bd}$. If we apply the projection $\Pi_{bd}, \Pi_{od}$ to the latter equation, using $\Pi_{bd}\big( [\nu \partial_{yy}\,,\, \Psi] \big) = 0$ and the fact that by Lemma \ref{lemma-comp-diag-block-diag}, one has $\Pi_{bd}\big( \Psi {\mathcal Z}\big) = 0$, $\Pi_{od}(\Psi {\mathcal Z}) = \Psi {\mathcal Z}$, we get 
\begin{equation}\label{equazione-Psi-cal-Z-main}
\begin{cases}
 [\nu \partial_{yy}\,,\, \Psi] =   - \Pi_{od}({\mathcal Q}) -  \Pi_{od}({\mathcal Q} \Psi) + \Psi {\mathcal Z}  \\
 {\mathcal Z}  -  \Pi_{bd}({\mathcal Q}) -  \Pi_{bd}({\mathcal Q} \Psi)  = 0\,. 
\end{cases}
\end{equation}
From the second equation we recover ${\mathcal Z} \equiv {\mathcal Z}(\Psi)$ as a function of $\Psi$, namely
$$
{\mathcal Z}(\Psi) \coloneqq     \Pi_{bd}({\mathcal Q}) +  \Pi_{bd}({\mathcal Q} \Psi)
$$
By applying Lemmas \ref{lemma-decay-remainder-Pi-neq} and \ref{lemma-algebra-decay}, one easily shows that 
\begin{equation}\label{stime-cal-Z-Psi}
\begin{aligned}
& |{\mathcal Z}(\Psi)|_s \lesssim \varepsilon (1 + |\Psi|_s)\,, \\
& |{\mathcal Z}(\Psi_1) - {\mathcal Z}(\Psi_2)|_s \lesssim \varepsilon |\Psi_1 - \Psi_2|_s\,. 
\end{aligned}
\end{equation}
Therefore, by substituting ${\mathcal Z}(\Psi)$ in the first equation in \cref{equazione-Psi-cal-Z-main}, one obtains 
\begin{equation}\label{equazione-ridotta-Psi}
[\nu \partial_{yy}\,,\, \Psi] = {\mathcal F}(\Psi)\,, \quad {\mathcal F}(\Psi) \coloneqq    -  \Pi_{od}({\mathcal Q}) -  \Pi_{od}({\mathcal Q} \Psi) + (\Psi {\mathcal Z}(\Psi))\,.
\end{equation}
By applying  Lemmas \ref{lemma-decay-remainder-Pi-neq} and \ref{lemma-algebra-decay} and using the estimates \cref{stime-cal-Z-Psi}, one deduces that the map 
$$
{\mathcal M}^s_{od} \to {\mathcal M}^s_{od}, \quad \Psi \mapsto {\mathcal F}(\Psi)
$$
satisfies the estimates 
\begin{equation}\label{stime-cal-F-Psi}
\begin{aligned}
& |{\mathcal F}(\Psi)|_s \lesssim \varepsilon (1 + |\Psi|_s + |\Psi|_s^2)\,, \quad \forall \Psi \in {\mathcal M}^s_{od} \\
& |{\mathcal F}(\Psi_1 ) - {\mathcal F}(\Psi_2)|_s \lesssim \varepsilon(1 + |\Psi_1|_s + |\Psi_2|_s) |\Psi_1 - \Psi_2|_s, \quad \forall \Psi_1, \Psi_2 \in {\mathcal M}^s_{od}\,. 
\end{aligned}
\end{equation}
By using Lemma \ref{eq-omologica-Pneq}, where the operator ${\bf T}$ is introduced, the equation \cref{equazione-ridotta-Psi} is equivalent to the fixed-point equation 
\begin{equation}\label{eq-punto-fisso-Psi}
\Psi = {\bf \Phi}(\Psi)\,, \quad {\bf \Phi}(\Psi) \coloneqq   {\bf T}({\mathcal F}(\Psi))\,. 
\end{equation}
We define
$$
{\mathcal B}_s(\rho) \coloneqq  \Big\{ \Psi \in {\mathcal M}^s_{od} : |\Psi|_s \leq \rho \Big\}
$$
and we claim the following. 
\begin{quote}
{\bf Claim.} There exist constants $C_* \equiv C_* (s) > 0$, $\delta_0 \equiv \delta_0(s) > 0$ such that if $\varepsilon \nu^{- 1} \leq \delta_0$, then the map 
$
{\bf \Phi} : {\mathcal B}_s(C_* \varepsilon \nu^{- 1}) \to {\mathcal B}_s(C_* \varepsilon \nu^{- 1})
$
is a contraction.    
\end{quote}
To prove this claim, let $\Psi \in {\mathcal B}_s(C_* \varepsilon \nu^{- 1})$. Then by Lemma \ref{eq-omologica-Pneq} and estimate \cref{stime-cal-F-Psi}, one gets 
$$
|{\bf \Phi}(\Psi)|_s \leq C(s) \varepsilon \nu^{- 1} (1 + C_* \varepsilon \nu^{- 1} + [C_* \varepsilon \nu^{- 1}]^2) \leq C_* \varepsilon \nu^{- 1} ,
$$
where we are taking $C_* \equiv C_* (s) \gg 1$ large enough and $\varepsilon \nu^{- 1} \ll 1$ small enough. 
Moreover by taking $\Psi_1, \Psi_2 \in {\mathcal B}_s(C_* \varepsilon \nu^{- 1})$, one gets that 
$$
|{\bf \Phi}(\Psi_1) - {\bf \Phi}(\Psi_2)|_s \leq  C(s) \varepsilon \nu^{- 1}  (1 + 2 C_* \varepsilon \nu^{- 1} )|\Psi_1 - \Psi_2|_s \leq \frac12 |\Psi_1 - \Psi_2|_s
$$
by taking $\varepsilon \nu^{- 1} \ll 1$ small enough. We have then proved the desired claim.

Thus, by the contraction mapping theorem we find a unique solution $\Psi \in {\mathcal B}_s(C_* \varepsilon \nu^{- 1})$ of the fixed point equation $\Psi = {\bf \Phi}(\Psi)$. Note that by using again  Lemma \ref{eq-omologica-Pneq} (together with \cref{stime-cal-F-Psi}), we also have that $\Psi \in {\mathcal M}^{s + 1}_{od}$ and $|\Psi|_{s + 1} \lesssim \varepsilon \nu^{- 1}$. By recalling \cref{stime-cal-Z-Psi}, we get $|{\mathcal Z}|_s \lesssim \varepsilon$. Finally, the invertibility of ${\rm Id} + \Psi$ follows by the Neumann series argument in Lemma \ref{lemma-algebra-decay}-\cref{item:Id+T} (clearly by taking $\varepsilon \nu^{- 1} \ll 1$ small enough). The proof of the Proposition is then concluded.
\end{proof}

With Proposition \ref{prop-block-diag-contraction} at hand, we are finally ready to conclude to proof of Theorem \ref{teorema coniugio cal Lk}.
\begin{proof}[Proof of Theorem \ref{teorema coniugio cal Lk}]
    We define $\Phi_{\nu, \varepsilon}: H^s \to H^s$ as the mapping associated, in the sense of \cref{matrixrepresentation}, to the matrix product
\begin{equation}\label{Phinueps}
 \Phi_{\nu, \varepsilon}\coloneqq \mathtt T^{- 1} \circ \begin{bmatrix}
1 & 0^\top \\
0 & {\rm Id} + \Psi
\end{bmatrix} \, .
\end{equation}
Then Theorem \ref{teorema coniugio cal Lk}, except for the expansion of $V_{\nu,\eps}^{(0)}$ in \cref{exp:NFfinal}, follows by combining  Lemma \ref{primo coniugio normal-form} and  Proposition \ref{prop-block-diag-contraction}. 

It remains to prove the asymptotic expansion of the eigenfunction $V_{\nu,\e}^{(0)}$ of $\cL_{\nu,\e}$ associated with its unstable eigenvalue $\lambda_{\nu,\e}^{(0)}$. An unstable eigenvector is given by $\Phi_{\nu,\e}(1)$, which from \cref{Phinueps} is associated with
\begin{equation}
      \ttT^{-1} \begin{bmatrix}
        1 \\ 0
    \end{bmatrix} \stackrel{\cref{partialTinverse}}{=} \frac{1}{1-\langle Y_{\nu,\e}, X_{\nu,\e}\rangle }\begin{bmatrix}
        1 \\ -Y_{\nu,\e}(y)
    \end{bmatrix}\, .
\end{equation}
Note that $Y_{\nu,\eps}\in H^{s+2}_0$ by Lemma \ref{lem:vheq},
Hence, combining \cref{YperX,actualY} and adapting  the notation $\cO_{\rm fun}$ in \cref{sec:unstable} to measure errors in $H^{s+2}$, we obtain
\begin{equation}
   \ttT^{-1} \begin{bmatrix}
        1 \\ 0
    \end{bmatrix} = \begin{bmatrix}
    1 \\ \cA_{\nu,\e}^{-1}\Big(\frac{\im}{\e}U''(y) +\im \e U(y) \Big) + \cO_{{\rm fun}}\big(\frac{\e}{\nu^3}\big)
\end{bmatrix} \Big(1 + \cO\big(\frac{\e^2}{\nu^2}\big) \Big)\,.
\end{equation}
Appealing to \cref{eq:A-1Usecond} and, the expression above becomes
\begin{equation}
\Phi_{\nu,\e}[1] = \begin{bmatrix}
    1 \\ \frac{\im }{\nu\e} U(y)  + \cO_{{\rm fun}}\big(\frac{\e}{\nu}+\frac{\e}{\nu^3}\big)
\end{bmatrix} \Big(1 + \cO\big(\frac{\e^2}{\nu^2}\big) \Big)\,.
\end{equation}
We conclude that the eigenvector $V_{\nu,\e}^{(0)}$ in Theorem \ref{teorema coniugio cal Lk} is associated with
\begin{equation}
  -\im \nu \e \Phi_{\nu,\e}[1]  = \begin{bmatrix}
    -\im \nu \e \\  U(y)  + \cO_{{\rm fun}}\big(\e^2 +\frac{\e^2}{\nu^2}\big)
\end{bmatrix} \Big(1 + \cO\big(\frac{\e^2}{\nu^2}\big) \Big)\,, 
\end{equation}
namely
\begin{equation}
    V_{\nu,\e}^{(0)} =U(y) -\im \nu \e   + \cO_{{\rm fun}}\big(\e^2 +\frac{\e^2}{\nu^2}\big) \, .
\end{equation}
Since $U(y)$ and the error are in $H^{s+2}_0$, the expansion above proves  \cref{exp:NFfinal}.
\end{proof}

\appendix
\section{Rank-$1$ update of a closed operator}\label{rank1proof}

In this section we prove Lemma \ref{lem:ShermanMorrisoninf}. We first observe that, if $
\langle g,\cA^{-1} f\rangle_H+1= 0$, then the vector $\cA^{-1} f \in H $ is such that
\begin{equation}
     (\cA+f\langle g,\cdot\rangle_H)  \cA^{-1} f = f + \langle g,\cA^{-1} f\rangle_H f = 0 \, ,
\end{equation}
and the operator $\cA+f\langle g,\cdot\rangle_H$ is not invertible.  On the other hand, if $
\langle g,\cA^{-1} f\rangle_H+1\neq 0$, then the operator 
\begin{equation}\label{theinverse}
    \cA^{-1}-\frac{\cA^{-1}(f\langle g,\cdot\rangle_H)\cA^{-1}}{1+\langle g,\cA^{-1} f\rangle_H}:H\to H
\end{equation}
is well-defined and bounded. We have
\begin{subequations}\label{id+robe}
\begin{align}\label{id+robe1}
  &\big(\cA+f\langle g,\cdot\rangle_H\big)  \Big(\cA^{-1}-\frac{\cA^{-1}(f\langle g,\cdot\rangle_H)\cA^{-1}}{1+\langle g,\cA^{-1} f\rangle_H}\Big) \\
  = &\uno -\frac{(f\langle g,\cdot\rangle_H)\cA^{-1}}{1+\langle g,\cA^{-1} f\rangle_H} + (f\langle g,\cdot\rangle_H)\cA^{-1} - \frac{(f\langle g,\cdot\rangle_H)\cA^{-1}(f\langle g,\cdot\rangle_H)\cA^{-1}}{1+\langle g,\cA^{-1} f\rangle_H}\, ,
\end{align}
and
\begin{align}\label{id+robe2}
  & \Big(\cA^{-1}-\frac{\cA^{-1}(f\langle g,\cdot\rangle_H)\cA^{-1}}{1+\langle g,\cA^{-1} f\rangle_H}\Big) \big(\cA+f\langle g,\cdot\rangle_H\big)  \\ 
  = &\uno+ \cA^{-1}(f\langle g,\cdot\rangle_H) -\frac{\cA^{-1}(f\langle g,\cdot\rangle_H)}{1+\langle g,\cA^{-1} f\rangle_H}  - \frac{\cA^{-1}(f\langle g,\cdot\rangle_H)\cA^{-1}(f\langle g,\cdot\rangle_H)}{1+\langle g,\cA^{-1} f\rangle_H}\, .
\end{align}
\end{subequations}
We observe that, by rearranging parentheses,
\begin{subequations}\label{propassociativa}
\begin{equation}\label{propassociativa1}
    (f\langle g,\cdot\rangle_H)\cA^{-1}(f\langle g,\cdot\rangle_H)\cA^{-1} =  f (\langle g,\cdot\rangle_H\cA^{-1}f)\langle g,\cdot\rangle_H \cA^{-1} =  (\langle g,\cA^{-1} f\rangle_H ) f\langle g,\cdot\rangle_H \cA^{-1}\, ,
\end{equation}
and
\begin{equation}\label{propassociativa2}
    \cA^{-1}(f\langle g,\cdot\rangle_H)\cA^{-1}(f\langle g,\cdot\rangle_H) = \cA^{-1} f (\langle g,\cdot\rangle_H\cA^{-1}f)\langle g,\cdot\rangle_H  =  (\langle g,\cA^{-1} f\rangle_H ) \cA^{-1}f\langle g,\cdot\rangle_H \, .
\end{equation}
\end{subequations}
By applying \cref{propassociativa1,propassociativa2} in \cref{id+robe1,id+robe2} respectively, we obtain
\begin{subequations}\label{identitaH}
\begin{align}\label{identitaH1}
  &\big(\cA+f\langle g,\cdot\rangle_H\big)  \Big(\cA^{-1}-\frac{\cA^{-1}(f\langle g,\cdot\rangle_H)\cA^{-1}}{1+\langle g,\cA^{-1} f\rangle_H}\Big) \\
  &= \uno -(1+\langle g,\cA^{-1} f\rangle_H)\frac{(f\langle g,\cdot\rangle_H)\cA^{-1}}{1+\langle g,\cA^{-1} f\rangle_H} + (f\langle g,\cdot\rangle_H)\cA^{-1} = \uno\, ,
\end{align}
and
\begin{align}\label{identitaH2}
  & \Big(\cA^{-1}-\frac{\cA^{-1}(f\langle g,\cdot\rangle_H)\cA^{-1}}{1+\langle g,\cA^{-1} f\rangle_H}\Big) \big(\cA+f\langle g,\cdot\rangle_H\big)  \\ 
  = &\uno+ \cA^{-1}(f\langle g,\cdot\rangle_H) -(1+\langle g,\cA^{-1} f\rangle_H) \frac{\cA^{-1}(f\langle g,\cdot\rangle_H)}{1+\langle g,\cA^{-1} f\rangle_H}  = \uno\, .
\end{align}
\end{subequations}
Identities \cref{identitaH1,identitaH2} show that the operator in \cref{theinverse} is the inverse of $\cA+f\langle g,\cdot\rangle_H$. \qed

\bibliographystyle{siam}

\end{document}

%% file: macros.tex
\def\dd{{\rm d}}

\def\im{\mathrm{i}\,}

\def\eps{\varepsilon}
\def\e{\varepsilon}
\def\epsilon{\varepsilon}
\def\de{{\partial}}

\def\drm{{\mathrm{d}}}

\def\uno{{\mathrm{Id}}}

\def\RR {\mathbb{R}}
\def\CC {\mathbb{C}}
\def\TT {\mathbb{T}}
\def\ZZ {\mathbb{Z}}
\def\NN {\mathbb{N}}

\def\Re{{\rm Re}}
\def\Im{{\rm Im}}

\newcommand{\cA}{\mathcal{A}}
\newcommand{\cB}{\mathcal{B}}

\newcommand{\cD}{\mathcal{D}}

\newcommand{\cF}{\mathcal{F}}

\newcommand{\cI}{\mathcal{I}}

\newcommand{\cL}{\mathcal{L}}
\newcommand{\cM}{\mathcal{M}}
\newcommand{\cN}{\mathcal{N}}
\newcommand{\cO}{\mathcal{O}}

\newcommand{\cQ}{\mathcal{Q}}
\newcommand{\cR}{\mathcal{R}}

\newcommand{\cT}{\mathcal{T}}
\newcommand{\cU}{\mathcal{U}}
\newcommand{\cV}{\mathcal{V}}

\newcommand{\cZ}{\mathcal{Z}}

\newcommand{\ttA}{\mathtt{A}}

\newcommand{\ttL}{\mathtt{L}}

\newcommand{\ttT}{\mathtt{T}}

\newcommand{\ttX}{\mathtt{X}}

\newtheorem{proposition}{Proposition}[section]
\newtheorem{theorem}[proposition]{Theorem}

\newtheorem{lemma}[proposition]{Lemma}
\theoremstyle{definition}

\newtheorem{remark}[proposition]{Remark}
\newtheorem*{remark*}{Remark}

\theoremstyle{definition}

\numberwithin{equation}{section}

%% file: Long_Wave_NS_arXiv.bbl
\begin{thebibliography}{10}

\bibitem{Arbon25CMP}
{\sc R.~Arbon and J.~Bedrossian}, {\em Quantitative hydrodynamic stability for {C}ouette flow on unbounded domains with {N}avier boundary conditions}, Comm. Math. Phys., 406 (2025), pp.~Paper No. 129, 57.

\bibitem{arnol1960kolmogorov}
{\sc V.~I. Arnol'd and L.~D. Meshalkin}, {\em {A.N. Kolmogorov's seminar on selected problems of analysis (1958/1959)}}, Uspekhi Matematicheskikh Nauk, 15 (1960), pp.~247--250.

\bibitem{BaldiMontalto}
{\sc P.~Baldi and R.~Montalto}, {\em Quasi-periodic incompressible {E}uler flows in 3{D}}, Adv. Math., 384 (2021), pp.~Paper No. 107730, 74.

\bibitem{Beck20ARMA}
{\sc M.~Beck, O.~Chaudhary, and C.~E. Wayne}, {\em Rigorous justification of {T}aylor dispersion via center manifolds and hypocoercivity}, Arch. Ration. Mech. Anal., 235 (2020), pp.~1105--1149.

\bibitem{BCZD24Taylor}
{\sc J.~Bedrossian, M.~Coti~Zelati, and M.~Dolce}, {\em Taylor dispersion and phase mixing in the non-cutoff {B}oltzmann equation on the whole space}, Proc. Lond. Math. Soc. (3), 129 (2024), pp.~Paper No. e12616, 70.

\bibitem{Bedrossian19Ann}
{\sc J.~Bedrossian, M.~Coti~Zelati, and V.~Vicol}, {\em Vortex axisymmetrization, inviscid damping, and vorticity depletion in the linearized 2{D} {E}uler equations}, Ann. PDE, 5 (2019), pp.~Paper No. 4, 192.

\bibitem{Bedrossian15PIHES}
{\sc J.~Bedrossian and N.~Masmoudi}, {\em Inviscid damping and the asymptotic stability of planar shear flows in the 2{D} {E}uler equations}, Publ. Math. Inst. Hautes \'Etudes Sci., 122 (2015), pp.~195--300.

\bibitem{beekie2024uniform}
{\sc R.~Beekie, S.~Chen, and H.~Jia}, {\em Uniform vorticity depletion and inviscid damping for periodic shear flows in the high {Reynolds} number regime}, arXiv:2403.13104,  (2024).

\bibitem{BCMV}
{\sc M.~Berti, L.~Corsi, A.~Maspero, and P.~Ventura}, {\em Infinitely many isolas of modulational instability for {S}tokes waves}, arXiv:2405.05854,  (2024).

\bibitem{vortex3}
{\sc M.~Berti, Z.~Hassainia, and N.~Masmoudi}, {\em Time quasi-periodic vortex patches of {E}uler equation in the plane}, Invent. Math., 233 (2023), pp.~1279--1391.

\bibitem{BMV}
{\sc M.~Berti, A.~Maspero, and P.~Ventura}, {\em Full description of {B}enjamin-{F}eir instability of {S}tokes waves in deep water}, Invent. math., 230 (2022), pp.~651--711.

\bibitem{BMV1}
\leavevmode\vrule height 2pt depth -1.6pt width 23pt, {\em Benjamin–feir instability of stokes waves in finite depth}, Archive for Rational Mechanics and Analysis, 247(5), 91 (2023).

\bibitem{BMV2}
\leavevmode\vrule height 2pt depth -1.6pt width 23pt, {\em Stokes waves at the critical depth are modulationally unstable}, Communications in Mathematical Physics, 405(3), 56 (2024).

\bibitem{bian2023asymptotic}
{\sc D.~Bian and E.~Grenier}, {\em {Asymptotic behaviour of solutions of linearized Navier Stokes equations in the long waves regime}}, arXiv preprint arXiv:2312.16938,  (2023).

\bibitem{BFMT}
{\sc R.~Bianchini, L.~Franzoi, R.~Montalto, and S.~Terracina}, {\em Large amplitude quasi-periodic traveling waves in two dimensional forced rotating fluids}, Comm. Math. Phys., 406 (2025), pp.~Paper No. 66, 67.

\bibitem{bianchini2025instabilities}
{\sc R.~Bianchini, A.~Maspero, and S.~Pasquali}, {\em Instabilities of internal gravity waves in the two-dimensional {Boussinesq} system}, arXiv preprint arXiv:2507.10390,  (2025).

\bibitem{boffetta2012two}
{\sc G.~Boffetta and R.~E. Ecke}, {\em Two-dimensional turbulence}, Annual review of fluid mechanics, 44 (2012), pp.~427--451.

\bibitem{CotiZelati23Stationary}
{\sc M.~Coti~Zelati, T.~M. Elgindi, and K.~Widmayer}, {\em Stationary structures near the {K}olmogorov and {P}oiseuille flows in the {$2d$} {E}uler equations}, Arch. Ration. Mech. Anal., 247 (2023), pp.~Paper No. 12, 37.

\bibitem{coti2023enhanced}
{\sc M.~Coti~Zelati and T.~Gallay}, {\em Enhanced dissipation and {T}aylor dispersion in higher-dimensional parallel shear flows}, J. Lond. Math. Soc. (2), 108 (2023), pp.~1358--1392.

\bibitem{DENG20111561}
{\sc C.~Y. Deng}, {\em A generalization of the {S}herman-{M}orrison-{W}oodbury formula}, Appl. Math. Lett., 24 (2011), pp.~1561--1564.

\bibitem{drazin2004hydrodynamic}
{\sc P.~G. Drazin and W.~H. Reid}, {\em Hydrodynamic stability}, Cambridge Monographs on Mechanics and Applied Mathematics, Cambridge University Press, Cambridge-New York, 1981.

\bibitem{FMM}
{\sc L.~Franzoi, N.~Masmoudi, and R.~Montalto}, {\em Space quasi-periodic steady {E}uler flows close to the inviscid {C}ouette flow}, Arch. Ration. Mech. Anal., 248 (2024), pp.~Paper No. 81, 79.

\bibitem{FMVV}
{\sc L.~Franzoi and R.~Montalto}, {\em A {KAM} approach to the inviscid limit for the 2{D} {N}avier-{S}tokes equations}, Ann. Henri Poincar\'e, 25 (2024), pp.~5231--5275.

\bibitem{Friedlander06CMP}
{\sc S.~Friedlander, N.~s. Pavlovi\'c, and R.~Shvydkoy}, {\em Nonlinear instability for the {N}avier-{S}tokes equations}, Comm. Math. Phys., 264 (2006), pp.~335--347.

\bibitem{friedlander1997nonlinear}
{\sc S.~Friedlander, W.~Strauss, and M.~Vishik}, {\em Nonlinear instability in an ideal fluid}, Ann. Inst. H. Poincar\'e{} C Anal. Non Lin\'eaire, 14 (1997), pp.~187--209.

\bibitem{Gervais24Hydrodynamic}
{\sc P.~Gervais and B.~Lods}, {\em Hydrodynamic limits for kinetic equations preserving mass, momentum and energy: a spectral and unified approach in the presence of a spectral gap}, Ann. H. Lebesgue, 7 (2024), pp.~969--1098.

\bibitem{vortex1}
{\sc J.~G{\'o}mez-Serrano, A.~D. Ionescu, and J.~Park}, {\em {Quasiperiodic solutions of the generalized SQG equation}}, To appear in Annals of Math. Studies. preprint arxiv:2303.03992,  (2023).

\bibitem{green1974two}
{\sc J.~Green}, {\em Two-dimensional turbulence near the viscous limit}, Journal of Fluid Mechanics, 62 (1974), pp.~273--287.

\bibitem{Grenier00CPAM}
{\sc E.~Grenier}, {\em On the nonlinear instability of {E}uler and {P}randtl equations}, Comm. Pure Appl. Math., 53 (2000), pp.~1067--1091.

\bibitem{Grenier16}
{\sc E.~Grenier, Y.~Guo, and T.~T. Nguyen}, {\em Spectral instability of characteristic boundary layer flows}, Duke Math. J., 165 (2016), pp.~3085--3146.

\bibitem{Grenier16Adv}
\leavevmode\vrule height 2pt depth -1.6pt width 23pt, {\em Spectral instability of general symmetric shear flows in a two-dimensional channel}, Adv. Math., 292 (2016), pp.~52--110.

\bibitem{vortex2}
{\sc Z.~Hassainia, T.~Hmidi, and N.~Masmoudi}, {\em {KAM theory for active scalar equations}}, to appear in Memoirs of the American Mathematical Society,  (2021).

\bibitem{vortex4}
{\sc Z.~Hassainia, T.~Hmidi, and E.~Roulley}, {\em Invariant {KAM} tori around annular vortex patches for 2{D} {E}uler equations}, Comm. Math. Phys., 405 (2024), pp.~Paper No. 270, 127.

\bibitem{howard2020stability}
{\sc M.~P. Howard, A.~Statt, H.~A. Stone, and T.~M. Truskett}, {\em Stability of force-driven shear flows in nonequilibrium molecular simulations with periodic boundaries}, The Journal of chemical physics, 152 (2020).

\bibitem{Ibrahim19Pseudo}
{\sc S.~Ibrahim, Y.~Maekawa, and N.~Masmoudi}, {\em On pseudospectral bound for non-selfadjoint operators and its application to stability of {K}olmogorov flows}, Ann. PDE, 5 (2019), pp.~Paper No. 14, 84.

\bibitem{Ionescu23Acta}
{\sc A.~D. Ionescu and H.~Jia}, {\em Non-linear inviscid damping near monotonic shear flows}, Acta Math., 230 (2023), pp.~321--399.

\bibitem{Kappeler1991}
{\sc T.~Kappeler}, {\em Fibration of the phase space for the {K}orteweg-de {V}ries equation}, Ann. Inst. Fourier (Grenoble), 41 (1991), pp.~539--575.

\bibitem{Kato}
{\sc T.~Kato}, {\em Perturbation Theory for Linear Operators}, Springer-Verlag, 1966.

\bibitem{Latushkin18JMFM}
{\sc Y.~Latushkin and S.~Vasudevan}, {\em Eigenvalues of the linearized 2{D} {E}uler equations via {B}irman-{S}chwinger and {L}in's operators}, J. Math. Fluid Mech., 20 (2018), pp.~1667--1680.

\bibitem{liao2023stability}
{\sc S.~Liao, Z.~Lin, and H.~Zhu}, {\em {On the stability and instability of Kelvin-Stuart cat's eyes flows}}, arXiv preprint arXiv:2304.00264,  (2023).

\bibitem{lin2003instability}
{\sc Z.~Lin}, {\em Instability of some ideal plane flows}, SIAM J. Math. Anal., 35 (2003), pp.~318--356.

\bibitem{Lin04IMRN}
\leavevmode\vrule height 2pt depth -1.6pt width 23pt, {\em Nonlinear instability of ideal plane flows}, Int. Math. Res. Not.,  (2004), pp.~2147--2178.

\bibitem{Lin04CMP}
\leavevmode\vrule height 2pt depth -1.6pt width 23pt, {\em Some stability and instability criteria for ideal plane flows}, Comm. Math. Phys., 246 (2004), pp.~87--112.

\bibitem{Lin19Metastability}
{\sc Z.~Lin and M.~Xu}, {\em Metastability of {K}olmogorov flows and inviscid damping of shear flows}, Arch. Ration. Mech. Anal., 231 (2019), pp.~1811--1852.

\bibitem{Lin22Zeng}
{\sc Z.~Lin and C.~Zeng}, {\em Instability, index theorem, and exponential trichotomy for linear {H}amiltonian {PDE}s}, Mem. Amer. Math. Soc., 275 (2022), pp.~v+136.

\bibitem{Masmoudi24Annals}
{\sc N.~Masmoudi and W.~Zhao}, {\em Nonlinear inviscid damping for a class of monotone shear flows in a finite channel}, Ann. of Math. (2), 199 (2024), pp.~1093--1175.

\bibitem{Mehsalkin61Investigation}
{\sc L.~D. Meshalkin and J.~G. Sinai}, {\em Investigation of the stability of a stationary solution of a system of equations for the plane movement of an incompressible viscous liquid}, J. Appl. Math. Mech., 25 (1961), pp.~1700--1705.

\bibitem{KIRCHOFF}
{\sc R.~Montalto}, {\em Quasi-periodic solutions of forced {K}irchhoff equation}, NoDEA Nonlinear Differential Equations Appl., 24 (2017), pp.~Paper No. 9, 71.

\bibitem{okamoto1993bifurcation}
{\sc H.~Okamoto and M.~Sh\={o}ji}, {\em Bifurcation diagrams in {K}olmogorov's problem of viscous incompressible fluid on {$2$}-{D} flat tori}, Japan J. Indust. Appl. Math., 10 (1993), pp.~191--218.

\bibitem{revina2017stability}
{\sc S.~V. Revina}, {\em Stability of the {K}olmogorov flow and its modifications}, Comput. Math. Math. Phys., 57 (2017), pp.~995--1012.

\bibitem{RieszNagy}
{\sc F.~Riesz and B.~S. Nagy}, {\em Functional Analysis}, Blackie \& Son Limited, 1956.

\bibitem{taylor1953dispersion}
{\sc G.~I. Taylor}, {\em Dispersion of soluble matter in solvent flowing slowly through a tube}, Proc. R. Soc. Lond., 219 (1953), pp.~186--203.

\bibitem{wei2019enhanced}
{\sc D.~Wei and Z.~Zhang}, {\em Enhanced dissipation for the {K}olmogorov flow via the hypocoercivity method}, Sci. China Math., 62 (2019), pp.~1219--1232.

\bibitem{Wei20Linear}
{\sc D.~Wei, Z.~Zhang, and W.~Zhao}, {\em Linear inviscid damping and enhanced dissipation for the {K}olmogorov flow}, Adv. Math., 362 (2020), pp.~106963, 103.

\bibitem{Yudovich66}
{\sc V.~I. Yudovich}, {\em Instability of parallel flows of a viscous incompressible fluid with respect to perturbations periodic in space}, \v Z. Vy\v cisl. Mat i Mat. Fiz., 6 (1966), pp.~242--249.

\bibitem{yudovich1989linearization}
{\sc V.~I. Yudovich}, {\em The linearization method in hydrodynamical stability theory}, vol.~74, American Mathematical Society, 1989.

\end{thebibliography}
